\documentclass[11pt,a4paper,leqno,twoside, headinclude]{amsart}
\usepackage[utf8]{inputenc}
\usepackage{amsmath}
\usepackage{amsfonts}
\usepackage{amssymb}
\usepackage{amsthm}
\usepackage{hyperref}
\usepackage{geometry}
\usepackage{xcolor}
\usepackage[all,cmtip]{xy}
\usepackage{tikz}
\usepackage{enumerate}
\usepackage{stmaryrd}

\newcommand{\U}{\mathrm{U}}

\newcommand{\SO}{\mathrm{SO}}

\newcommand{\rk}{\mathrm{rk}}
\newcommand{\ra}{\rightarrow}

\newcommand{\qq}{{\mathbb{Q}}}                                     

\newtheorem{thm}{Theorem}[section]
\newtheorem{prop}[thm]{Proposition}
\newtheorem{cor}[thm]{Corollary}
\newtheorem{lem}[thm]{Lemma}

\newtheorem{main}{Theorem}
\newtheorem*{main*}{Theorem}

\newtheorem{maincor}[main]{Corollary}

\theoremstyle{definition}
\newtheorem{rem}[thm]{Remark}
\newtheorem{defn}[thm]{Definition}
\newtheorem{ex}[thm]{Example}
\newtheorem{conj2}[thm]{Conjecture}

\newtheorem*{conj*}{Conjecture}

\newcommand{\im} {{\operatorname{im\,}}}
\newcommand{\coker}{{\operatorname{coker}}}
\newcommand{\Id}{{\bf{1}}}

\newcommand{\ack}{\noindent\textbf{Acknowledgements. }}            
\newcommand{\str}{\noindent\textbf{Structure of the article. }}    


\author{Manuel Amann and Leopold Zoller}
\title[The Toral Rank Conjecture and variants of equivariant formality]{
The Toral Rank Conjecture\\ and variants of equivariant formality}

\date{October, 8th, 2019}

\begin{document}

\begin{abstract}
An action of a compact Lie group is called equivariantly formal, if the Leray--Serre spectral sequence of its Borel fibration degenerates at the $E_2$-term. This term is as prominent as it is restrictive.

In this article, also motivated by the lack of junction between the notion of equivariant formality and the concept of formality of spaces (surging from rational homotopy theory) we suggest two new variations of equivariant formality: ``$\mathcal{MOD}$-formal actions'' and ``actions of formal core''.

We investigate and characterize these new terms in many different ways involving various tools from rational homotopy theory, Hirsch--Brown models, $A_\infty$-algebras, etc., and, in particular, we provide different applications ranging from actions on symplectic manifolds and rationally elliptic spaces to manifolds of non-negative sectional curvature.

A major motivation for the new definitions was that an almost free action of a torus $T^n\curvearrowright X$ possessing any of the two new properties satisfies the toral rank conjecture, i.e.~$\dim H^*(X;\qq)\geq 2^n$. This generalizes and proves the toral rank conjecture for actions with formal orbit spaces.
\end{abstract}

\maketitle
\section{Introduction}

What are the restrictions imposed on the topology of a space by the existence of a non-trivial Lie group action?
This is a classical question in topology and geometry which has met several different answers in various respective subfields. A prominent variant of this problem is encoded in the so-called \emph{toral rank conjecture}, which in this form is due to Steve Halperin and which also became known as the Halperin--Carlsson conjecture after suitably extending it to finite coefficients.
\begin{conj*}
Let $M$ be a simply-connected compact manifold equipped with an almost-free action of a torus $T^n$ of rank $n$. Then the sum of all Betti numbers, the total cohomological dimension satisfies
\begin{align*}
\dim H^*(M;\qq)\geq 2^n
\end{align*}
\end{conj*}

\bigskip

In a slighlty more general version, one conjectures the same to hold for almost free actions on compact Hausdorff spaces. While this conjecture has met high interest not only in geometry but also in commutative algebra and has been verified in a lot of special cases ranging from certain nilmanifolds, over compact homogeneous spaces, to K\"ahler manifolds, it is highly open in general with known lower bounds just linear in $n$.

The purpose of this article is two-fold:
\begin{itemize}
\item
On the one hand, we prove the toral rank conjecture in several previously unknown cases, and we provide simple reproofs in some already understood situations.
\item
On the other hand, these confirmations of the conjecture, so to speak, merely form the peak of the iceberg of a substantially more in-depth study of the equivariant cohomology of compact Lie group actions
\end{itemize}
Part of this progress is enabled by recent advances on the Buchsbaum--Eisenbud--Horrocks Conjecture due to Walker which---essentially drawing on two newly shaped concepts of ``formality'' in an equivariant context---we can link to the toral rank conjecture.

So it is not surprising that an integral part of the work carried out in
this article is constituted by molding two new concepts in equivariant cohomology theory, which we label
\begin{align*}
\textrm{\emph{$\mathcal{MOD}$-formality} \qquad and \qquad \emph{formal core}}
\end{align*}
Both are intended to provide variations of the classical term ``equivariant formality" which are able to capture a certain ``cohomological degeneracy property'' (see next paragraph)
as well as a connection to classical formality of spaces in the sense of rational homotopy theory, i.e.\ a weakening of the formality of the Borel construction. 
(Recall that formality describes the fact that the rational homotopy type of a space is already fully encoded by its rational cohomology algebra.)
These refined notions will come in very handy in studying the toral rank conjecture, for example, whereas equivariant formality implies the existence of fixed points.


Indeed, recall that a compact Lie group action $G \curvearrowright M$ is called \emph{equivariantly formal} if, as a module, $H^*(M_G;\qq)\cong H^*(M;\qq) \otimes H^*(B G;\qq)$ (where $M_G$ denotes the Borel construction), i.e.~if equivariant cohomology as a module is free over the classifying space cohomology, or, equivalently, the Borel fibration is totally non-homologous to zero, i.e.~its spectral sequence degenerates at the $E_2$-term. This concept does provide a strong tool in transformation groups and arises in several distinct geometric situations like Hamiltonian torus actions or actions on compact K\"ahler manifolds. Moreover, by Chang--Skjelbred, Atiyah--Bredon this property allows to reconstruct equivariant cohomology from lower dimensional orbit strata. However, (although, for example, there do arise some connections as for isotropy actions on homogeneous spaces)
it is safe to say that ``a priori'' equivariant cohomology is rather unrelated to the concept of formality which we described above.

This lack of congruence was one source of motivation for the definition of our new concepts, and has led to several reformulations and extensions of equivariant formality in the literature.  Notably, in work by Lillywhite he suggests to replace equivariant formality by several ways (in different categories) to encode the formality of the orbit space. Skull imposes several compatibility and formality conditions on the family of orbits extending its formality as a space. Moreover, there is work by Triantafillou in this direction and also the concept of Cohen--Macaulay actions (stating that the equivariant cohomology be a Cohen--Macaulay module over the equivariant cohomology of the classifying space). Since, however, for example, any (almost) free action clearly is Cohen--Macaulay, this concept---very reasonable and interesting in its own right---is too relaxed for our purposes.

It is hence our goal to provide variations of the term which on the one hand side are more flexible than these redefinitions or on the other hand can effectively cope with fixed-point free actions and are better suited to specifying and dealing with certain free actions for example. Hence they indeed interplay nicely with the toral rank conjecture. Furthermore, the two new terms permit many examples. Note for example that in favour of flexibility, we built in some ``locality conditions'', i.e.~the terms can detect a situation in which a group acts say on one factor of a direct product only and the second factor is highly non-formal, yet the question whether the action satisfies our requirements is determined by the first factor only---this is a new feature compared to a ``formality of the Borel construction''-definition.

\vspace{5mm}

Hence as one outcome of this we prove the following theorem in larger generality thus providing new evidence to the toral rank conjecture. Recall again that as usual and as defined above the term ``formality'' is used in the sense of Rational Homotopy Theory.
``Rational ellipticity'' of a nilpotent space $X$ refers to $\dim \pi_*(X)\otimes \qq<\infty$ and $\dim H_*(X;\qq)<\infty$.
\begin{main}\label{theoA}
The toral rank conjecture for $T\curvearrowright X$ and $X$ a compact Hausdorff space holds whenever
\begin{itemize}
\item the Borel construction $X_T$ is formal,
respectively, much more generally and in the newly developed terminology, if the action is either \emph{$\mathcal{MOD}$-formal} or has \emph{formal core}.
\item the space $X$ itself is rationally elliptic and formal---for which we provide a structure result complementing its previous notation in \cite{KotaniYamaguchi}.
\item spaces of dimension at most $7$, simply connected spaces of dimension at most $8$, and $k$-connected Poincar\'e duality spaces of dimension at most $2k+4$ ($k\geq 3$).
\item the action is \emph{hyperformal}.
\end{itemize}
\end{main}
In particular, this improves on respective separate partial results by Mun\~oz, Ustinovsky, Hilali, etc.
Many of these results surge from our refined terminology and new machinery, some are established by independent means.

The results do provide another geometric application in the field of manifolds admitting metric of non-negative sectional curvature, respectively for rationally elliptic manifolds (see Corollary \ref{cornonneg}).
\begin{maincor}\label{corB}
%
Suppose an effective $T^k$-action on a simply-connected closed manifold $M^n$ which is either $\mathcal{MOD}$-formal or of formal core, and let  $\dim H^*(M)< 2^{2k+1-n}$.
\begin{itemize}
\item
Suppose the action is isometric and $M$ is a non-negatively curved manifold $M^n$. Then the action is isotropy maximal, and $M$ is equivariantly diffeomorphic to a quotient of a free linear torus action on a product of spheres.
\item
Suppose $M$ is rationally elliptic. Then the action is isotropy maximal, and $M$ is equivariantly rationally equivalent to a quotient of a free linear torus action on a product of spheres.
\end{itemize}
\end{maincor}

As yet another easy outcome of our newly established approach, we can provide a systematic reproof of the conjecture
in the case when $M$ is a two-step nilpotent Lie algebra. Moreover, also the classical reasoning which shows that a Hard-Lefschetz manifold, e.g.~a K\"ahler manifold, satisfies the conjecture can be understood within and fits nicely to the new setting.

\bigskip


As mentioned already, these results fall off an elaborate discussion of properties and examples of $\mathcal{MOD}$-formal actions and actions with formal core. For example, we investigate these
\begin{itemize}
\item via their inheritance properties for products, and, more interestingly for several equivariant connected sums, and subgroup actions.

\item in terms of characterisations using $A_\infty$-algebras and $A_\infty$-modules.
\item via providing several intricate counter- and non-examples thereby sharpening terminology.
\end{itemize}
Nearly all of these discussions have a backlash and provide new insight into the toral rank conjecture in the respective cases.

\bigskip

The article draws on several concepts from the theory of Transformation Groups like equivariant cohomology, and on techniques from Rational Homotopy Theory, in particular on several characterisations of formality. In this context the term \emph{formal core} is constructed and analysed. Moreover, we also use $A_\infty$ and $C_\infty$ structures to characterise such formality properties together with the Hirsch--Brown model for further computations. This also contributes to investigating the concept of \emph{$\mathcal{MOD}$-formality} which is established in the context of module resolutions and Koszul complexes. As for geometric applications we use structure theory of manifolds of non-negative sectional curvature, K\"ahler and symplectic manifolds, and we deal with different equivariant connected sums.

We shall point the reader to the relevant literature, respectively, when relevant references were elusive, for the convenience of the reader, we added a rather self-contained introduction to the lacking concepts in the Appendix.

\bigskip

\str In Section \ref{sec2} we recall several classical definitions and properties before we shape the new terms of ``$\mathcal{MOD}$-formality'' and ``formal core''---central to this article---in Section \ref{sec:notionsection}. Section \ref{Massey} is devoted to characterizing these terms by means of $A_\infty$-structures.  This is put to use and combined with further techniques in \ref{secform} wherein the various aspects of Theorem \ref{theoA} are proved. We show that the new concepts are relevant in many situations of practical importance by discussing some special classes of spaces---amongst which we focus on symplectic actions, spaces without non-trivial derivations of negative degree on the rational cohomology algebra, rationally elliptic spaces, and manifolds of non-negative sectional curvature---in Section \ref{secspec}. Here, in particular, we prove Corollary \ref{corB}. Finally, in Section \ref{secex} we provide 
several concrete examples which sharpen our new definitions and relate them to other properties.

Background material on differential graded modules as well as $A_\infty$-algebras and modules can be found compiled in the appendix sections \ref{app:HB} and \ref{secstr}---although this material may be well-known to experts, parts of it seem hard to find in the literature which is why we included it here.

\bigskip

\ack
Large parts of this article were also used for the second named author's doctoral thesis \cite{ZollerThesis} written under the supervision of Oliver Goertsches whom he would like to thank for his support.

The first named author was supported both by a Heisenberg grant and his research grant AM 342/4-1 of the German Research Foundation; he is moreover associated to the DFG Priority Programme 2026, ``Geometry at Infinity''.
The second named author was supported by the German Academic Scholarship Foundation.

Moreover, the authors thank  Christopher Allday, Bernhard Keller, and Mark Walker for helpful conversations.

The authors are also grateful to Fernando Galaz-Garc\'ia, Steve Halperin, Martin Kerin, and Volker Puppe for their respective feedback on a previous version of the preprint, which, in particular, helped to improve the presentation of the article.


\section{Preliminaries}\label{sec2}
Throughout the article, $G$ will denote a compact and connected Lie-Group and $X$ will be a topological space with a continuous $G$-action. The $G$-spaces considered are assumed to be Hausdorff, connected, and have finite-dimensional rational cohomology. The latter always refers to singular cohomology. Coefficients will be taken in the field $\mathbb{Q}$ if not stated otherwise and will be suppressed in the notation.

Our main tool for studying topological aspects of a $G$-action on $X$ is the Borel fibration

\[X\rightarrow X_G \rightarrow BG\]
where $BG=EG/G$ for some contractible space $EG$ on which $X$ acts freely and $X_G$ is the orbit space of the diagonal action on $EG\times X$. The map $X_G\rightarrow BG$ is given by projection onto the first component.

The cohomology of $X_G$ is called the equivariant cohomology of $X$ and denoted by $H_G^*(X)$. The map $X_G\rightarrow BG$ induces a natural $H^*(BG)$-module structure on $H^*_G(X)$. The ring $H^*(BG)$ is a polynomial algebra with generators of even degree and will be denoted by $R$. In case $G$ is a torus, the generators of $R$ are all of degree $2$. The assumption $\dim H^*(X)<\infty$ ensures that $H_G^*(X)$ is finitely generated as an $R$-module (see \cite[Proposition 3.10.1]{AlldayPuppe2}).
Equivariant cohomology provides an essential link between geometry and algebra and captures many important properties of the group action. For example, the information of an action being almost free (meaning that all isotropy groups are finite) can immediately be read off from the algebraic data in case $X$ is compact.

\begin{thm}\label{thm:hsaiNGSTHEOREM}
Assume $X$ is compact. The $G$-action on $X$ is almost free if and only if \[\dim H_G^*(X)<\infty.\]
\end{thm}

This is a classical theorem due to Hsiang. In the torus case it is actually a direct consequence of the more general Lemma \ref{lem:codimisminimalorbitsdim}. We have the following supplementary proposition. Here $\mathrm{fd}(X)$ denotes the formal dimension, which is the highest integer $n$ such that $\dim H^n(X)\neq 0$.

\begin{prop}\label{thm:hsiangandfriends}
Assume $G$ acts almost freely on $X$.
\begin{enumerate}[(i)]
\item If $X$ is compact, we have $\mathrm{fd}(X_G)=\mathrm{fd}(X)-\dim G$.
\item If $X$ is a manifold, $X_G\rightarrow X/G$ induces an isomorphism in cohomology.

\item Suppose $X$ is compact, $H^*(X)$ satisfies Poincaré duality, and $H^1(X)=0$. Then $H^*_G(X)$ satisfies Poincar\'e duality.

\end{enumerate}
\end{prop}

Part $(i)$ is an easy observation using the Serre spectral sequence of the fibration (up to homotopy) $G\rightarrow X\rightarrow X_G$. For $(ii)$ see \cite[Theorem 7.6]{AMIG}. Finally, $(iii)$ follows from \cite[Theorems 3.6 and 4.3]{FHT2} applied to the Borel fibration $X\rightarrow{} X_G \to B G$. For this we use \cite[Proposition 3.4]{FHT2} in order to identify $B G$ as a Gorenstein space, as it is simply-connected with finite dimensional rational homotopy. Moreover, we note that the prerequiste of ``finite right $\operatorname{M-cat}$'' in \cite[Theorems 3.6]{FHT2} is satisfied, since $\dim H^*_G(X;\qq)<\infty$ implies finite Lusternik--Schnirelmann category. Finite LS-category however, in its algebraic definition on Sullivan models, is analogous to and actually stronger than the definition of ``finite right $\operatorname{M-cat}$'', as it requires the defining retract to be even multiplicative.

Rather contrary to the above case, where $H_G^*(X)$ is a torsion module, we formally recall the following classical

\begin{defn}
The $G$-action on $X$ is equivariantly formal if one of the following equivalent conditions holds:
\begin{enumerate}[(i)]
\item $H_G^*(X)$ is a free $R$-module.
\item $H_G^*(X)\cong R\otimes H^*(X)$ as modules.
\item The Serre spectral sequence of the Borel fibration is totally non-homologous to zero (TNHZ), i.e.\ it collapses at $E_2$.
\item The map $H_G^*(X)\rightarrow H^*(X)$ is surjective.
\end{enumerate}
\end{defn}
The equivalence of (ii), (iii), and (iv) follows from standard considerations on the Serre spectral sequence of the Borel fibration. For the equivalence of the condition (i) see \cite[Cor.\ 4.2.3]{AlldayPuppe2}, \cite[Prop.\ 2.3]{GoertschesRollenske}.

While cohomology is a powerful tool to extract information from the Borel fibration, we want to go deeper to the cochain level. Our main tool is the language of rational homotopy theory and in particular commutative differential graded algebras (cdga) and Sullivan models. We assume the reader is familiar with those theories and refer to \cite{Bibel} for missing definitions. The point here is to preemptively sort out technical difficulties that arise in later discussions, as well as to comment on the problem of realizing algebra through geometry. We fix some notation: if $V$ is a graded vector space, then $\Lambda V$ will denote the free unital commutative graded algebra on $V$. If not stated otherwise, all cdgas will be assumed to be non-negatively graded. We will furthermore assume all (cd)gas to be unital which means that they come with a fixed multiplicative unit element which is preserved by morphisms.

One of the central objects in this article is the following construction: we fix a Sullivan minimal model of the algebra $A_{pl}(BG)$ of the piecewise linear forms on $BG$. It is of the form $(R,0)$ where $R=H^*(BG)$. The fibration induces a map $(R,0)\rightarrow A_{pl}(X_G)$ for which we choose a relative minimal model. This results in an extension sequence
\[(R,0)\rightarrow (R\otimes \Lambda V,D)\rightarrow (\Lambda V,d)\]
where the first map is the inclusion of the relative minimal model and the second one is the projection onto $\Lambda V$ (with the induced differential) via the canonical augmentation of $R$ that sends $R^+$ to $0$ and is the identity on $R^0=\mathbb{Q}$. Then said projection is actually a Sullivan model for $X\rightarrow X_G$. We will refer to such an extension sequence as a (minimal) model for the Borel fibration.

\begin{rem}\label{rem:relmodelseindeutig}
\begin{enumerate}[(i)]
\item Throughout the article, we will work with a fixed minimal model $R\rightarrow A_{pl}(BG)$.
As the minimal model of the Borel fibration will be a central object, we want to point out that it is independent of the choices made in the construction: if $(R',0)\rightarrow A_{pl}(BG)$ is another minimal model and $(R'\otimes\Lambda V',D')$ is a relative minimal model for $R'\rightarrow A_{pl}(X_G)$, then there is an isomorphism $(R\otimes\Lambda V,D)\cong (R'\otimes \Lambda V',D')$ that restricts to an isomorphism $R\cong R'$. This follows from the basic homotopy theory of cdgas: by uniqueness of the minimal model there is an isomorphism $R\cong R'$ such that $\varphi_0\colon R\rightarrow A_{pl}(X_G)$ and $\varphi_1\colon R\cong R'\rightarrow A_{pl}(X_G)$ are homotopic. Let $h\colon R\rightarrow A_{pl}(X_G)\otimes (t,dt)$ be a homotopy such that $\varphi_0=p_0\circ h$ and $\varphi_1=p_1\circ h$, where $p_i$ is evaluation at $t=i$. Consider the commutative diagram

\[\xymatrix{
 R\ar[r]\ar[d] &  A_{pl}(X_G)\otimes (t,dt)\ar[d]^{p_0}\\
 R\otimes \Lambda V\ar[r]^{\varphi_0}\ar@{-->}[ur] & A_{pl}(X_G)
}\]

with the dashed arrow obtained by relative lifting (see e.g.\ \cite[Lemma 14.4]{Bibel}). As a result we deduce that by changing the homotopy class of $\varphi_0$, through composing the dashed arrow with $p_1$, we can obtain a relative minimal model for $\varphi_1$. Then the claim follows from uniqueness of the relative minimal model.

\item The Sullivan model $(R\otimes \Lambda V,D)$ for $X_G$ is not minimal in general (however it is always minimal for torus actions on simply-connected spaces). Still, it is usually our preferred choice of model since it comes with a fixed $R$-module structure. Equivariant maps induce morphisms that respect this structure: given an equivariant map between $G$-spaces, we obtain a strictly commutative diagram between the Borel fibrations. If minimal models of the Borel fibrations are constructed as above, then one can show through relative lifting (see \cite[Prop.\ 14.6]{Bibel}) that there is a strictly commutative diagram
\[\xymatrix{
(R,0)\ar[r]\ar[d]^{\Id_R} & (R\otimes \Lambda V,D)\ar[r]\ar[d] &(\Lambda V,d)\ar[d]\\
(R,0)\ar[r] & (R\otimes \Lambda W,D)\ar[r] & (\Lambda W,d)
}\]
in which the rows are the relative minimal models and the horizontal morphisms are Sullivan representatives for the corresponding maps between the Borel fibrations. If such structure is present, we will usually assume the morphisms to be of this type.
\end{enumerate}
\end{rem}

As we will often care about $R$-module structures, the following lemma will be useful throughout the article. By an $R$-cdga we mean a morphism $(R,0)\rightarrow B$ of cdgas. We will often just write $B$ in case the specific morphism is not important or clear from the context.

\begin{lem}\label{lem:fbchar}
\begin{enumerate}[(i)]
\item For any cdga $B$, two morphisms $(R,0)\rightarrow B$ are homotopic if and only if they induce the same map on cohomology.
\item Consider any $R$-cdga $B$ and a relative Sullivan algebra $(R,0)\rightarrow(R\otimes \Lambda V,D)$. Then any morphism $(R\otimes\Lambda V,D)\rightarrow B$ such that $\varphi^*\colon H^*(R\otimes \Lambda V)\rightarrow H^*(B)$ respects the $R$-module structures is homotopic to a morphism of $R$-cdgas.
\end{enumerate}
\end{lem}

\begin{proof}
For the proof of $(i)$, let $f,g$ be two such morphisms and $R=\Lambda(X_1,\ldots,X_r)$. Then, by assumption, for any $X_i$ there is some $v_i\in B$ satisfying $D(v_i)=g(X_i)-f(X_i)$.
We can define a homotopy $h\colon R\rightarrow B\otimes \Lambda(t,dt)$, with $t$ of degree $0$, by setting \[h(X_i)=f(X_i)+(g(X_i)-f(X_i))t-v_idt.\]
In the situation of $(ii)$, we have a diagram
\[\xymatrix{R \ar[dr]\ar[d]& \\ R\otimes\Lambda V\ar[r] & B}\]
which commutes on the level of cohomology. By $(i)$ it is homotopy commutative so the claim follows by extension of homotopies (see e.g.\ \cite[Proposition 2.22]{AMIG})
\end{proof}

\begin{rem}\label{rem:andereextseq} We will frequently make use of the fact that we can obtain the model $(\Lambda V,d)$ of a space $X$ from the (preferred) model of $X_G$ by forming another extension sequence
\[(R\otimes \Lambda V,D)\rightarrow (R\otimes \Lambda V\otimes S,D)\rightarrow (S,0),\]
where $S=\Lambda(s_1,\ldots,s_r)$ is generated in odd degrees and $D$ maps the $s_i$ bijectively to the generators of $R$. In fact, we have $S=H^*(G)$. Sending $R$ and $S$ to $0$ yields a quasi-isomorphism $(R\otimes \Lambda V\otimes S,D)\simeq (\Lambda V,d)$. The extension sequence above is a model for $G\rightarrow X\rightarrow X_G$, which is a fibration up to homotopy equivalence.
\end{rem}

For free torus actions, it is also possible to pass from algebra to geometry: given a base space $Y$, any choice of $r$ classes from $H^2(Y)$ defines a morphism $R\rightarrow H^*(Y)$ which lifts (uniquely up to homotopy by Lemma \ref{lem:fbchar}) to a map $R\rightarrow A_{pl}(Y)$. In particular we obtain a unique relative minimal model of the form $R\rightarrow R\otimes \Lambda V$.
 In the proposition below, we expand on the discussion in \cite[Prop.\ 7.17]{AMIG} and show that such algebraic data is always realizable by the Borel fibration of a free torus action.

\begin{prop}\label{prop:realizeasaction}
Let $R=\Lambda(X_1,\ldots,X_r)$ with $X_i$ in degree $2$, $(\Lambda V,d)$ be a finite type minimal Sullivan algebra, and
\[(R,0)\rightarrow (R\otimes \Lambda V,D)\rightarrow (\Lambda V,d)\]
be an extension sequence with maps given by canonical inclusion and projection. Then there exists a free $T^r$-action on some space $X$ such that the above sequence is the minimal model of the associated Borel fibration.
If the cohomology of the middle cdga is finite-dimensional, then we can take $X$ to be compact. If additionally $H^*(\Lambda V,d)$ is simply-connected, satisfies Poincaré duality with fundamental class in degree $n$, and $n-r$ is not divisible by $4$, then we can take $X$ to be a compact simply-connected manifold.
\end{prop}

In the above proposition, $(\Lambda V,d)$ is the minimal model of $X$. In particular, the cohomology of $X$ is not necessarily finite-dimensional, which is an exception to the general assumption of all $G$-spaces having finite-dimensional cohomology. Note that even if $\dim H^*(\Lambda V,d)<\infty$, we can only choose $X$ to be compact if also the cohomology of $(R\otimes \Lambda V,D)$ is finite-dimensional as otherwise Theorem \ref{thm:hsaiNGSTHEOREM} would be violated.

\begin{proof}
Let $Y$ be a CW-complex with Sullivan model $(R\otimes \Lambda V,D)$. We find an integer $k$ such that the cohomology classes of the $kX_i$ in $H^2(Y;\mathbb{Q})$ come from classes in $H^2(Y;\mathbb{Z})$. Those uniquely determine the homotopy class of a map $f\colon Y\rightarrow K(\mathbb{Z}^r,2)=BT^r$. The minimal model $R\rightarrow A_{pl}(BT^r)$ can be chosen in a way such that the canonical inclusion $R\rightarrow R\otimes V$ is a Sullivan model for $f$. Pulling back the universal principal bundle along $f$ yields a principal bundle
\[T^r\rightarrow X\rightarrow Y.\]
We consider the following commutative diagram of principal $T^r$-bundles

\[\xymatrix{T^r\ar[d] & T^r\ar[l]\ar[d]\ar[r]& T^r\ar[d]\ar[r] & T^r\ar[d]\\
ET^r\ar[d] & ET^r\times X\ar[l]\ar[d]\ar[r] & X\ar[d]\ar[r] & ET^r\ar[d]\\
BT^r & X_{T^r}\ar[l]\ar[r] &Y\ar[r] & BT^r
}\]
in which the left morphism of principal bundles is given by projection on the first component, the central one is projection on the second component, and the right one is the pullback diagram induced by $f$. Note that the central morphism actually consists of weak equivalences so $(R\otimes \Lambda V,D)$ is a model for $X_{T^r}$. By naturality of the spectral sequence, the transgressions of the associated Serre spectral sequence commute with the maps between the base spaces. Thus the inner triangles in the diagram

\[\xymatrix{ & & H^2(Y)\\
H^2(BT^r)\ar@/^15pt/[rru]^{f^*} \ar@/_15pt/[rrd] & H^1(T)\ar[ur]\ar[l]\ar[rd] & \\ & & H^2(X_{T^r})\ar[uu]
}\]
are commutative. But the transgression $H^1(T^r)\rightarrow H^2(BT^r)$ is actually an isomorphism so the whole diagram commutes. Since $R$ is generated in degree $2$, it follows by Lemma \ref{lem:fbchar} that the canonical inclusion $R\rightarrow R\otimes \Lambda V$ is not only a model for $f$ but also for the Borel fibration of $X$.

If $(R\otimes \Lambda V,D)$ has finite-dimensional cohomology we can choose $Y$ to be a finite CW-complex and homotope $f$ such that it has image in some compact skeleton. As a consequence, $X$ will be compact. If $(V,d)$ is additionally simply-connected and satisfies Poincaré duality with fundamental class of degree $n$, then $(R\otimes \Lambda V,D)$ satisfies Poincaré duality with fundamental class in degree $n-r$ (by the same reasoning as Proposition \ref{thm:hsiangandfriends} $(iii)$). Then by \cite[Theorem 3.2]{AMIG} we can choose $Y$ as a compact simply-connected manifold. As before we can homotope $f$ to have image in a compact skeleton which is in fact contained in some $(\mathbb{C}P^N)^r$ for $N$ large enough. If we go on to homotope $f$ to a smooth map, then $X$ will be a smooth compact simply-connected manifold.
\end{proof}

\section{New definitions and their properties}\label{sec:notionsection}

\subsection{$\mathcal{MOD}$-Formality}
Let $X$ be a $G$-space and $(R,0)\rightarrow(R\otimes \Lambda V,D)\rightarrow(\Lambda V,d)$ be a model of the Borel fibration. In particular $(R\otimes \Lambda V,D)$ is a Sullivan model for $X_G$. When it comes to common generalizations of equivariantly formal actions and actions with formal homotopy quotient, we have the following natural

\begin{defn}
The action is called $\mathcal{MOD}$-formal if $(R\otimes\Lambda V,D)$ is formal as a differential graded $R$-module (dg$R$m), i.e.\ it is connected to $(H_G^*(X),0)$ via quasi-isomorphisms of dg$R$ms.
\end{defn}

\begin{rem}
In the definition above, as opposed to the condition of $X_G$ being a formal space, we only require the quasi-isomorphisms to be multiplicative with respect to the classes coming from $BG$. In this way one gets rid of formality obstructions that exist within $X$ independently of the action.
\end{rem}

\begin{lem}\label{lem:efismodformal} Equivariantly formal actions and actions with formal homotopy quotient are $ \mathcal{MOD}$-formal.
\end{lem}

\begin{proof}
If $X_G$ is formal, there is a quasi-isomorphism of cdgas $(R\otimes\Lambda V,D)\rightarrow (H_G^*(X),0)$ which covers the canonical projection on closed elements. This is in particular a morphism of dg$R$ms so the statement follows.

Now if $H_G^*(X)$ is free, let $b_1,\ldots,b_k\in R\otimes\Lambda V$ be representatives of an $R$-basis. Then the inclusion $H_G^*(X)\cong R\otimes \langle b_1,\ldots,b_k\rangle_\mathbb{Q}\rightarrow R\otimes\Lambda V$ induces an isomorphism on cohomology if we take the differential on the left hand side to be trivial.
\end{proof}

\begin{ex}\label{ex:trivialex}
When it comes to examples of actions that are neither equivariantly formal nor have a formal homotopy quotient but satisfy $\mathcal{MOD}$-formality, there are a few trivial candidates: For example every $S^1$-action can be seen to automatically be $\mathcal{MOD}$-formal (see Remark \ref{rem:S1modformal}). Also take any $G$-action on $X$ such that $X_G$ is formal and let $Y$ be a non-formal space. Then $(X\times Y)_G=X_G\times Y$ is not formal (see \cite[Prop. 5]{BodyDouglas}). The Hirsch--Brown model of $X\times Y$ however arises from the Hirsch--Brown model $(R\otimes H^*(X),D)$ of $X$ by tensoring with $H^*(Y)$ and extending the differential to $R\otimes H^*(X)\otimes H^*(Y)$ in the obvious way. A quasi-isomorphism $(R\otimes H^*(X),D)\simeq (H_G^*(X),0)$ thus induces a quasi-isomorphism $(R\otimes H^*(X)\otimes H^*(Y),D)\rightarrow (H_G^*(X)\otimes H^*(Y),0)$ so the action is $\mathcal{MOD}$-formal.
The discussion in Section \ref{Massey} is helpful for constructing more interesting examples as Example \ref{ex:nontrivmodform}.
\end{ex}

One of the defining traits of $\mathcal{MOD}$-formal actions is given by the following observation:
Let $X$ be a $\mathcal{MOD}$-formal $G$-space. As minimal models of dg$R$ms are unique among a quasi-isomorphism type, it follows that the Hirsch--Brown model of the action is the minimal model of the dg$R$m $(H^*_G(X),0)$. To construct the latter, recall the notion of minimal graded free resolution from commutative algebra: a graded free resolution of $H^*(X_{G})$ is an exact complex
\[0\leftarrow H^*_G(X)\leftarrow F_0\xleftarrow{d} F_1\xleftarrow{d}\ldots\xleftarrow{d} F_r\leftarrow 0\]
consisting of free graded $R$-modules and graded maps. In the usual conventions, those maps are of degree $0$ but we apply suitable degree shifts to consider them to be of degree $1$. The resolution is said to be minimal if $d(F_i)\subset \mathfrak{m}F_{i-1}$, where $\mathfrak{m}=R^{+}$ is the maximal homogeneous ideal. Thus for a minimal resolution, the projection map
\[\bigoplus_{i=0}^r F_i\rightarrow F_0/d(F_1)\cong H^*_G(X)\]
defines a minimal dg$R$m-model for $(H^*_G(X),0)$, where we equip $\bigoplus_i F_i$ with the differential that sends $F_0$ to $0$ and equals $d$ on $F_i$ for $i\geq 1$. We have shown

\begin{thm}\label{thm:modformalfreeres}
An action is $\mathcal{MOD}$-formal if and only if the minimal Hirsch--Brown model is isomorphic to the minimal graded free resolution of the $R$-module $H_G^*(X)$ as a differential graded $R$-module.
\end{thm}

\begin{rem}
It is, of course, not true that the rational homotopy type of $X$ and $X_T$ is ``a formal consequence'' of the $R$-Algebra structure on $H_G^*(X)$, as is the case for actions with formal homotopy quotient. However for $\mathcal{MOD}$-formal actions, the spirit of formality still lives on in the fact that the (additive) cohomology of $X$ can be retrieved from the $R$-module $H_G^*(X)$. Also, the algebra structure on the image of $H^*_G(X)\rightarrow H^*(X)$ can be reconstructed from the $R$-algebra structure of $H^*_G(X)$, which holds more generally for spherical actions which we define below. Note however that this map is usually not surjective as its surjectivity is equivalent to classical equivariant formality.
\end{rem}

\begin{rem}\label{rem:S1modformal}
In view of the characterization of Theorem \ref{thm:modformalfreeres}, one can see from the explicit construction of the minimal Hirsch--Brown model in Proposition \ref{HBconstruction} that any $S^1$-action is $\mathcal{MOD}$-formal. Assume that we have constructed the model up until a certain degree and that it has the structure of a free resolution
\[0\leftarrow F_0\xleftarrow{D} F_1\leftarrow 0\]
with cohomology concentrated in $F_0$. Then when adding another generator to generate cohomology, we can obviously add it to $F_0$ and keep the structure of a free resolution. When adding a generator $\alpha$ in order to kill cohomology, we can choose $D(\alpha)$ to be some (non-exact) element of $F_0$. Also $D$ is injective on $F_1\oplus R\alpha$, which means the free resolution structure is preserved. To see this, we write $R=\mathbb{Q}[X_1]$ and assume the existence of $v\in F_1$ with $D(v+X_1^k\alpha)=0$ for some $k\geq 1$. Then, since $\alpha$ is of maximal degree among the generators, $v$ is divisible by $X_1^k$ so $D\alpha=-D(vX_1^{-k})\in D(F_1)$ was already exact, which is a contradiction.
\end{rem}

Before investigating further aspects of formality and their relations, we give some reformulations of $\mathcal{MOD}$-formality which will be useful throughout the article.

\begin{lem}\label{formality criterion}
Let $(R\otimes H^*(X),D)$ be the Hirsch--Brown model of a $G$-action on a space $X$. The following are equivalent:
\begin{enumerate}[(i)]
\item The action is $\mathcal{MOD}$-formal.
\item There exists a splitting $R\otimes H^*(X)=V\oplus W$ of the Hirsch--Brown model into free submodules such that $D(V)=0$ and every closed element in $W$ is exact.
\item There is a vector space splitting $R\otimes H^*(X)=\ker D\oplus C$ such that $C\oplus\im D$ is an $R$-submodule.
\end{enumerate}
\end{lem}

\begin{proof}
As we have seen, in case the action is $\mathcal{MOD}$-formal, we can choose the minimal Hirsch--Brown model to be the minimal free resolution $\bigoplus F_i$ of $H_G^*(X)$. The desired decomposition in $(ii)$ is given by $V=F_0$ and $W=\bigoplus_{i\geq 2} F_i$.

In the situation of $(ii)$ we can choose a vector space splitting $R\otimes H^*(X)=\ker D\oplus C$ such that $C\subset W$. Now if $\alpha$ is any $R$-linear combination of elements in  $C\oplus \im D$, then, since $\im D=\im D|_C$, we find some $c\in C$ with $D\alpha=Dc$. Thus $\alpha-c$ is the sum of exact elements and closed elements in $W$ which are exact by assumption. Consequently we have $\alpha-c\in\im D$ and $\alpha\in C\oplus\im D$.

Finally, assume $(iii)$ holds. We define a map $\varphi$ from $R\otimes H^*(X)$ to its cohomology such that it is the canonical projection on $\ker D$ and trivial on $C$. This obviously commutes with the differential (which is trivial on $H_G^*(X)$) and all that remains to show is the $R$-linearity of $\varphi$. This clearly holds on $\ker D$ and by assumption, for any $c\in C$, $r\in R$ we have $rc\in C\oplus \im D=\ker \varphi$ which proves the claim.
\end{proof}

\begin{defn}
We say that an action is \emph{almost $\mathcal{MOD}$-formal} if $\dim H^*(X)$ is equal to the rank of the minimal free resolution of $H_G^*(X)$ as an $R$-module.
\end{defn}

\begin{rem}\label{rem:EilenbergMoore}
The Eilenberg-Moore spectral sequence of the Borel fibration converges to $H^*(X)$ and has $E_2^{*,*}=\mathrm{Tor}_{R}^{*,*}(H_G^*(X),R/\mathfrak{m})$. By the definition of $\mathrm{Tor}$, the $\mathbb{Q}$-dimension of the right hand expression is precisely the total rank of the minimal free resolution of $H_G^*(X)$. We deduce that almost $\mathcal{MOD}$-formality is equivalent to the $E_2$-degeneration of the Eilenberg-Moore spectral sequence.
\end{rem}

In case $G=T$ is a torus, almost $\mathcal{MOD}$-formality has the following interpretation in terms of another spectral sequence (see also \cite[Lemma 1.4]{ustinovsky} and \cite[Remark 7]{munoz}):
taking one step back in the Barratt-Puppe sequence of the Borel fibration, we obtain the fibration (up to homotopy equivalence)
\[T\rightarrow X\rightarrow X_T,\]
which in the free case is equivalent to $T\rightarrow X\rightarrow X/T$.

\begin{prop}\label{E3collapse} The $T$-action is almost $\mathcal{MOD}$-formal if and only if the
Serre spectral sequence of \[T\rightarrow X\rightarrow X_T\]
degenerates at the $E_3$ term.
\end{prop}

\begin{proof}
The second page of the spectral sequence is the Koszul complex $H_T(X)\otimes S$ where $S=\Lambda (s_i)$ consists of a degree 1 generator for each variable $X_i$ of $R$ and the differential maps $s_i$ to the image of $X_i$ in $H_T^2(X)$ (see Remark \ref{rem:andereextseq}). Thus the $E_3$-page is precisely $\mathrm{Tor}_R(H_T^*(X),R/\mathfrak{m})$. By the commutativity of $\mathrm{Tor}$, we deduce that $\dim_\mathbb{Q}E_3$ is the rank of the minimal free resolution of $H_T^*(X)$. As the spectral sequence converges to $H^*(X)$, this is equal to $\dim H^*(X)$ if and only if there are no more nontrivial differentials starting from $E_3$.
\end{proof}

Before we study implications, let us discuss one more closely related property. The choice of name in the following definition is motivated by \cite[Section 8]{StasheffHalperin}, which discusses similar properties in the context of path space fibrations.

\begin{defn}
We call the action \emph{spherical} if $\ker(H^*_G(X)\ra H^*(X))$ is equal to $\mathfrak{m}H^*_G(X)$.
\end{defn}

Of course, $\mathfrak{m}H_G^*(X)$ is always contained in the kernel because on the level of Sullivan models the map is just the projection
\[R\otimes \Lambda V\rightarrow\Lambda V\]
obtained by sending $\mathfrak{m}$ to $0$. However the kernel on the level of cohomology may be bigger for there may be Massey products represented in $\mathfrak{m}\otimes \Lambda V$ which on the level of cohomology might not lie in the multiples of $\mathfrak{m}$, see e.g.\ Example \ref{ex:restriction}. In this light, being spherical is a restriction on the existence of nontrivial Massey products and therefore related to formality properties.
Again, for $G=T$ we can express this as degeneracy in a spectral sequence.

\begin{prop}\label{sphericalSS}
The action of a torus $T$ on $X$ is spherical if and only if in the Serre spectral sequence of
\[T\rightarrow X\rightarrow X_T\]
no nontrivial differentials enter $E_r^{*,0}$ for $r\geq 3$.
\end{prop}

\begin{proof}
Since the map $H_T^*(X)\rightarrow H^*(X)$ factors as $H^*(X_T)\cong E_2^{*,0}\rightarrow E_\infty^{*,0}\subset H^*(X)$, it suffices to analyse the spectral sequence. From the description of the $E_2$-page in the proof of Proposition \ref{E3collapse} we deduce that the image of $d_2$ in $E_2^{*,0}$ corresponds exactly to $\mathfrak{m}H_T^*(X)$. Thus there is more in the kernel if and only if there is a nontrivial differential mapping to the bottom row after the $E_2$ page.
\end{proof}

Of course, the above proposition as well as Proposition \ref{E3collapse} can be (less elegantly) formulated for arbitrary compact $G$. The problem however is that the cohomological generators of $H^*(G)$ will not transgress on the $E_2$-page for degree reasons.
Still, for a general $G$-action we have the following

\begin{thm}
For any $G$-action we have the implications \[\mathcal{MOD}\text{-formal} \Rightarrow \text{almost $\mathcal{MOD}$-formal} \Rightarrow spherical.\]
\end{thm}

\begin{proof}
The first implication is a consequence of Theorem \ref{thm:modformalfreeres}.
For torus actions the second implication is a consequence of the above propositions. In the general setting we instead consider the Eilenberg-Moore spectral sequence of the Borel-fibration. The column $E_\infty^{0,*}$ can be identified with the image of the map $H^*_G(X)\rightarrow H^*(X)$ (see \cite[Exercise 7.5]{McCleary}).
This is a subspace of $E_2^{0,*}=\mathrm{Tor}^{0,*}_R(H_G^*(X),R/\mathfrak{m})=H_G^*(X)/\mathfrak{m}H_G^*(X)$. Hence for dimensional reasons (all objects are degree wise finite-dimensional), being spherical is equivalent to the vanishing of all differentials starting in this column. By Remark \ref{rem:EilenbergMoore} this holds if the action is almost $\mathcal{MOD}$-formal.
\end{proof}

In the above theorem, none of the converse implications hold, not even for simply-connected compact manifolds. This is demonstrated by Examples \ref{ex:sphericalnotalmostmod} and \ref{ex:almostmodnotmod} as well as Remark \ref{rem:manifoldcounterexample}. In the discussion, the following description of spherical actions will be helpful.

\begin{lem}\label{spherical-generator-lem}
The $G$-action on $X$ is spherical if and only if there is a generating set of the $R$-module $H_G^*(X)$ such that the restriction $H_G^*(X)\rightarrow H^*(X)$ is injective on its $\mathbb{Q}$-span.
\end{lem}

\begin{proof}
The condition is equivalent to the existence of a $\mathbb{Q}$-subspace $V\subset H_G^*(X)$ such that the projection $V\rightarrow H_G^*(X)/\mathfrak{m}H_G^*(X)$ is surjective and the projection $V\rightarrow H_G^*(X)/\ker(r)$ is injective. Since all spaces are degreewise finite-dimensional and $\mathfrak{m}H_G^*(X)\subset \ker(r)$, the condition is equivalent to equality in the last inclusion.
\end{proof}

\subsection{Formally based actions}
As for the previously introduced notions, we search for ways to impose the triviality of certain Massey products in the equivariant cohomology and thus create a more local version of formality.

One of the more direct ways to do this is to demand that, rationally, the map $X_G\rightarrow BG$ factors cohomologically injectively through a formal space. On models this is equivalent to the following condition: let \[(R,0)\rightarrow (R\otimes \Lambda V,D)\rightarrow (\Lambda V,d)\]
be a model of the Borel fibration of the action. Let $A\subset H_G^*(X)$ be a subalgebra that contains $\im(R\rightarrow H^*_G(X))$ and let $(C,d)\simeq (A,0)$ be a relative minimal model of the canonical morphism $(R,0)\rightarrow(A,0)$.

\begin{defn}\label{def:formallybased}
We call the action \emph{formally based with respect to} $A$ if there exists a morphism $(C,d)\rightarrow (R\otimes \Lambda V,D)$ of $R$-cdgas for which the induced map \[A=H^*(A,0)\cong H^*(C,d)\rightarrow H^*(R\otimes \Lambda V,D)\cong H_G^*(X)\]
is the inclusion. If the action is formally based with respect to $\im(R\rightarrow H^*_G(X))$, we also just refer to it as being \emph{formally based}.
\end{defn}

\begin{rem}\begin{enumerate}[(i)]\label{rem:formallybased}
\item The existence of the morphism $(C,d)\rightarrow (R\otimes \Lambda V,D)$ in the definition does not depend on the particular choice of relative minimal models (see Remark \ref{rem:relmodelseindeutig}). However, note that its homotopy class may not be unique.
\item If an action is formally based with respect to some $A$, then it is automatically formally based with respect to any $A'$ satisfying $\im(R\rightarrow H^*_G(X))\subset A'\subset A$ as the inclusion $A'\subset A$ lifts to the Sullivan models. In particular the action is formally based if the homotopy quotient $X_G$ is formal.
\item The existence of an orbit with isotropy of maximal rank implies the injectivity of $R\rightarrow
 H^*_G(X)$ (the converse follows by Borel localization in case $X$ is compact). In this case, for $A=\im(R\rightarrow H^*_G(X))$, one may choose $(C,d)=(R,0)$ so the action is formally based.
\end{enumerate}
\end{rem}

A special case of this is given by the following

\begin{defn}
An action is \emph{hyperformal} if the kernel of $R\rightarrow H^*_G(X)$ is generated by a homogeneous regular sequence.
\end{defn}

\begin{lem}\label{lem:hyperformal is formally based}
Hyperformal actions are formally based.
\end{lem}

\begin{proof}
In the notation above we have $A=R/(a_1,\ldots,a_k)$ for homogeneous $a_i$ which form a regular sequence in $R$. A relative minimal model of $(R,0)\rightarrow(A,0)$ is given by $(C,d)=(R\otimes \Lambda (s_1,\ldots,s_k),d)$ with $d|_R=0$ and $d(s_i)=a_i$. The fact that the $a_i$ form a regular sequence implies that the map $C\rightarrow A$, defined by sending the $s_i$ to $0$ and projecting $R$ canonically to $A$, is a quasi-isomorphism. Also the $a_i$ are exact in $R\otimes \Lambda V$ by definition. We choose $z_i\in R\otimes \Lambda V$ with $D(z_i)=a_i$. The desired lift $C\rightarrow R\otimes \Lambda V$ is now obtained by sending $s_i$ to $z_i$ and $R$ identically to $R$.
\end{proof}

It is natural to not only demand ``formality of the image'' of $H^*(BG)\rightarrow H^*_G(X)$ through the formally based condition but to also pay attention to how said image is embedded in the ambient space. Especially with regards to the TRC, it is interesting to impose additional degeneracy conditions. For example, this shows up in the requirements of Theorem \ref{MasseyThm} compared to those of Theorem \ref{thm:Masseyformallybased}.

We will investigate another such condition:
given an action that is formally based with respect to some $A\subset H^*_G(X)$, we define the algebra $(C,d)$ as above. We may form the cdga $(\overline{C},\overline{d})$ where $\overline{C}=C/(R^+)$, which is again a Sullivan algebra. Now any morphism $(C,d)\rightarrow (R\otimes \Lambda V,D)$ of $R$-cdgas induces a morphism $(\overline{C},\overline{d})\rightarrow (\Lambda V,d)$ into the model of $X$ because $\Lambda V=(R\otimes \Lambda V)/(R^+)$.

\begin{defn}\label{def:injformbased}
Let $A$ be an $R$-subalgebra of $H_G^*(X)$. We say an action has \emph{formal core with respect to $A$} if there is a morphism $(C,d)\rightarrow (R\otimes \Lambda V,D)$ as in Definition \ref{def:formallybased} such that additionally the induced morphism $(\overline{C},\overline{d})\rightarrow (\Lambda V,d)$, obtained by dividing out $R^+$, is cohomologically injective. If such an $A$ exists, we also just refer to the action as having \emph{formal core}.
\end{defn}

\begin{rem}\begin{enumerate}[(i)]\label{rem:formalcore}
\item If $X_G$ is formal, then we can choose $A=H_G^*(X)$, $C=R\otimes \Lambda V$. The map $\overline{C}\rightarrow \Lambda V$ is just the identity and the action has formal core.
\item Other than for the notion of being formally based, having formal core with respect to $A$ does not imply formal core with respect to any $A'\subset A$: the injectivity of $H^*(\overline{C},\overline{d})\rightarrow H^*_G(X)$ depends on the particular choice of $A$. This can be observed in Example \ref{ex:biquotnichtinjbased}.
\item In case the action has an orbit with isotropy of maximal rank, as in Remark \ref{rem:formallybased}, the action is formally based with $(C,d)=(R,0)$. Thus $\overline{C}=\mathbb{Q}$ and the action automatically has formal core.
\item Any $S^1$-action has formal core: we are either in case $(iii)$ or $\im(R\rightarrow H_{S^1}(X))$ is isomorphic to $R/(X_1^n)$ for some $n$, where $R=\mathbb{Q}[X_1]$. We may thus set $C=(R\otimes \Lambda (\alpha),D)$ with $D\alpha= X_1^n$ and obtain a lift $C\rightarrow R\otimes \Lambda V$ by mapping $\alpha$ to some $\beta$ with $D\beta=X_1^n$. It suffices to argue that the projection $\overline{\beta}\in \Lambda V$ is not exact. If $d\gamma=\overline{\beta}$, then $\beta-D\gamma$ is divisible by $X_1$ and thus $D(X_1^{-1}(\beta-D\gamma))=X_1^{n-1}$ which is a contradiction.
\end{enumerate}
\end{rem}

\begin{prop}
Let $f\colon X\rightarrow Y$ be an equivariant map of $G$-spaces that induces injections on both regular and equivariant cohomology (e.g.\ an equivariant retract). If the action on $Y$ has formal core, then so does the action on $X$.
\end{prop}

\begin{proof}
The equivariant map $f$ induces a commutative diagram
\[\xymatrix{R\ar[r]& \mathcal{M}_G^X\ar[r]&\mathcal{M}^X\\
R\ar[r]\ar[u]_{\Id_R} & \mathcal{M}_G^Y\ar[r]\ar[u]_{f_G^*} &\mathcal{M}^Y\ar[u]_{f^*}}\]
in which the rows are relative minimal models for the Borel fibrations of $X$ and $Y$. Suppose $Y$ has formal core with respect to $A\subset H_G^*(Y)$ and let $\iota\colon (C,d)\rightarrow \mathcal{M}_G^Y$ and $\overline{\iota}\colon (\overline{C},\overline{d})\rightarrow \mathcal{M}^Y$ be constructed as above. By assumption the maps $f_G^*\circ\iota$ and $f\circ\overline{\iota}$ are both injective on cohomology. Consequently the action on $X$ has formal core with respect to $f_G^*(A)\subset H_G^*(X)$.
\end{proof}

\begin{thm}\begin{enumerate}[(i)]
\item Let $G$ act on $X$ such that the action is almost free and formally based with respect to $A\subset H_G^*(X)$. If $A$ satisfies Poincaré duality, then the action has formal core with respect to $A$.
\item Any hyperformal action has formal core.
\end{enumerate}
\end{thm}

\begin{proof}
Let $C$ and $\overline{C}$ be defined as above. By assumption
we have a commutative diagram
\[\xymatrix{(C\otimes S,d)\ar[r]\ar[d] & (R\otimes \Lambda V\otimes  S,D)\ar[d]\\
(\overline{C},\overline{d})\ar[r] &(\Lambda V,d)}\]
where the differentials in the top row map the generators of $S=\Lambda (s_1,\ldots,s_r)$ bijectively onto the generators of $R=\Lambda (X_1,\ldots,X_r)$ and the vertical maps are quasi-isomorphisms defined by sending $R^+$ and $S^+$ to $0$. It suffices to show cohomological injectivity of the top horizontal map which we denote by $\varphi$.

As $A=H^*(C)$ is a Poincaré duality algebra, it follows that the cohomology of $C\otimes S$ satisfies Poincaré duality as well (see \cite[Lemma 38.2]{Bibel}). The fundamental class of $C\otimes S$ is contained in every nontrivial ideal of
$H^*(C\otimes  S)$. Hence it maps to $0$ under \[\varphi^*\colon H^*(C\otimes S)\rightarrow H^*(R\otimes \Lambda V\otimes S)\]
if and only if $\ker \varphi^*\neq 0$. As a consequence, we only need to prove injectivity of $\varphi^*$ on the top degree cohomology.

To see this, filter $C\otimes  S$ by degree in $C$ and $R\otimes \Lambda V\otimes S$ by degree in $R\otimes \Lambda V$. The morphism $\varphi$ respects this filtration and consequently induces a map between the spectral sequences. On the $E_2$ page, this morphism is given by the inclusion
\[A\otimes S\rightarrow H_G^*(X)\otimes S.\]
Since the top degree cohomology of $A\otimes S$ is located in the top row, its image on the right cannot be killed by any differential and hence induces a nontrivial element on the $\infty$-page. This shows that $\varphi^*$ is not $0$ in the top degree which concludes the proof of $(i)$.

In the situation of $(ii)$, if the kernel of $R\rightarrow H_G^*(X)$ is generated by a homogeneous regular sequence $f_1,\ldots,f_k$, then the action is formally based by Lemma \ref{lem:hyperformal is formally based} and we have a morphism of $R$-cdgas $\psi\colon(C,d)\rightarrow (R\otimes \Lambda V,D)$, where $C=R\otimes\Lambda(a_1,\ldots,a_k)$ and $da_i=f_i$. We want to prove that the map $\overline{\psi}\colon\overline{C}\rightarrow\Lambda V$ is cohomologically injective. Observe that if the regular sequence was maximal in $R=\Lambda(X_1,\ldots,X_r)$, meaning $k=r$, then $H^*(C)$ would be a Poincaré duality algebra and we could use $(i)$ to finish the proof.

If $k<r$ we may complete the $f_i$ to a maximal regular sequence and extend $\psi$ to a map $(C\otimes \Lambda(a_{k+1},\ldots,a_r),d) \rightarrow (R\otimes \Lambda V\otimes \Lambda(a_{k+1},\ldots,a_r),D)$ where the differentials map the additional $a_i$ onto the additional $f_i$. Cohomological injectivity of
\[\overline{\psi}\otimes \Id\colon \overline{C}\otimes \Lambda(a_{k+1},\ldots,a_r)\rightarrow \Lambda V\otimes \Lambda(a_{k+1},\ldots,a_r)\] is equivalent to the original map $\overline{\psi}$ being injective. Applying $(i)$ yields the claim.
\end{proof}

\begin{rem}
We want to investigate how all of the previously defined notions interact.
As it turns out, the concepts introduced in this section seem to be rather independent from the notion of $\mathcal{MOD}$-formality: Example \ref{ex:injectively based, not spherical} is a hyperformal action (with formal core), which is not spherical and thus in particular not (almost) $\mathcal{MOD}$-formal.
Furthermore, in \ref{ex:masseynotnecessary}, we construct an example of a $\mathcal{MOD}$-formal action which is not formally based. However, we want to point out that while no direct implications exist, common roots can be found in the vanishing of certain Massey products, which is best captured in the degeneracy of minimal $C_\infty$-structures: in Theorem \ref{MasseyThm} we give a sufficient condition for $\mathcal{MOD}$-formality that builds upon the notion of being formally based (see also Theorem \ref{thm:Masseyformallybased}).
\end{rem}

\begin{rem}
A class of examples that has a formal homotopy quotient, and is thus in particular $\mathcal{MOD}$-formal and has formal core, is provided by $2$-step nilpotent Lie algebras. An associated nilmanifold $M$ has the rational structure of the total space of a torus fibration over a torus, i.e.~
\begin{align*}
T_1\hookrightarrow{} M \to T_2,
\end{align*}
where $T_1$ has maximal dimension among the tori acting almost freely on a space in the rational homotopy type of $M$ (see e.g.\ \cite[Sections 3.2 and 7.3.4]{AMIG}). So, rationally, we are dealing with the equivalent of a free $T_1$-action on $M$ with formal quotient $T_2$. In view of Theorem \ref{thm:TRCholdsforbla} this gives another proof of the TRC for 2-step nilpotent Lie Algebras.
\end{rem}

\subsection{Inheritance under elementary constructions}

\subsubsection{Products} If $X$ is a $G$-space and $Y$ is a $G'$ space, then $X\times Y$ is naturally a $G\times G'$ space.

\begin{prop}\label{prop:productinheritance}
Let $X$ be a $G$-space, $Y$ be a $G'$ space. If both, the $G$- and the $G'$-action, satisfy one of the conditions of being spherical, (almost) $\mathcal{MOD}$-formal, formally based or having formal core, then the same is true for the $G\times G'$-action on $X\times Y$.
\end{prop}

\begin{proof}
Let $R\rightarrow \mathcal{M}_G^X\rightarrow \mathcal{M}^X$ and $R'\rightarrow \mathcal{M}_{G'}^Y\rightarrow \mathcal{M}^Y$ be relative minimal models for the respective Borel fibrations. Then
\[R\otimes R'\rightarrow \mathcal{M}_G^X\otimes \mathcal{M}_{G'}^Y\rightarrow \mathcal{M}^X\otimes \mathcal{M}^Y\]
is a relative minimal model for the Borel fibration of the $G\times G'$-action. The associated Hirsch--Brown model is given by the tensor product of the minimal Hirsch--Brown models of the $G$- and the $G'$-action hence the statement for almost $\mathcal{MOD}$-formal actions follows. Furthermore, formality is compatible with the tensor product and we obtain the proposition under the assumption of $\mathcal{MOD}$-formality. In the case of spherical actions, the kernel of $H_G^*(X)\otimes H_{G'}^*(Y)\rightarrow H^*(X)\otimes H^*(Y)$ is given by \[(R^+\cdot H_G^*(X))\otimes H_{G'}^*(Y)+H_G^*(X)\otimes (R'^+\cdot H_{G'}^*(Y))= (R\otimes R')^+\cdot H_{G\times G'}^*(X\times Y).\]
Finally, if the $G$-action is formally based with respect to $A\subset H_G^*(X)$ and the $G'$-action is formally based with respect to $A'\subset H_{G'}^*(Y)$, then, by considering the tensor product of the occurring maps, we see that the $G\times G'$-action is formally based with respect to $A\otimes A'\subset H_G^*(X)\otimes H_{G'}^*(Y)\cong H_{G\times G'}^*(X\times Y)$. In the same way one obtains the statement for actions with formal core.
\end{proof}

\subsubsection{Gluing}
In this section we assume all $G$-spaces to be Tychonoff spaces (so e.g.\ CW-complexes) in order to ensure the existence of tubular neighbourhoods (see \cite[Theorem 5.4]{bredon}).
Let $X$ and $Y$ be two $G$-spaces. If $G\rightarrow X$ and $G\rightarrow Y$ are equivariant maps onto orbits of $X$ and $Y$, we may construct their pushout $X\vee_G Y$, which is naturally a $G$-space. If the stabilizers of the image of $1\in G$ in $X$ and $Y$ agree, then $X\vee_G Y$ is just $X$ and $Y$ glued together at these orbits.

\begin{prop}\label{prop:orbitgluing}
Suppose that $X$ and $Y$ are $\mathcal{MOD}$-formal (resp.\ have formal homotopy quotients $X_G$ and $Y_G$) and have an almost free $G$-orbit of the same orbit type. Then the $G$-space $X\vee_G Y$ obtained by gluing $X$ and $Y$ along this orbit is $\mathcal{MOD}$-formal (resp.\ has formal homotopy quotient).
\end{prop}

In the proof we make use of the following notation: if $A$ and $B$ are connected cdgas, then $A\oplus_\mathbb{Q}B$ is the sub-cdga of the product cdga $A\oplus B$ which in degree $0$ is generated by $(1,1)$ and agrees with $A\oplus B$ in positive degrees. We use the analogous notation for dg$R$ms whose degree $0$ component is $\mathbb{Q}$.

\begin{proof}
The orbits are equivariant retracts of open $G$-invariant neighbourhoods in $X$ and $Y$. Using these, we can cover $X\vee_G Y$ with open sets $U,V$ which equivariantly retract onto $X$ and $Y$ and whose intersection retracts equivariantly onto the glued orbit. We consider the equivariant Mayer--Vietoris sequence of this cover. As the equivariant cohomology of an almost free orbit is just $\mathbb{Q}$, it follows that \[H^*_G(X\vee_G Y) \cong H^*_G(X)\oplus_\mathbb{Q} H^*_G(Y).\]

We now turn our attention to models. Let $\mathcal{M}_G^{X\vee_G Y}$, $\mathcal{M}_G^X$, and $\mathcal{M}_G^Y$ denote the Sullivan models of the homotopy quotients of the respective $G$-spaces as in Remark \ref{rem:relmodelseindeutig} $(ii)$. The previous discussion implies that the induced map $\mathcal{M}_G^{X\vee_G Y}\rightarrow \mathcal{M}_G^X\oplus_\mathbb{Q}\mathcal{M}_G^Y$ is a quasi-isomorphism. In fact it can be chosen as a morphism of $R$-cdgas, where the $R$-module structure on the right is the diagonal one.
If the actions have formal homotopy quotient, then we have a quasi-isomorphism \[\mathcal{M}_G^X\oplus_\mathbb{Q}\mathcal{M}_G^Y\rightarrow H_G^*(X)\oplus_\mathbb{Q}H^*_G(Y)\]
of cdgas which ends the proof. In case the actions are $\mathcal{MOD}$-formal, let $\mathfrak{M}_X$ and $\mathfrak{M}_Y$ denote the Hirsch--Brown models of the action  and note that by the previous discussion, $\mathfrak{M}_X\oplus_\mathbb{Q} \mathfrak{M}_Y$ is a dg$R$m-model for $\mathcal{M}_G^{X\vee_G Y}$. We have a quasi-isomorphism
\[\mathfrak{M}_G^X\oplus_\mathbb{Q}\mathfrak{M}_G^Y\rightarrow H_G^*(X)\oplus_\mathbb{Q}H^*_G(Y)\]
of dg$R$ms, which finishes the proof.
\end{proof}

\begin{prop}\label{prop:orbitglueingformalcore}
Let $X$ and $Y$ be formally based $G$-spaces. Then $X\vee_G Y$ obtained by gluing $X$ and $Y$ along an almost free orbit of the same orbit type is formally based. Moreover, if the actions have formal core and the Sullivan models for $X_G$ and $Y_G$ in a relative minimal model of the respective Borel fibrations are already minimal, then the action on $X\vee_G Y$ has formal core.
\end{prop}

If $X_G$ and $Y_G$ are nilpotent spaces, then the technical minimality condition on their models in the above proposition is equivalent to the surjectivity of \[\pi_k(X_G)\otimes \mathbb{Q}\rightarrow \pi_k(BG)\otimes \mathbb{Q}\quad\text{and}\quad\pi_k(Y_G)\otimes \mathbb{Q}\rightarrow \pi_k(BG)\otimes \mathbb{Q}\] for $k\geq 2$.
In particular the condition is automatically fulfilled in case $G$ is a torus and $X$ and $Y$ are simply-connected.

\begin{proof}
Suppose the $G$-spaces are formally based with respect to $A\subset H_G^*(X)$ and $A'\subset H_G^*(Y)$. Let $R\rightarrow \mathcal{M}_G^X$ and $R\rightarrow \mathcal{M}_G^Y$ be relative minimal models for the respective Borel fibrations. As in the proof of the previous proposition, we see that $\mathcal{M}_G^X\oplus_\mathbb{Q}\mathcal{M}_G^Y$ and $\mathcal{M}_G^{X\vee_G Y}$ are quasi-isomorphic as $R$-cdgas, where the $R$-module structure of the sum is defined by the diagonal inclusion. Let $R\otimes C$ and $R\otimes C'$ be relative minimal models for $R\rightarrow A$ and $R\rightarrow A'$. Also let $R\rightarrow C''$ be a relative minimal model for $R\rightarrow A\oplus_\mathbb{Q} A'$. We may lift $C''\rightarrow A\oplus_\mathbb{Q} A'$ to $C\oplus_\mathbb{Q} C'$ by applying the lifting lemma of Sullivan algebras to each summand separately.
By assumption we have maps $C\rightarrow \mathcal{M}_G^X$ and $C'\rightarrow \mathcal{M}_G^Y$ inducing the inclusion on cohomology. Piecing everything together we obtain the composition
\[C''\rightarrow C\oplus_\mathbb{Q} C'\rightarrow \mathcal{M}_G^X\oplus_\mathbb{Q} \mathcal{M}_G^Y\]
which induces the inclusion $A\oplus_\mathbb{Q}A'\rightarrow H_G^*(X)\oplus_\mathbb{Q} H_G^*(Y)\cong H_G^*(X\vee_G Y)$ on cohomology. Lifting the morphism to $\mathcal{M}_G^{X\vee_G Y}$ relative to $R$ shows that the action on $X\vee_G Y$ is formally based with respect to $A\oplus_\mathbb{Q} A'$.

Now assume that the actions on $X$ and $Y$ have formal core with respect to $A$ and $A'$. For $X$ (and analogously for $Y$) this means that we can assume the map $(C\otimes S,D)\rightarrow (\mathcal{M}_G^X\otimes S,D)$ between Hirsch extensions to be injective on cohomology, where $D$ maps generators $(s_i)$ of $S=\Lambda (s_i)$ to the generators $X_i$ of $R=\Lambda(X_1,\ldots,X_r)$. In order to show that $X\vee_G Y$ has formal core, it suffices to show that \[((C\oplus_\mathbb{Q} C')\otimes S,D)\rightarrow ((\mathcal{M}_G^X\oplus_\mathbb{Q}\mathcal{M}_G^Y)\otimes S,D)\] is injective on cohomology where $Ds_i=(X_i,X_i)$. Consider the commutative diagram
\[\xymatrix{
(C\oplus_\mathbb{Q} C')\otimes S\ar[r]\ar[d] & (\mathcal{M}_G^X\oplus_\mathbb{Q}\mathcal{M}_G^Y)\otimes S\ar[d]\\
(C\oplus C')\otimes S\ar[r] & (\mathcal{M}_G^X\oplus\mathcal{M}_G^Y)\otimes S
}\]
in which the bottom horizontal map is actually the direct sum of $C\otimes S\rightarrow \mathcal{M}_G^X\otimes S$ and $C'\otimes S\rightarrow \mathcal{M}_G^Y\otimes S$. By assumption this map is injective on cohomology, and it remains to prove that, on cohomology, the top horizontal map is injective on the kernel
\[K=\ker\left( H^*((C\oplus_\mathbb{Q} C')\otimes S)\longrightarrow H^*((C\oplus C')\otimes S)\right)\]
of the morphism induced by the left vertical inclusion.
Observe that the cokernel of this inclusion has a basis represented by a basis of $(1,0)\otimes S$. The differential induced on the cokernel vanishes. The resulting short exact sequence of complexes
\[0\longrightarrow ((C\oplus_\mathbb{Q} C)\otimes S,D) \longrightarrow ((C\oplus C)\otimes S,D)\longrightarrow ((1,0)\otimes S,0)\longrightarrow 0\]
induces a long exact sequence on homology from which we see that
$K$ is represented by $D((1,0)\otimes S)$.

Thus it suffices to show that $D((1,0)\otimes S)$ descends injectively to the cohomology of $(\mathcal{M}_G^X\oplus_\mathbb{Q}\mathcal{M}_G^Y)\otimes S$.
By assumption, $\mathcal{M}_G^X$ and $\mathcal{M}_G^Y$ are actually minimal as Sullivan models. Consequently, $D((\mathcal{M}_G^X\oplus \mathcal{M}_G^Y)^+\otimes S)\subset ((\mathcal{M}_G^X)^+\cdot (\mathcal{M}_G^X)^+\oplus (\mathcal{M}_G^Y)^+\cdot (\mathcal{M}_G^Y)^+)\otimes S)$. The only way to hit an element of $D((1,0)\otimes S)\subset \langle (X_1,0),\ldots, (X_r,0)\rangle\otimes S$ with the differential of $(\mathcal{M}_G^X\oplus_\mathbb{Q}\mathcal{M}_G^Y)\otimes S$ is as the image of some element in $(1,1)\otimes S$. These, however, are never $0$ in the second component which proves that $K$ maps injectively to $H^*((\mathcal{M}_G^X\oplus_\mathbb{Q}\mathcal{M}_G^Y)\otimes S)$.
\end{proof}

Let $M$ be a smooth $G$-manifold with (possibly empty) boundary. By the slice theorem, the orbit of any interior point $x\in M$ has a $G$-invariant tubular neighbourhood which is equivariantly diffeomorphic to
\[V_x=G\times_{G_x} T_x M/T_x(G\cdot x)\]
where $G_x$ acts on the right hand side via the isotropy action and $G\cdot x$ corresponds to $[G,0]\subset V_x$. Now if $N$ is a $G$-manifold with boundary of the same dimension, with an interior point $y\in N$ such that $G_y=G_x$ and $T_x M/T_x(G\cdot x)\cong T_y N/T_y(G\cdot y)$ as representations, then we may form the equivariant connected sum as follows: the isomorphism defines an equivariant diffeomorphism $\varphi\colon V_x\cong V_y$ of tubular neighbourhoods.
Choose some $G_x$-invariant inner product on $T_x M/T_x(G\cdot x)$, and let $S\subset T_x M/T_x(G\cdot x)$ be the associated unit sphere. The equivariant connected sum $M\#_G N$ is the quotient of $(M-\{G\cdot x\})\sqcup (N-\{G\cdot y\})$ obtained by gluing the points $[g,ts]$ and $\varphi([g,(1-t)s])$ for all $g\in G$, $t\in (0,1)$, and $s\in S$. The result is a $G$-manifold with boundary whose homotopy type does however depend on the choice of $\varphi$.

\begin{prop}\label{prop:connectedsummodformal}
Let $M$ and $N$ be almost free $G$-manifolds with boundary such that the equivariant connected sum $M\#_G N$ at some interior orbit is defined.
\begin{enumerate}[(i)]
\item If $M$ and $N$ are $\mathcal{MOD}$-formal, then so is $M\#_G N$.
\item If $M$ and $N$ are simply-connected and $M_G$ and $N_G$ are formal, then also $(M\#_G N)_G$ is formal.
\end{enumerate}
\end{prop}

We want to point out that in the proof of $(ii)$, one of the cases is essentially the proof of the fact that the (nonequivariant) connected sum of compact simply-connected manifolds preserves formality (see e.g.\ \cite[Theorem 3.13]{AMIG}).

\begin{proof}
In the notation of the construction of $M\#_G N$, let $D\subset T_x M/T_x(G\cdot x)$ be the unit disk. The collapsing of $D-\{0\}$ induces an equivariant map $G\times_{G_x} (D-\{0\})\rightarrow G\cdot x$ which induces an equivariant map
\[p\colon M\#_G N\longrightarrow M\vee_G N.\]
Both are almost free $G$-spaces so the map on equivariant cohomology can be determined from the orbit spaces (see Proposition \ref{thm:hsiangandfriends}). There it induces the collapse of the subspace
\[X:=(G\times_{G_x} (D-\{0\}))/G \cong (D-\{0\})/G_x\simeq S/G_x.\]
Thus we may understand the map $p^*\colon H^*_G(M\vee_G N)\rightarrow H^*_G(M\#_G N)$ via the long exact homology sequence
\[\cdots\rightarrow H^k((M\#_G N)/G, X)\rightarrow H^k((M\#_G N)/G)\rightarrow H^k(X)\rightarrow\cdots\]
where we can identify $H^k((M\#_G N)/G, X)\cong H^k(M/G\vee N/G)=H_G^k(M\vee_G N)$ in positive degrees.

The algebra $H^*(X)=H^*(S)^{G_x}$ is isomorphic to either $\mathbb{Q}$ or $H^*(S)$. Let $n=\dim M-\dim G$, so $\dim S=n-1$. There are three possible scenarios. The map $p^*$ is either surjective with $1$-dimensional kernel in degree $n$, injective with $1$-dimensional cokernel in degree $n-1$, or a quasi-isomorphism. In the last case we are done since $M\vee_G N$ is $\mathcal{MOD}$-formal by Proposition \ref{prop:orbitgluing}.
Otherwise consider the map $\tilde{p}\colon \mathfrak{M}_{\vee}\rightarrow \mathfrak{M}_{\#}$ between minimal Hirsch--Brown models of $M\vee_G N$ and $M\#_G N$. As $M\vee_G N$ is $\mathcal{MOD}$-formal, the Hirsch--Brown model takes the form of a free resolution
\[\mathfrak{M}_\vee=\left(\bigoplus_{i\geq 0} F_i,d\right)\]
with $d\colon F_i\rightarrow F_{i-1}$ being exact at every $i\geq 1$. If $p^*$ has nontrivial kernel, we add a generator $\alpha$ of degree $n-1$ to $F_1$ and define $d\alpha\in F_0$ to be a representative for the generator of $\ker p^*$. Also $\tilde{p}(d\alpha)$ is exact in $\mathfrak{M}_\#$ so we may extend $\tilde{p}$ to $\alpha$. At this point, $\tilde{p}$ induces an isomorphism $F_0/d(F_1)\cong H^*_G(M\#_G N)$ but there may now be additional cohomology represented in $F_1$. As the newly introduced generator lives in degree $n-1$ and $R$ is simply-connected, $\ker_d|_{F_1}$ remains unchanged up to degree $n+1$. But the cohomology of $M\#_G N$ vanishes in degrees above $n$ so it follows that $\tilde{p}$ maps the newly introduced cohomology to exact elements. Hence if we introduce new generators in $F_2$ and use them to kill the cohomology generated in $F_1$, then $\tilde{p}$ extends to the new generators. We may repeat this process inductively and obtain a free resolution quasi-isomorphic to $\mathfrak{M}_\#$. Thus $M\#_G N$ is $\mathcal{MOD}$-formal.

If $p^*$ is injective, then we start by adding a generator $\alpha$ to $F_0$ in degree $n-1$ and define $\tilde{p}(\alpha)$ to be a representative of the cokernel of $p^*$. Now add generators of degree $\geq n$ to $F_1$ and map them to a minimal generating set of $\ker(F_0\rightarrow H_G^*(M\#_G N))$. Again $\tilde{p}$ extends to the new $F_1$. We are now in the same position as before and we analogously conclude that $M\#_G N$ is $\mathcal{MOD}$-formal. This proves $(i)$.

The proof of $(ii)$ works by applying the analogous argument to the cdga machinery. Consider the map $\varphi\colon \mathcal{M}_\vee\rightarrow \mathcal{M}_\#$ between the Sullivan minimal models of $(M\vee_G N)_G$ and $(M\#_G N)_G$. As $(M\vee_G N)_G$ is formal by Proposition \ref{prop:orbitgluing}, it has a bigraded minimal model with an additional lower grading
\[\mathcal{M}_\vee=\Lambda \left(\bigoplus_{i\geq 0} V_i\right)\]
and cohomology concentrated in $(\Lambda V)_0$ as in e.g.\ \cite[Theorem 2.93]{AMIG}. We observe that $M\vee_G N$ is simply-connected if $M$ and $N$ are and therefore $\mathcal{M}_\vee^1=0$. The proof now proceeds as before with the role of $F_i$ replaced by $(\Lambda V)_i$. If e.g.\ $p^*$ is surjective with $1$-dimensional kernel in degree $n$ (and necessarily lower degree $0$), then we add a generator to $V_1$ and use it to kill the existing kernel in cohomology. As $\mathcal{M}_\vee^1=0$, no new cohomology is generated up until cohomological degree $n+1$. Now for the same reasons as in the module case, we can extend $\varphi$ to a quasi-isomorphism from a bigraded minimal model in the above sense, showing that $M\#_G N$ is formal. The other cases transfer analogously.
\end{proof}

\begin{prop}
Let $M$ and $N$ be almost free $m$-dimensional $G$-manifolds with boundary such that the equivariant connected sum along an interior orbit is defined. Assume that they have formal core with respect to $A\subset H_G^*(M)$, $A'\subset H_G^*(N)$, with $A^{m-\dim G}=A'^{m-\dim G}=0$, and that the conditions of Proposition \ref{prop:orbitglueingformalcore} are satisfied.
Then $M\#_G N$ has formal core.
\end{prop}

We want to point out that the condition $A^{m-\dim G}=0$ is automatically fulfilled if $M$ has non-empty boundary, is non-compact, or is not orientable: in this case $H^m(M)=0$ whence $H_G^{m-\dim G}(M)=0$ by Proposition \ref{thm:hsiangandfriends}. In case $M$ is closed and orientable, it just means that $A$ is supposed to not contain the fundamental class of $M/G$.

\begin{proof}
Let $n=m-\dim G$ and recall the equivariant map $M\#_G N\rightarrow M\vee_G N$ from the proof of the previous proposition. As we showed there, on cohomology, it is either injective or has $1$-dimensional kernel in degree $n$. Hence if $A$ and $A'$ are trivial in degree $n$, then
\[\psi\colon A\oplus_\mathbb{Q} A'\longrightarrow H_G^*(M\vee_G N)\longrightarrow H_G^*(M\#_G N)\]
is injective.

Let $\mathcal{M}_\vee$ and $\mathcal{M}_\#$ be models for $(M\vee N)_G$ and $(M\#_G N)$ arising from relative minimal models of the respective Borel fibrations and let $\varphi\colon \mathcal{M}_\vee\rightarrow\mathcal{M}_\#$ be a map of $R$-cdgas that is a Sullivan representative of the equivariant map $M\#_G N\rightarrow M\vee_G N$. By Proposition \ref{prop:orbitglueingformalcore}, $M\vee_G N$ has formal core with respect to $A\oplus_\mathbb{Q} A'$, which means we have a map $\phi\colon C''\rightarrow \mathcal{M}_\vee$ of $R$-cdgas as in the proof of \ref{prop:orbitglueingformalcore} ($X$ and $Y$ replaced by $M$ and $N$). Since $\psi$ is injective, we see that $M\#_G N$ is formally based with respect to $\psi(A\oplus_\mathbb{Q} A')$ by considering the composition $\varphi\circ \phi$.

To prove that the action has formal core it remains to see that this map is still cohomologically injective when extending it to the Hirsch extensions
\[C''\otimes S\longrightarrow \mathcal{M}_\vee\otimes S\longrightarrow\mathcal{M}_\#\otimes S\]
in which the differential maps generators of $S=\Lambda(s_i)$ bijectively to generators of $R$. Cohomological injectivity of $\phi\otimes \Id_S$ is part of Proposition \ref{prop:orbitglueingformalcore}, so we only need to prove injectivity of $(\varphi\otimes\Id_S)^*$ on the cohomological image of $\phi\otimes \Id_S$. To see this we consider the maps between the Serre spectral sequences arising by filtering the Hirsch extensions above in the degrees of the left hand cdgas. The second pages are isomorphic to the tensor products of the (non-twisted) cohomologies of the respective factors so by naturality we obtain the maps
\[(A\oplus_\mathbb{Q} A')\otimes S\longrightarrow H_G^*(M\vee_G N)\otimes S\longrightarrow H_G^*(M\#_G N)\otimes S.\]
Again, we distinguish the three possible cases for the map $\varphi^*$. If it is an isomorphism, then so is $\varphi^*\otimes \Id_S$ and we are done. If $\varphi^*$ is injective and its cokernel is generated by some $\alpha\in H^{n-1}(M\#_G N)$, then the map $\varphi^*\otimes \Id_S$ on the second pages is an isomorphism up to the column  $\alpha\otimes S\subset E_2^{n-1,*}$. As the differentials in the spectral sequence vanish on this column for degree reasons, we deduce that $\varphi\otimes \Id_S$ induces an injective map between the $E_\infty$-pages which implies injectivity on cohomology.

Finally, consider the case where $\varphi^*$ is an isomorphism up to $1$-dimensional kernel contained in $ H_G^n(M\vee_G N)$. We claim that the kernel of the map induced by $\varphi\otimes \Id_S$ on the $E_\infty$-pages is contained in $E_\infty^{n,*}$. Since the image of $\phi^*\otimes \Id_S$ on the $E_2$ pages is contained in $E_2^{<n,*}$, and the same degree restrictions carry over to the $E_\infty$ pages, this will imply injectivity of $(\varphi\circ\phi)\otimes \Id_S$ on the $E_\infty$-pages and thus finish the proof of the proposition. The claim can be verified via induction: assume the map between the $r$th pages is injective on $E_r^{<n,*}$ and an isomorphism on $E_r^{\leq n-r,*}$. Then it is a straightforward diagram chase to show that on the $(r+1)$th pages the induced map is injective on $E_{r+1}^{<n,*}$ and an isomorphism on $E_{r+1}^{\leq n-r-1,*}$.

\end{proof}

\begin{prop}
Let $M_1$ and $M_2$ be orientable $G$-manifolds and that there are fixed points $x_1\in M_1$ and $x_2\in M_2$ which have equivariantly diffeomorphic tubular neighbourhoods. Then we may form the $G$-equivariant connected sum at those fixed points, and the induced $G$-action on $M_1\#M_2$ (as well as the one on $M_1 \vee M_2$) has formal core.
\end{prop}
\begin{proof}
Since tubular neighbourhoods are equivariantly diffeomorphic, we may form the equivariant connected sum $M_1\# M_2$ at $x_1$ and $x_2$ to which the $G$-action extends. We have a $G$-equivariant contraction map $p\colon M_1\# M_2\to M_1\vee M_2$ between the connected sum and the one-point union at $x_1\sim x_2$. This morphism induces a morphism of Leray--Serre spectral sequences of associated $G$-Borel fibrations.

The action on $M_1\vee M_2$ has a fixed-point $x_1\sim x_2$. Thus the Borel fibration has a section and $H^*(BG)\hookrightarrow{} H^*_G(M_1\vee M_2)$. (Hence the action on the one-point union is of formal core.)

Since the cohomology $H^*(M_1\# M_2)$ is a quotient of $H^*(M_1 \vee M_2)$ (by gluing volume forms), i.e.~the corresponding $E_2$-term is a quotient of the one-point union and all differentials are the induced ones, we deduce that $H^*(B G)$ also injects into $H_G^*(M_1 \# M_2)$. This yields the result.
\end{proof}

\subsubsection{Subgroups}
We investigate how the previously defined notions behave under restriction of the action to subgroups.
As it turns out, problems arise
when restricting to subgroups of smaller rank. The only one of the discussed concepts which behaves well under restriction to arbitrary subgroups is the classical equivariant formality. Example \ref{ex:restriction} is an action with formal homotopy quotient such that the restriction to a certain subgroup is neither formally based nor spherical. However, we have the following

\begin{prop}
A $G$-action fulfils one of the conditions of being spherical, almost $\mathcal{MOD}$-formal, or $\mathcal{MOD}$-formal if and only if the respective condition is fulfilled by the action of a maximal subtorus.
\end{prop}

\begin{proof}
Let $G$ act on $X$ and let $T$ be a maximal torus of $G$. The central observation needed for the proof is the fact that the Borel fibration of the $T$-action is the pullback of the Borel fibration of the $G$-action along the map $BT\rightarrow BG$. We denote the minimal models of $BT$ and $BG$ by $R$ and $S$. Now if
\[(S,0)\rightarrow (S\otimes \Lambda V,D)\rightarrow (\Lambda V,d)\]
is a minimal model for the fibration of the $G$-action, a minimal model for the pullback fibration is given by
\[(R,0)\rightarrow (R\otimes_S(S\otimes\Lambda V),\Id_R\otimes_S D)\rightarrow (\Lambda V,d)\]
(see \cite[Prop.\ 15.8]{Bibel}). Since $T$ is maximal, the map $S\rightarrow R$ turns $R$ into a finitely generated, free $S$-module. In particular, $R\otimes_S(S\otimes\Lambda V)\cong R\otimes \Lambda V$ splits as a sum  of multiple degree shifted copies of $S\otimes \Lambda V$. This decomposition is respected by the differential so it induces an analogous splitting of the cohomology. We obtain $H_T^*(X)=R\otimes_S H_G^*(X)$ (actually as algebras although the splitting is one of $S$-modules).

As tensoring with $R$ over $S$ is exact, we deduce that the minimal free resolution of $H_T^*(X)$ is obtained from the one of $H^*_G(X)$ by tensoring with $R$. This implies that the $G$-action is almost $\mathcal{MOD}$-formal if and only if this holds for the $T$-action.

In view of the statement for spherical actions, this also shows that a minimal generating set of $H_G^*(X)$, i.e.\ one which descends to a basis of $H_G^*(X)/S^+\cdot H_G^*(X)$, is also a minimal generating set of $H_T^*(X)$. By Lemma \ref{spherical-generator-lem}, the condition of being spherical is equivalent to the restriction to $H^*(X)$ being injective on the span of such a generating set. The claim now follows from the observation that the inclusion \[H_G^*(X)=S\otimes_S H_G^*(X)\subset R\otimes_S H_G^*(X)=H_T^*(X)\]
commutes with the restriction to $H^*(X)$.

We turn our attention to $\mathcal{MOD}$-formal actions. Let $(\mathfrak{M},\tilde{D})\rightarrow (S\otimes\Lambda V,D)$ be the minimal Hirsch--Brown model of the $G$-action, with $\mathfrak{M}=S\otimes H^*(X)$. By the previous discussion it follows that \[(R\otimes H^*(X),\tilde{D})\cong (R\otimes_S \mathfrak{M},\Id_R\otimes_S \tilde{D})\rightarrow (R\otimes_S (S\otimes \Lambda V),D)\] induces a quasi-isomorphism. Note that the induced differential on $R\otimes_S \mathfrak{M}$ satisfies the minimality condition so this is indeed the minimal Hirsch--Brown model of the $T$-action.

A quasi-isomorphism $(\mathfrak{M},\tilde{D})\rightarrow (H_G^*(X),0)$ induces a quasi-isomorphism \[(R\otimes_S \mathfrak{M},\tilde{D})\rightarrow (R\otimes_S H_G^*(X),0).\]
Thus $\mathcal{MOD}$-formality of the $G$-action implies $\mathcal{MOD}$-formality of the $T$-action. For the converse implication we use criterion $(iii)$ of Lemma \ref{formality criterion} by which we obtain a vector space splitting $R\otimes_S \mathfrak{M}=\ker \tilde{D}\oplus C$ such that $C\oplus\im \tilde{D}$ is an $R$-submodule. As argued above, $R\otimes_S \mathfrak{M}$ splits as the sum of multiple degree shifted copies of $\mathfrak{M}$ when regarded as a differential graded $S$-module. We identify $\mathfrak{M}$ with the summand $S\otimes_S \mathfrak{M}\subset R\otimes_S \mathfrak{M}$. Denote by $\pi\colon R\otimes_S \mathfrak{M}\rightarrow \mathfrak{M}$ the projection onto this summand. The differential restricts to an isomorphism $\tilde{D}\colon C\rightarrow\im \tilde{D}$ and we denote its inverse by $\tilde{D}^{-1}$. Set \[C'=\pi\circ \tilde{D}^{-1}(\mathfrak{M}\cap\im\tilde{D}).\]
We claim that $C'$ fulfils the requirements of Lemma \ref{formality criterion} $(iii)$ with respect to the $G$-action. We observe that in the commutative diagram

\[\xymatrix{ \tilde{D}^{-1}(\mathfrak{M}\cap\im\tilde{D}\ar[d]^\pi\ar[rd]^{\tilde{D}}) & \\ C' \ar[r]^{\tilde{D}} & \mathfrak{M}\cap \im \tilde{D}}\]
all maps are isomorphisms. It follows that $C'$ is a complement of $\ker \tilde{D}|_\mathfrak{M}$ in $\mathfrak{M}$ and it remains to show that a closed $S$-linear combination
$\sum s_i c_i'$
of elements in $C'$ is already exact in $\mathfrak{M}$. As $\tilde{D}|_C$ is an isomorphism onto $\im \tilde{D}$, there are unique elements $c_i\in C$ with $\tilde{D}c_i=\tilde{D}c_i'$. They fulfil $\pi(c_i)=c_i'$ and have closed elements in all the other components with respect to the decomposition of $R\otimes_S \mathfrak{M}$ (because the differential respects the decomposition). Consequently, the element $\sum s_i c_i$ is closed. By the choice of $C$ it follows that it is already exact which is equivalent to exactness in every component. In particular, $\sum s_i c_i'$ is exact in $\mathfrak{M}$, which proves the claim.
\end{proof}

\begin{prop}
If the $G$-action is formally based or has formal core, then the same holds for the action of its maximal torus.
\end{prop}

\begin{proof}
Let $S,R$ as above and assume the $G$-action is formally based with respect to some $A\subset H_G^*(X)$. Let $(C,d)\simeq (A,0)$ be a relative minimal model of the canonical morphism $(S,0)\rightarrow(A,0)$. The map $R\rightarrow H_T^*(X)$ corresponds to $R\otimes_S S\rightarrow R\otimes_S H_G^*(X)$ which has image $R\otimes_S A$, and the induced map $(R\otimes_S C,d)\simeq (R\otimes_S A,0)$ is a relative minimal model for $(R\otimes_S S,0)\rightarrow (R\otimes_S A,0)$. Clearly, a morphism $(C,d)\rightarrow (S\otimes \Lambda V,D)$ of $S$-cdgas induces a morphism $(R\otimes _S C,d)\rightarrow (R\otimes_S (S\otimes \Lambda V),D)$ of $R$-cdgas. Thus the $T$-action is formally based if the $G$-action is.

If the $G$-action has formal core with respect to $A$, then we may take $C/S^+\rightarrow \Lambda V$ to be cohomologically injective. It factors through the morphism $C/S^+\rightarrow (R\otimes_S C)/R^+$, which is an isomorphism. This implies that also $(R\otimes_S C)/R^+\rightarrow \Lambda V$ is cohomologically injective. Consequently, the $T$-action has formal core as well.
\end{proof}

\section{Higher $A_\infty$-operations}\label{Massey}

As established earlier (see Prop. \ref{prop:realizeasaction} and before), fixing a base space $Y$, there is a correspondence between free torus actions with orbit space $Y$ and degree $2$ cohomology classes of $Y$. This correspondence is one-to-one in a suitable rational sense so it is a natural question how the formality properties of those actions are encoded in the corresponding cohomology classes. The usual algebra structure on the cohomology is not sufficient for answering this kind of question. Instead, in this section we attack the problem via certain higher operations on the cohomology.

We consider, more generally, any $G$ action on $X$. As before, let $R\otimes\Lambda V$ be a Sullivan model for $X_G$. Then we can consider its minimal $C_\infty$-model $(H_G^*(X);0,m_2,m_3,\ldots)$ which is unique up to isomorphism of $C_\infty$-algebras (see Section \ref{sec:minmodsec}). It is known that $X_G$ is formal if and only if it admits a $C_\infty$-model of the form $(H_G^*(X),0,m_2,0,\ldots)$ where all higher operations vanish (see Theorem \ref{thm:formalityofcdga}). Thus there is a characterization of actions with formal homotopy quotient in terms of the higher operation on the equivariant cohomology. Our goal is to find something similar for $\mathcal{MOD}$-formal actions.

\begin{thm}\label{MasseyThm}
Let $A\subset H^*_G(X)$ be the image of $H^*(BG)\rightarrow H_G^*(X)$. If the unital minimal $C_\infty$-model $(H^*_G(X);m_i)$ can be chosen in a way that $m_i$ vanishes on the subspace $H_G^*(X)\otimes A^{\otimes i-1}$ for $i\geq 3$, then the action is $\mathcal{MOD}$-formal.
\end{thm}

\begin{rem}\label{Cinftyexampleschmiede}
The above theorem is particularly useful for constructing free $\mathcal{MOD}$-formal torus actions (see Example \ref{ex:nontrivmodform}): isomorphism classes of simply-connected, minimal, unital $C_\infty$-models are in one-to-one correspondence with simply-connected rational homotopy types. Hence, starting with any finite-dimensional, simply-connected, minimal, unital $C_\infty$-algebra $(H,m_i)$, we find a finite $CW$-complex $Y$ with minimal $C_\infty$-model $(H;m_i)$. Now any choice of $r$ elements in $H^2(Y)=H^2$ defines (the rational homotopy type of) a free $T^r$-space with orbit space $M$ (see Prop.\ \ref{prop:realizeasaction}). If we choose the degree $2$ classes in a way that their spanned subalgebra $A\subset H$ fulfils $m_i(x,a_1,\ldots,a_{i-1})=0$ for $a_i\in A$, $i\geq 3$, then the corresponding action will be $\mathcal{MOD}$-formal.
\end{rem}
\begin{proof}
Let
\[(R,0)\rightarrow (R\otimes \Lambda V,D)\rightarrow (\Lambda V,d)\] be a Sullivan minimal model of the Borel fibration of the action. Suppose there is a minimal unital $C_\infty$-model $\varphi\colon (H_G^*(X);m_i)\rightarrow (R\otimes \Lambda V,D)$ satisfying the properties from the theorem. Then the canonical map $f_1\colon R\rightarrow H_G^*(X)$ can be extended to a $C_\infty$-morphism $f\colon (R,0)\rightarrow (H_G^*(X);m_i)$ by setting the higher components to be trivial. This yields a diagram
\[\xymatrix{(R,0)\ar[r]\ar[dr]^f& (R\otimes \Lambda V,D)\\ & (H_G^*(X);m_i)\ar[u]_\varphi}\]
of augmented $C_\infty$-algebras (see Remark \ref{rem:augmented}) which commutes on the level of cohomology. We claim that it commutes up to homotopy of $A_\infty$-algebras.

The inclusion functor $\mathrm{cdga}^+\rightarrow \mathcal{C}_\infty\text{-}\mathrm{alg}^+$ between the augmented cdgas and augmented $C_\infty$-algebras induces an equivalence between the homotopy categories
$\mathrm{Ho}(\mathrm{cdga}^+)$ and $ \mathrm{Ho}(\mathcal{C}_\infty\text{-}\mathrm{alg}^+)$, which are the localizations of the respective categories at the quasi-isomorphisms (see Theorem \ref{thm:homotopycategories}). Also, since $(R,0)$ and $(R\otimes \Lambda V,D)$ are both free cdgas, they are both fibrant and cofibrant with respect to a model category structure whose weak equivalences are the quasi-isomorphisms (see e.g.\ \cite[Section B.6.11]{LodayValette}). Through the equivalence
\[\mathrm{cdga}^+_{cf}/\sim\rightarrow\mathrm{Ho}(\mathrm{cdga}^+)\] ($\sim$ being the homotopy relation), we deduce that the equivalence class of the morphism $\varphi\circ f$ in $\mathrm{Ho}(\mathcal{C}_\infty\text{-}\mathrm{alg}^+)$ contains a (unital) cdga-morphism $\psi\colon (R,0)\rightarrow (R\otimes \Lambda V,D)$. This is not necessarily the standard inclusion, which we denote by $i$, but it induces the same map in cohomology. By Lemma \ref{lem:fbchar}, this already implies that $i$ and $\psi$ are homotopic.
It follows that $i$ and $\varphi\circ f$ give rise to the same morphism in $\mathrm{Ho}(C_\infty\text{-}\mathrm{alg})$, which implies they are in particular homotopic when considered as $A_\infty$-morphisms through the forgetful functor (see \cite[Cor. 1.3.1.3]{lefevre}).

Now by Lemma \ref{lem:homotopicrestriction} the two $A_\infty$-$R$-module structures on $R\otimes\Lambda V$ defined by the morphisms $i$ and $\varphi\circ f$ are quasi-isomorphic as $A_\infty$-$R$-modules and by Lemma \ref{lem:trianglemorphism} they are also quasi-isomorphic to the $A_\infty$-$R$-module structure on $H_G^*(X)$ defined by $f$. As the higher operations of $H_G^*(X)$ vanish on the image of $f$ by assumption, all but the binary operation of the $A_\infty$-$R$-module $H_G^*(X)$ vanish. Thus the latter is just the differential graded $R$-module $(H_G^*(X),0)$.

We have shown that the differential graded $R$-modules $(H^*_G(X),0)$ and $(R\otimes \Lambda V,D)$ are quasi-isomorphic as $A_\infty$-$R$-modules. But then Theorem \ref{thm:homotopycategories} implies that they are also quasi-isomorphic as ordinary differential graded $R$-modules.
\end{proof}

By the formal cohomogeneity of a $G$-action on $X$ we mean the difference $\mathrm{fd}(X)-\dim G$ of the formal dimensions of $X$ and $G$, where formal dimension is the highest degree in which nontrivial cohomology exists.

\begin{cor}\label{cor:modfomalsmallcodim}
Let $G$ act almost freely on $X$. Let $c$ be the formal cohomogeneity of the action and assume that one of the following holds:

\begin{enumerate}[(i)]
\item $c\leq 3$.
\item $G$ is semisimple, $X$ is $k$-connected for $0\leq k\leq 3$ and $c\leq 7+k$.
\end{enumerate}
Then the action is $\mathcal{MOD}$-formal.
\end{cor}

\begin{proof} It follows from Theorem \ref{thm:hsiangandfriends} that $H_G^*(X)$ vanishes in degrees above the codimension $c$. We argue that in the situation of $(i)$ and $(ii)$, the conditions of Theorem \ref{MasseyThm} are fulfilled for degree reasons.
Choose a unital minimal $C_\infty$-model structure on $H_G^*(X)$, which means that the higher operations $m_i$, $i\geq 3$ vanish if the argument has a tensor component of degree $0$. Thus we only need to check the vanishing of the $m_i$ on $H_G^+(X)\otimes A^+\otimes\ldots\otimes A^+$ where $A$ is the image of $H^*(BG)\rightarrow H_G^*(X)$. In the situation of $(i)$, the minimal nonzero and nontrivial degree of $A^+$ is at least $2$. Hence $m_i$, which is of degree $2-i$, takes values in degrees $\geq i+1$ when restricted to this subspace. This proves $(i)$.

If $G$ is semisimple, then the first nontrivial degree of $A^+$ is $4$. If $X$ is $k$-connected, $0\leq k\leq 3$, so is $X_G$ and it follows that $m_i$ takes values in degrees $\geq k+3i-1$ when restricted to $H_G^+(X)\otimes A^+\otimes\ldots\otimes A^+$. This implies $(ii)$.
\end{proof}

\begin{rem}
Instead of arguing via minimal $C_\infty$-models, the corollary above could also be deduced from analogous degree considerations in the minimal $A_\infty$-$R$-module model without the detour through algebras made in Theorem \ref{MasseyThm}.
\end{rem}

The precise nature of the connection between Massey products of algebras and $\mathcal{MOD}$-formality is hard to grasp and the sufficient condition of Theorem \ref{MasseyThm} is not necessary as shown by example \ref{ex:masseynotnecessary}: in the example, the action is $\mathcal{MOD}$-formal despite the existence of nontrivial quadruple Massey products which cause $m_4$ to be nontrivial on $A^{\otimes 4}$ for any $C_\infty$-model structure on the equivariant cohomology. We want to add that contrary to this observation, the nontriviality of certain quadruple Massey products in $A^4$ can be an obstruction to $\mathcal{MOD}$-formality in the right situation.

Other than $\mathcal{MOD}$-formality, the notion of being formally based has a precise description via higher $C_\infty$-operations and is equivalent to a weakened form of the requirement of Theorem \ref{MasseyThm}.

\begin{thm}\label{thm:Masseyformallybased}
Let $A\subset H_G^*(X)$ be an $R$-subalgebra. The action is formally based with respect to $A$ if and only if the unital minimal $C_\infty$-model $(H_G^*(X);m_i)$
can be chosen in a way that $m_i$ vanishes on $A^{\otimes i}$ for $i\geq 3$.
\end{thm}

\begin{proof}
Suppose that the action is formally based which means we have a morphism \[\varphi\colon (C,d)\rightarrow (R\otimes \Lambda V,D)\] of $R$-cdgas, where $(C,d)$ is a relative minimal model for $R\rightarrow A$. Then by part $(iii)$ of Lemma \ref{lem:inftyminmodkram}, we can construct the unital $C_\infty$-model of $X_G$ in the desired way.

Conversely suppose we have a unital minimal model $(H_G^*(X);m_i^{X_G})$ where the $m_i^{X_G}$ vanish on $A$. The cdga $(C,d)$ is formal with cohomology equal to $A$ so by part $(ii)$ of Lemma \ref{lem:inftyminmodkram} we can choose a unital minimal $C_\infty$-model of the form $(A;m_i^A)$ with $m_i^A=0$ for $i\neq 2$ and $m_2^A$ the ordinary multiplication.
The inclusion $(A;m_i^A)\rightarrow (H_G^*(X);m_i^{X_G})$ defines a unital morphism of $C_\infty$-algebras with trivial higher components. Thus we have a morphism $(C,d)\rightarrow (R\otimes \Lambda V,D)$ in $\mathrm{Ho}(C_\infty\text{-alg}^+)$ defined by \[(C,d)\leftarrow (A;m_i^A)\rightarrow (H_G^*(X);m_i^{X_G})\rightarrow (R\otimes \Lambda V,D).\] We observe that $(C,d)$ and $(R\otimes \Lambda V,D)$ are Sullivan cdgas and conclude as in the proof of Theorem \ref{MasseyThm} that the morphism in $\mathrm{Ho}(C_\infty\text{-alg}^+)$ is represented by a morphism $(C,d)\rightarrow (R\otimes \Lambda V,D)$ of cdgas. Then it is also represented by a morphism of $R$-cdgas by Lemma \ref{lem:fbchar}.

\end{proof}

\section{Formality and the TRC}\label{secform}

We investigate the effects of certain aspects of formality with regards to the TRC. Those aspects can be divided into two categories: the equivariant ones that manifest on the homotopy quotient $X_T$ and, on the other hand, formality properties of the space $X$ itself. The considerations of the previous sections fall into the first catergory and will be discussed first, followed by a discussion of actions on spaces of low dimensions. Formality of the space $X$ does not seem to yield immediate results when attacking the TRC in full generality. However, in the third part of this section we recall how to deduce TRC in case $X$ is formal and elliptic. This has been noted previously in \cite{KotaniYamaguchi}, and we do not claim originality of the result. However, it essentially builds upon  a more general structural observation on formal elliptic spaces, which we did not find explicitly stated. We feel like there is some value to a compact display of the material and this seems like a fitting place to do so.

\subsection{The TRC and refinements of equivariant formality}

Formality can help link the Buchsbaum--Eisenbud--Horrocks Conjecture {(see below)} and the Toral Rank Conjecture. For example it was observed in \cite{ustinovsky} (see also \cite{munoz}) that such a link is given by the fact that the Serre spectral sequence of the homotopy fibration
\[T \rightarrow X\rightarrow X_T\]
collapses at $E_3$ if $X_T$ is a formal space. In our language this comes down to the fact that those actions are in particular almost $\mathcal{MOD}$-formal (see Prop.\ \ref{E3collapse}).
Rather recently, in \cite{walker}, the following breakthrough theorem was proved, solving a weak form of the Buchsbaum--Eisenbud--Horrocks conjecture.

\begin{thm}\label{thm:BEHC}
Let $R$ be a commutative Noetherian ring that is locally a complete intersection such that $spec(R)$ is connected. Further, let $M$ be a nonzero finitely generated $R$-module of finite projective dimension such that $M$ is $2$-torsion free and
\[0\leftarrow M\leftarrow P_0\leftarrow\ldots\leftarrow P_d\leftarrow 0\]
a projective resolution. Then
\[\sum_{i=0}^d \rk_R(P_i)\geq 2^c,\]
where $c$ is the codimension of $M$.
\end{thm}

In its strong incarnation one conjectures the more specific bounds \[{\rk_R(P_i)\geq}  {{c}\choose{i}}\]
to hold.
The link to the TRC is provided by the equivariant cohomology: for a torus $T$ acting on a space $X$, take $R=H^*(BT)$ and $M=H_T^*(X)$. Then $R$ is a polynomial ring and thus regular which implies all the conditions of the above theorem. Also $H_T^*(X)$ is finitely generated (see \cite[Prop. 3.10.1]{AlldayPuppe2}) and of finite projective dimension. The codimension $c$ of $H_T^*(X)$ has a nice geometrical interpretation:

\begin{lem}\label{lem:codimisminimalorbitsdim}
Let $X$ be a compact $T$-space. Then the codimension of $H_T^*(X)$ as an $R$-module is the minimal dimension among the orbits.
\end{lem}

\begin{rem}\label{rem:voraussetzungen}
The lemma is a consequence of Borel localization and has been observed e.g.\ in \cite[Proposition 5.1]{FranzPuppe}. We will give a short proof nonetheless for the sake of a compact presentation of this essential ingredient and to provide clarity on the necessary topological requirements. Those stem solely from Borel localization which states that for a multiplicative subset $S\subset R$, there is an isomorphism of localizations
\[S^{-1}H_T(X)\longrightarrow S^{-1}H_T(X^S),\]
where $X^S=\{x\in X~|~S^{-1}H_T(T\cdot x)\neq 0\}$. In fact we will only be concerned with the case $X^S=\emptyset$ in which case Borel localization has an easy proof which works for singular cohomology given the existence of tubular neighbourhoods (which is assured by the the fact that compact Hausdorff spaces are Tychonoff) see e.g.\ \cite[Theorem III.1]{hsiang}. There are other conditions under which Borel localization is known to hold, see \cite[Section 3.2]{AlldayPuppe2}. In particular, the compactness condition can be replaced by other finiteness conditions. Note that the latter reference uses Alexander-Spanier cohomology which agrees with singular cohomology under the additional condition of $X$ being locally contractible.
\end{rem}

\begin{proof}[Proof of Lemma \ref{lem:codimisminimalorbitsdim}]
If $X$ has an orbit $O=T/H$ of dimension $c$, then $\ker(R\rightarrow H_T^*(O)=H^*(BH))$ is generated by $c$ linearly independent generators of $R^2$ and is thus an ideal of height $c$. Since this map factors through $H_T^*(X)$, it follows that $\mathrm{Ann}(H_T^*(X))=\ker(R\rightarrow H_T^*(X))$ is contained in this ideal and is therefore of height $\leq c$.

Assume now that $c$ is the minimal dimension among the orbits and let $\mathfrak{p}$ be a prime ideal of height $c-1$. By the previous considerations, $\ker(R\rightarrow H_T^*(O))$ is not contained in $\mathfrak{p}$ for any orbit $O\subset X$. Hence, setting $S=R\backslash \mathfrak{p}$, we have $X^S=0$. Borel localization implies $S^{-1} H_T^*(X)=0$ which means that every element from $H_T^*(X)$ is annihilated by some element in $S$. Since $H_T^*(X)$ is finitely generated, we obtain an element in $S\cap \mathrm{Ann}(H^*_T(X))$ by taking products. We have shown that $\mathrm{Ann}(H^*_T(X))$ is not contained in $\mathfrak{p}$.
\end{proof}

Thus the only missing piece is linking the projective resolution of $H_T^*(X)$ to $H^*(X)$. The number $\sum\rk_R(P_i)$ from the above theorem gives an upper bound for $\dim H^*(X)$, which is not sharp in general as it is the dimension of the $E_2$ page in Remark \ref{rem:EilenbergMoore}. Equality holds if and only if said spectral sequences collapse at $E_2$.

\begin{thm}\label{thm:TRCholdsforbla}
Suppose the $T$-action on the compact space $X$ is (almost) $\mathcal{MOD}$-formal or has formal core. Then
\[\dim H^*(X)\geq 2^{c},\]
where $c$ is the minimal dimension among the orbits. In particular the TRC holds for those kinds of actions.
\end{thm}

\begin{proof}
The rank of the minimal Hirsch--Brown model is precisely $\dim H^*(X)$ so the statement for (almost) $\mathcal{MOD}$-formal actions follows directly from Theorems \ref{thm:modformalfreeres} and \ref{thm:BEHC}.

Regarding actions with formal core with respect to some $A\subset H^*_G(X)$, note first that $\dim H^*(X) \geq \dim H^*(\overline{C},\overline{d})$ where we use the notation surrounding Definitions \ref{def:formallybased} and \ref{def:injformbased}. The map
$(R,0)\rightarrow (C,d)$ turns $(C,d)$ into a formal dg$R$m so by the same arguments as in the $\mathcal{MOD}$-formal case we obtain $\dim H^*(\overline{C},\overline{d})\geq 2^{c'}$, where $c'$ is the height of the annihilator of $H^*(C)$ as an $R$-module.
But the annihilators of $H_T^*(X)$ and $H^*(C,d)$ are just given by the kernel of the map $R\rightarrow H^*(C,d)\subset H^*_T(X)$. In particular, $c=c'$ and
\begin{align*}
 \dim H^*(X)\geq \dim H^*(\overline{C},\overline{d}) \geq 2^{c}.
\end{align*}
When $T$ acts almost freely, we have $c=\dim T$ which yields the TRC.
\end{proof}

\begin{rem} In the above proof we did essentially only use the Hirsch--Brown model of the action, disregarding much of the rich multiplicative structure encoded in the rational homotopy types of $X_T$ and $X$. This reduction of information leads to much easier models and is the key to the success of the argument. However, this strategy can not be successful when attacking the TRC in full generality: in \cite{IyengarWalker} an example of a finitely generated free differential graded $R$-module (with $R$ a polynomial ring in $r$ variables) is given whose total rank is less than $2^r$ but which has finite dimensional (nonzero) cohomology. This shows that in general the information needed to prove the TRC is not contained in the Hirsch--Brown model. The condition of $\mathcal{MOD}$-formality provides a natural setting in which the Hirsch--Brown model is indeed sufficient for the proof of the TRC. Note that, while this new counterexample raises some doubts, it is not a counterexample to the TRC as not every free dg$R$m comes from an $R$-cdga and this example in particular is shown to be not topologically realizable.
\end{rem}

Also, there is the following addendum to Theorem \ref{thm:BEHC} from \cite{walker}.

\begin{thm}\label{thm:BEHCequality}
Suppose $R$ is a local (Noetherian, commutative) ring of Krull dimension $d$ which is the quotient of a regular local ring by a regular sequence. Assume further that $2$ is invertible in $R$ and let $M$ be a finitely generated $R$-module of finite projective dimension and finite length. If the sum of the Betti numbers of $M$ is $2^d$ then $M$ is the quotient of $R$ by a regular sequence of $d$ elements.
\end{thm}

For torus actions, we deduce the following

\begin{prop}\label{prop:eq}
Suppose a $T$-action on a compact space $X$ is almost $\mathcal{MOD}$-formal or has formal core and fulfils
\[\dim H^*(X)= 2^{c},\]
where $c$ is the minimal dimension among the orbits.
Then $X$ is rationally equivalent to a product of $c$ odd-dimensional spheres.
\end{prop}

\begin{proof}
If the action is formally based with respect to some $A\subset H_G^*(X)$, observe that, in the notation surrounding Definitions \ref{def:formallybased} and \ref{def:injformbased}, $\dim H^*(\overline{C},\overline{d})\geq 2^c$. As it cohomologically injects into $H^*(X)$, we deduce that $(\overline{C},\overline{d})$ is a model for $X$. Consequently, $(C,d)$ is a model for $X_T$ and the action is $\mathcal{MOD}$-formal.

For an almost $\mathcal{MOD}$-formal action, the total rank of the minimal graded free resolution of $H_T^*(X)$ as an $R$-module is $\dim H^*(X)=2^c$. Let $\mathfrak{p}$ be a minimal prime containing $\mathrm{Ann}(H_T^*(X))$. Then $H_T^*(X)_\mathfrak{p}$ has finite length. Since localization is exact, we obtain a free resolution of $H_T^*(X)_\mathfrak{p}$ by localizing the minimal graded free resolution of $H_T^*(X)$. The codimension of $H_T^*(X)_\mathfrak{p}$ is also $c$ and thus this resolution has to be minimal by Theorem \ref{thm:BEHC}. This means that the sum of the Betti numbers of $H_T^*(X)_\mathfrak{p}$ is equal to $2^c$.

We may now apply Theorem \ref{thm:BEHCequality} and conclude that $H_T^*(X)_\mathfrak{p}$ is a quotient of $R_\mathfrak{p}$ by a regular sequence of $c$ elements. Since the minimal free resolution of $H_T^*(X)_\mathfrak{p}$ was constructed from the one of $H_T^*(X)$, we conclude that also $H_T^*(X)$ is a quotient of $R$ by $c$ elements. Those elements span $\mathrm{Ann}(H_T^*(X))$ which is of height $c$, so it follows that they also form a regular sequence in $R$. To see this formally, note that a sequence of homogeneous elements of positive degree is regular in $R$ if and only if it is regular in $R_\mathfrak{m}$, where $\mathfrak{m}=R^+$. After localizing at $R^+$ we may use \cite[Corollary 17.7]{eisenbud} combined with the fact that $R$ is a Cohen--Macaulay ring.

We observe that, in this special case, the $R$-module structure determines also the algebra structure on $H^*_T(X)$. A quotient of a polynomial ring by a regular sequence is intrinsically formal (see \cite[Remark 3.1]{FelixHalperin}) so a Sullivan model for $X_T$ is given by the Koszul complex $(R\otimes \Lambda Z,D)$ where $D$ maps a basis of $Z$ to the regular sequence. It follows that $(\Lambda Z,0)$ is a model for $X$.

\end{proof}

\subsection{Small dimensions}

As we have shown in Corollary \ref{cor:modfomalsmallcodim}, actions of small enough codimension are $\mathcal{MOD}$-formal and hence fulfil the TRC by Theorem \ref{thm:TRCholdsforbla}. We can achieve stronger results by placing additional topological restrictions on $X$. Recall that by the formal cohomogeneity of a $G$-action on $X$ we mean the number $\mathrm{fd}(X)-\dim G$.

\begin{lem}\label{lem:smallcodims}
Let $G$ act almost freely on $X$ with formal cohomogeneity $c$ and assume one of the following holds
\begin{enumerate}[(i)]
\item $X$ is simply-connected and $c\leq 4$.
\item $X$ is simply-connected, satisfies Poincaré duality, and $c\leq 2k$, where $k\geq 3$ is the minimal odd degree such that $\pi_k(X)\otimes\mathbb{Q}\neq 0$.
\end{enumerate}
Then $X_G$ is formal.
\end{lem}

\begin{proof}
By Proposition \ref{thm:hsiangandfriends} we have $\mathrm{fd}(X_G)=c$. If $X$ is simply-connected, so is $X_G$. Any simply-connected space of formal dimension $\leq 4$ is formal as higher operations in the minimal unital $C_\infty$-model vanish for degree reasons. This proves the lemma under condition $(i)$.

In the situation of $(ii)$, $X_G$ is also a Poincaré duality space by Proposition \ref{thm:hsiangandfriends}. The Sullivan minimal model of $X$ does not have an odd degree generator up until degree $k$. As the rational homotopy of $BG$ is concentrated in even degrees, we deduce that the minimal model of $X_G$ also does not have odd generators of degree $<k$. This implies that the differential in the minimal model of $X_G$ vanishes in degrees from $1$ up to $k-1$. Thus $X_G$ is $(k-1)$-formal and by \cite[Theorem 3.1]{FernandezMunoz}, $X_G$ is formal because $\mathrm{fd}(X_G)\leq 2k$.
\end{proof}

Not only do we know that the TRC holds for small cohomogeneities but it is also a classical result that it holds for actions of $T^r$ if $r\leq 3$ (\cite[Theorem 4.4.3]{AlldayPuppe2}). Together with Corollary \ref{cor:modfomalsmallcodim}, Lemma \ref{lem:smallcodims}, and Theorem \ref{thm:TRCholdsforbla} this yields

\begin{thm}\label{thm:smallcodims}
The toral rank conjecture holds for
\begin{enumerate}[(i)]
\item spaces of formal dimension $\leq 7$.
\item simply-connected spaces of formal dimension $\leq 8$.
\item simply-connected Poincaré duality spaces of formal dimension $\leq 2k+4$, where $k\geq 3$ is the minimal odd degree such that $\pi_k(X)\otimes \mathbb{Q}\neq 0$.
\end{enumerate}
\end{thm}

Case $(iii)$ does in particular imply the TRC for simply-connected orientable manifolds of dimension $\leq 10$. This was proved earlier in \cite[Théorème A]{hilali}, also using formality of the homotopy quotient but concluding differently.

\subsection{The TRC for formal elliptic spaces}

Recall that an elliptic space of positive Euler characteristic, i.e.~a \emph{positively elliptic space}, has a minimal model given by a pure algebra $(\Lambda V,d)$ such that $d$ maps a basis of $V^{odd}$ to a maximal regular sequence of $V^{even}$. Recall that we say a fibration is totally non-homologous to zero (TNHZ) if its Serre spectral sequence collapses at $E_2$.


\begin{prop}\label{prop:fibration}
Let $X$ be a formal elliptic space. Then rationally it is the total space of a TNHZ fibration with model
\[(\Lambda B,0)\rightarrow (\Lambda B\otimes \Lambda V,D)\rightarrow (\Lambda V,d),\]
where $B=B^{odd}$ and $(\Lambda V,d)$ is positively elliptic.
\end{prop}

\begin{proof}
By \cite{FelixHalperin}, $X$ has a two-stage model of the form $(\Lambda Z,D)$, where $Z=Z_0\oplus Z_1$, such that $Z_1=Z_1^{odd}$ and $D$ maps a basis of $Z_1$ to a regular sequence $a_1,\ldots,a_k$ in $\Lambda Z_0$. Now set $V=Z_0^{even}\oplus Z_1$ and $B=Z_0^{odd}$, which produces the desired extension sequence. We now show that $(\Lambda V,d)$ is positively elliptic, where $d$ is the differential obtained by projecting $B^+$ to $0$.

Observe that $a_1,\ldots,a_k$ is in particular a regular sequence in the strictly commutative ring \[(\Lambda Z_0)^{even}=\Lambda Z_0^{even}\otimes (\Lambda Z_0^{odd})^{even}.\]
The right hand tensor factor has Krull dimension $0$ whence the Krull dimension of $(\Lambda Z_0)^{even}$ is equal to  
$r:=\dim Z_0^{even}$. In particular, we have $k\leq r$. Denote by $J\subset\Lambda Z_0^{even}$ the ideal generated by the canonical projections $\overline{a_1},\ldots,\overline{a_k}$ of the $a_i$ to $\Lambda Z_0^{even}$.
It follows from the odd spectral sequence of the elliptic algebra $(\Lambda Z,D)$ that the quotient $\Lambda Z_0^{even}/J$ is finite-dimensional. In particular, $J$ has height $r$. Since $J$ is generated by $k\leq r$ elements, it follows that $r=k$ and that the $\overline{a_i}$ form a regular sequence in $\Lambda Z_0^{even}$.

It remains to prove that the Serre spectral sequence collapses at $E_2$. Observe that $E_2^{*,*}=H^*(\Lambda B,d)\otimes H^*(\Lambda V,d)$. In particular we have $E_2^{0,*}=H^*(\Lambda V,d)$. This is completely represented by polynomials in $Z_0^{even}$ due to the fact that $(\Lambda V,d)$ is positively elliptic. As $D$ vanishes on $Z_0^{even}$, it follows that also the differentials of the spectral sequence vanish on $E_r^{0,*}$, $r\geq 2$, which causes the spectral sequence to collapse.
\end{proof}

\begin{cor}[\cite{KotaniYamaguchi}]
The TRC holds for formal elliptic spaces.
\end{cor}

\begin{proof}
Let $X$ be formal and elliptic and display its model as in the previous proposition. For the homotopy Euler characteristic of $X$, we obtain $\chi_\pi(X)=- \dim B$ which implies $\mathrm{rk}_0(X)\leq \dim B$. On the other hand, due to the collapse of the Serre spectral sequence, we have \[\dim H^*(X)=\dim H^*(\Lambda V,d)\cdot 2^{\dim B}.\]
\end{proof}

\section{Special classes of spaces and actions}\label{secspec}

\subsection{Vanishing derivations and symplectic actions}

In recent decades, the topology of torus actions on symplectic and K\"ahler manifolds has been a very successful field of study. In particular it was proved in \cite{AlldayPuppe1} that Hard Lefschetz manifolds satisfy the TRC (going back to the study of derivations on the cohomology algebra by \cite{blanchard}) and the result has been generalized to free actions on cohomologically symplectic spaces of Lefschetz type in \cite{Allday} and \cite{OpreaLupton}. It is therefore no surprise that among those spaces, we find classes of examples satisfying our formality properties. However, it will turn out that in the generalized setting of Lefschetz type spaces, the topology alone will not yield the formality properties we seek and we will need to resort to the geometric condition of the action being symplectic, which means that it is smooth and leaves invariant the symplectic form. Let us recall some notions.

\begin{defn}
A $2n$-dimensional compact symplectic manifold $(M,\omega)$ is said to be of Lefschetz type if multiplication with $\omega^{n-1}$ defines an isomorphism $H^1(M;\mathbb{R})\rightarrow H^{2n-1}(M;\mathbb{R})$.
\end{defn}

\begin{thm}\label{thm:symplecticexamples}
Let $T$ be a torus acting on $X$ such that one of the following holds:
\begin{enumerate}[(i)]
\item Any derivation of negative odd degree on $H^*(X)$ vanishes if it vanishes on $H^1(X)$.
\item $X$ is a compact symplectic manifold of Lefschetz type and the $T$-action is smooth and symplectic.
\end{enumerate}
Then the action is $\mathcal{MOD}$-formal and has formal core.
\end{thm}

The property in $(i)$ is fulfilled in particular for compact Kähler or more generally Hard Lefschetz manifolds (see \cite[Théorème II.1.2]{blanchard}). It is however a little more general as it holds e.g.\ for any space with cohomology concentrated in even degrees and is stable under products (\cite[Prop.\ 3.5]{BazzLupOprea}).
The proof of the theorem relies on the following property that unifies both types of actions.

\begin{lem}\label{lem:symplecticef}
Suppose the $T$-action on $X$ satisfies either condition $(i)$ or $(ii)$ above and that the map $H^2(BT)\rightarrow H^2_T(X)$ is injective. Then the action is equivariantly formal.
\end{lem}

\begin{proof} Observe that the injectivity assumption is equivalent to the vanishing of the transgression on the second page of the Serre spectral sequence of the Borel fibration. Under condition $(i)$ the lemma follows from \cite[Proposition 4.40]{AMIG} (note that differentials on odd pages vanish automatically as $R$ is concentrated in even degrees). In case $(ii)$, note first that it suffices to consider real coefficients. We may thus work with the spectral sequence of the Borel fibration which is obtained from the Cartan model by filtering in polynomial degree. The symplectic form induces an element $[\omega]\in E_2^{0,2}$ and we claim that $d_2[\omega]=0$.
If $d_2[\omega]=v\in E_2^{2,1}= H^2(BT)\otimes H^1(X)$ was nontrivial, Poincaré duality of $X$ would imply the existence of some element $x\in E_2^{0,2n-1}=H^{2n-1}(X)$ such that $xv\neq 0$. As $X$ is of Lefschetz type, we may write $x=[\omega]^{n-1}u$ for some $u\in E_2^{0,1}$. By assumption we have $d_2u=0$ and $[\omega]^nu=0$ for degree reasons. This implies $0=d_2[\omega]^nu=n[\omega]^{n-1}uv$ which is a contradiction.

By the definition of the Cartan differential, this means precisely that contractions of $\omega$ with the fundamental vector fields are exact. Consequently the action is Hamiltonian and thus equivariantly formal (\cite[Prop. 5.8]{kirwan}).
\end{proof}

\begin{rem}
Interestingly, although the proof of the TRC generalizes to cohomologically symplectic spaces of Lefschetz type in the purely topological setting, the statement of $(ii)$ in the above lemma is false without the geometric assumptions: \cite[Example 1]{Allday2} is an $S^1$-action with a fixed point on a simply-connected cohomologically symplectic space that is however not equivariantly formal.
\end{rem}

\begin{proof}[Proof of Theorem \ref{thm:symplecticexamples}]
Set $V=\ker(R^2\rightarrow H_T^*(X))$. There is a subtorus $T'\subset T$ such that the kernel of $H^2(BT)\rightarrow H^2(BT')$ is exactly $V$. Let $X_i\in V$ be a basis and $Y_i\in R^2$ be a basis of a complement of $V$. Then we have the following commutative diagram
\[\xymatrix{ (\Lambda (X_i, Y_j),0)\ar[r]\ar[d] & (\Lambda (X_i,Y_i)\otimes \Lambda V,D)\ar[d]\ar[r] & (\Lambda V,d)\ar[d]^{\Id_{\Lambda V}}\\
(\Lambda(Y_j),0)\ar[r] & (\Lambda (Y_j)\otimes \Lambda V,\overline{D})\ar[r]& (\Lambda V,d)
}\]
where the top row is a model for the Borel fibration of the $T$ action on $X$, the bottom row is a model for the Borel fibration of the restricted $T'$-action (note that the $T'$-Borel fibration is up to homotopy the pullback of the $T$-Borel fibration along $BT'\rightarrow BT$), and the vertical maps are defined by sending the $X_i$ to $0$. By construction, the $Y_j$ map injectively into the cohomology of $(\Lambda (Y_j)\otimes \Lambda V,\overline{D})$ and thus the bottom row fibration is TNHZ by Lemma \ref{lem:symplecticef}.

By assumption there are $s_i\in V^1$ with $D(s_i)=X_i$. In particular $\Lambda V=\Lambda (s_i)\otimes \Lambda W$ where $d(s_i)=0$ and $\overline{D}(s_i)=0$. Thus in the bottom row it makes sense to quotient out the ideal generated by the $s_i$, which yields an extension sequence
\begin{equation}
(\Lambda (Y_j),0)\rightarrow(\Lambda(Y_j)\otimes \Lambda W,\tilde{D})\rightarrow (\Lambda W,\overline{d}).\tag{$*$}
\end{equation}
We argue that this is TNHZ as well. By naturality of spectral sequences it suffices to argue that the projection $(\Lambda V,d)\rightarrow (\Lambda W,\overline{d})$ is surjective on cohomology. To see this, consider the composition
\[ (\Lambda(X_i,s_i) \otimes\Lambda W,\overline{D'})\rightarrow (\Lambda V,d)\rightarrow (\Lambda W,\overline{d})
\]
where $\overline{D'}$ is obtained from $D$ by dividing out the $Y_j$. It is defined by sending the contractible algebra $(\Lambda (X_i,s_i),D)$ to $0$ and is thus a quasi-isomorphism. In particular this shows cohomological surjectivity of the second morphism and thus degeneracy of the Serre spectral sequence of ($*$) at $E_2$.

As in the proof of Lemma \ref{lem:efismodformal}, we obtain a quasi-isomorphism between the free differential graded $\Lambda(Y_j)$-modules $(\Lambda(Y_j)\otimes \Lambda W,\tilde{D})$ and $(H^*(\Lambda(Y_j)\otimes \Lambda W),0)$, which we consider as $\Lambda(X_i,Y_j)$-modules via the projection. As before, sending $\Lambda (X_i,s_i)$ to $0$ yields a quasi-isomorphism
\[(\Lambda (X_i,Y_i)\otimes \Lambda V,D)\rightarrow (\Lambda(Y_j)\otimes \Lambda W,\tilde{D})\] of $\Lambda(X_i,Y_j)$ modules. We have shown that the action is $\mathcal{MOD}$-formal.

To see that it has formal core, set $A=\im(R\rightarrow H_G^*(X))$ and note that in the language of Definition \ref{def:formallybased}, $\Lambda C$ is given by $(\Lambda(X_i,Y_j,s_i),D)$. Thus the morphism $\overline{C}\rightarrow \Lambda V$ is just given by the inclusion of $(\Lambda(s_i),0)$. Note that in the Serre spectral sequence of the Borel fibration the classes of the $s_i$ transgress onto the classes of the $X_i$. Using the fact that the differential on the second page is a derivation one can inductively conclude that the product of the classes of all $s_i$ is nonzero in $H^*(\Lambda V,d)$ (see the proof of \cite[Theorem 2.2]{AlldayPuppe1}). This shows the cohomological injectivity of $\overline{C}\rightarrow \Lambda V$.

\end{proof}

\subsection{Rational ellipticity and non-negative sectional curvature}

The goal of this subsection is to provide a variation of the maximal symmetry rank conjecture for manifolds of non-negative sectional curvature. The Bott--Grove--Halperin conjecture which states that (almost) non-negatively curved manifolds are (rationally) elliptic provides a link to rationally elliptic spaces. Let us begin with some easy general observations.



The dimension formula (see \cite[Theorem 32.6]{Bibel})) for rationally elliptic spaces  implies that the maximum  homotopy Euler characteristic achievable for a simply-connected rationally elliptic space in dimension $n$ is given by $\lfloor n/3\rfloor$, indeed by
\begin{itemize}
\item $n/3$ if $n\equiv 0 \mod 3$ and all rational homotopy groups concentrated in degree $3$,
\item $(n-2)/3$ if $n\equiv 2 \mod 3$ and all rational homotopy groups concentrated in degree $3$ except for one in degree $5$, or all in degree $3$ except for two more corresponding to the ones of $S^2$.
\end{itemize}
In case $n\equiv 1\mod 3$ the maximal realizable homotopy Euler characteristic is $(n-4)/3$ because of the assumption of simply-connectedness.
Due to degree restrictions it follows in all cases that any representative splits as a product of the corresponding sphere factors. Hence we can find manifolds equipped with even  free torus actions of a torus of rank $\lfloor n/3\rfloor$ in any case. The second part of the following corollary provides a variation of Proposition \ref{prop:eq}.
\begin{cor}\label{CorEll}
\begin{itemize}
\item An almost free $T^{\lfloor \tfrac{n}{3}\rfloor}$-action on the simply-connected compact rationally elliptic space $X$ has formal core and is $\mathcal{MOD}$-formal.
\item Let $X$ be rationally elliptic with $\pi_*(X)\otimes \qq=\pi_\textrm{odd} (X)\otimes \qq$. If $X$ admits an almost free $T^{\chi_\pi(M)}$-action which is either of formal core or $\mathcal{MOD}$-formal, then $X$ rationally is a product of odd spheres.
    \end{itemize}
\end{cor}
\begin{proof}
As for the first part, we observe that the torus acts on a space of the rational type of one of
\begin{itemize}
\item $S^3\times \stackrel{(n/3)}{\ldots} \times S^3$,
\item $S^3\times \stackrel{((n-5)/3)}{\ldots} \times S^3 \times S^5$
\item $S^3\times \stackrel{((n-2)/3)}{\ldots}\times S^3 \times S^2$.
\end{itemize}
In order to justify this, as for the second item, it remains to observe that the rational homotopy group corresponding to $S^5$ is necessarily spherical and cannot have a twisted differential. Indeed, this being the case would reduce cohomology (cf.~the arguments for the second part) and would contradict the fact that the toral rank conjecture is confirmed on coformal two-stage spaces (see \cite [Proposition 3.1, Corollary 3.5]{AlldayPuppe1}).

The corresponding Borel constructions---since they have finite dimensional cohomology ---have the rational types of positively elliptic hence formal spaces whence the action has formal core and is $\mathcal{MOD}$-formal.

%

The second part can easily be deduced as follows: the associated pure model of the minimal model of $X$ is the one of a product of odd spheres. It has total cohomological dimension $2^{\chi_\pi(M)}$, and constitutes the $E_0$-term of the odd spectral sequence converging to the model of $X$ (see \cite[p.~438]{Bibel}). By  Theorem \ref{thm:TRCholdsforbla} we derive that cohomological dimension may not decrease whence the spectral sequence has to degenerate at the $E_0$-term. The result then follows easily.
\end{proof}
We remark that the first statement of Corollary \ref{CorEll} is a variation/generalization of \cite[Corollary C]{GKR17} in which it is shown that the TRC holds if $T^{\lfloor n/3\rfloor}$ acts almost freely on $X$. In \cite[Corollary C, Theorem D]{GKRW18} it was shown that the TRC holds on a closed smooth simply-connected rationally elliptic $n$-manifold (respectively on a closed simply-connected non-negatively curved Riemannanian manifold) which admits an additional continuous (respectively isometric) torus action (independent of the almost free torus action) which is \emph{isotropy maximal}. This notion of maximality in combination with non-negative sectional curvature and in relation to formal core and $\mathcal{MOD}$-formality will be discussed in the following.

\vspace{5mm}

So indeed, let us now come to an application in the field of manifolds with non-negative sectional curvature. For this we recall (see \cite{ES17})
the \emph{maximal symmetry rank conjecture} on these, which was originally stated in \cite{GalazSearle}. Note that a torus action of $T^k$ on the manifold $M$ is called \emph{isotropy maximal} (or \emph{slice maximal}) if there is a point for which the isotropy group is of maximal dimension, or equivalently, if and only if, the dimension of a minimal orbit is $2k-n$. In this case the action is maximal in the sense that it is not a proper subaction of an effective action of a higher dimensional torus.


\begin{conj2}
Let $T^k$ act isometrically and effectively on the simply-connected closed non-negatively curved manifold $M^n$. Then $k\leq \lfloor 2n/3\rfloor$ and the action is isotropy maximal for $k=\lfloor 2n/3\rfloor$.
\end{conj2}
In the case of an isotropy maximal effective isometric $T^k$-action on $M$ (closed, simply-connected, of non-negative curvature)
it was shown in \cite{ES17} that $M$ is equivariantly diffeomorphic to a quotient of a free linear
torus action on
\begin{align*}
Z=\prod_{i<r} S^{2n_i} \times \prod_{i\geq r} S^{2n_i-1}
\end{align*}
for all $n_i\geq 2$, and $n\leq \dim Z\leq 3n-3k$.


Recall that by a linear torus action on $Z$ we refer to the action of a subtorus of a maximal torus of $\prod_{i<r} \SO(2n_i+1) \times \prod_{i\geq r} \SO(2n_i)$. This group acts by the standard representation on $\prod_{i<r} \mathbb{R}^{2n_i+1} \times \prod_{i\geq r} \mathbb{R}^{2n_i}$ restricting to an action on $Z$.

\vspace{5mm}

In \cite[Theorem B]{GKRW18} an analog result was shown in the category of rationally elliptic manifolds and smooth actions: An isotropy maximal action on a smooth closed simply-connected manifold $M^n$ is up to equivariant rational homotopy equivalence a linear one on a quotient $Z$ on a of product of spheres (of dimension at least $3$) by a linear free torus action.

As a direct corollary of Theorem \ref{thm:TRCholdsforbla} we can prove this extended form of the maximal symmetry rank conjecture in a special case.


\begin{cor}\label{cornonneg}
Suppose an effective $T^k$-action on a simply-connected closed manifold $M^n$ which is either $\mathcal{MOD}$-formal or of formal core, and let  $\dim H^*(M)< 2^{2k+1-n}$.
\begin{itemize}
\item
Suppose the action is isometric and $M$ is  non-negatively curved manifold $M^n$. Then the action is isotropy maximal, and $M$ is equivariantly diffeomorphic to a quotient of a free linear torus action on a product of spheres $Z$ (as above).
\item
Suppose $M$ is rationally elliptic. Then the action is isotropy maximal, and $M$ is equivariantly rationally equivalent to a quotient of a free linear torus action on a product of spheres (as described above).
\end{itemize}
\end{cor}
\begin{proof}
Under these restrictions Theorem \ref{thm:TRCholdsforbla} yields that $\dim H^*(M)\geq 2^c$ with $c$ the dimension of a minimal orbit. From the dimension restriction on cohomology we deduce that
\begin{align}\label{eqnnn}
c\leq 2k-n
\end{align}
(as it is an integer).

Consider a $T^k$-orbit at $x \in T^k x\cong T^k/T'$ where $T'$ denotes the isotropy subtorus at $x$ (up to finite covering). By standard representation theory we obtain that the maximal dimension of a faithful isotropy representation of the torus $T'$ at $x$ is bounded from above by $(n-c)/2$. It follows that the maximal dimension $k$ of an effectively acting torus $T^k$ is bounded by
\begin{align*}
c + (n-c)/2 \geq k  \iff  c\geq 2k-n.
\end{align*}
Together with Equation \eqref{eqnnn} this yields that $c=2k-n$. It follows that the action is isotropy maximal. The two respective results follow from the literature as depicted above the assertion.
\end{proof}
Conversely, note that any such linear action on a product of spheres is both $\mathcal{MOD}$-formal and of formal core. Clearly, by O'Neill any such quotient can be equipped with a metric of non-negative curvature.

We also recommend to contrast this result with Proposition \ref{prop:eq}, where we showed that---applied to this concrete case---the equality $H^*(M)=2^{2k-n}$ implies that $M$ has the rational type of a product of spheres. Hence, in the setting of this Section, this result in particular is sharpened to a diffeomorphism classification.

\vspace{5mm}

Let us end this section by providing another result for rationally elliptic spaces.
\begin{lem}\label{lemform}
If $(\Lambda V,d)$ is a formal elliptic minimal Sullivan model, then so is its associated pure model $(\Lambda V,d_\sigma)$.
\end{lem}
\begin{proof}
This is a direct consequence of Proposition \ref{prop:fibration}, where we showed that $(\Lambda V,d)$ admits a decomposition as a TNHZ rational fibration
\[(\Lambda B,0)\rightarrow (\Lambda B\otimes \Lambda V,d)\rightarrow (\Lambda V,\bar d),\]
where $B=B^{odd}$ and $(\Lambda V,\bar d)$ is positively elliptic. Indeed, passing to the associated pure model makes this fibration degenerate to a genuine product, and $(\Lambda V,d_\sigma)$ is formal, consequently.
\end{proof}

\begin{thm}
Let $M^n$ be a simply-connected rationally $(r-1)$-connected formal rationally elliptic closed manifold with an effective action of $T^k$ which is $\mathcal{MOD}$-formal or of formal core.
Then
\begin{align*}
k\leq n/2+n/(2r)
\end{align*}

In the case of equality the action is isotropy maximal and equivariantly rationally equivalent to a linear action on a quotient of a product of spheres by a linear free torus action---if $r\geq 3$ then $M$ is equivariantly rationally equivalent to a product of spheres.
\end{thm}
\begin{proof}
%
%
%
Let us first estimate the dimension of the cohomology of $M$ from above. For this, due to the odd spectral sequence we may assume that $M$ is a pure space whilst preserving that it is rationally elliptic and formal in view of Lemma \ref{lemform}. Hence we may decompose its minimal model as the product of a  positively elliptic Sullivan algebra and a free algebra generated in odd degress (see for example \cite[Lemma 2.5]{Ama13}).

We denote by $(a_1,\ldots, a_\alpha)$ the multiset of the degrees of a basis of the vector space of non-trivial rational homotopy groups $\langle \pi_i(X)\otimes \qq\rangle_i$ for $i$ odd and by $(b_1,\ldots, b_\beta)$ the corresponding multiset of degrees for $i$ even. Clearly, $\alpha\geq \beta$. We may order the $a_i$ in such a way that the differential is trivial on $a_i$ for $i>\beta$. Moreover, from  \cite[Proof of Theorem 32.6, Page 443]{Bibel} we recall that we may order the $b_i$ and $a_i$ (for $i\leq \beta$) in such a way that the $b_i$ and the $a_i$ form a respective increasing sequence and $a_i+1\geq 2 b_i$ for $1\leq i\leq \beta$.

Using \cite[Theorem 32.15]{Bibel} we compute the dimension $n$ of $M$ as
\begin{align*}
\dim M &= \sum_{i=1}^\alpha a_i-\sum_{i=1}^\beta (b_i-1)=\sum_{i=1}^\alpha (a_i+1)-b_i
\intertext{(where, by convention, $\beta_i=0$ for $i>\beta$), and its cohomology as}
\dim H^*(M)&= \bigg(\prod_{i=1}^\beta \frac{a_i+1}{ b_i}\bigg) \cdot 2^{\alpha-\beta}
\end{align*}
where the first factor computes the dimension of the cohomology of the positively elliptic factor, the second factor the one of the free factor. For this we use \cite[Property (32.14), p.~446]{Bibel} in order to compute the dimension of the cohomology of a rationally elliptic pure space.

Hence we need to maximize the product of quotients ($\geq 2$) of numbers while the sum of their differences remains constant. The solution to this problem is to produce as many factors as possible. That is, we estimate the cohomology by
\begin{align*}
\dim H^*(M)&\leq 2^{n/r} 
\end{align*}

\vspace{5mm}

At the same time we recall from the proof of Corollary \ref{cornonneg} the inequality
\begin{align*}
 c\geq 2k-n.
\end{align*}
with $c$ the minimal dimension of an orbit.  Theorem \ref{thm:TRCholdsforbla} yields that $\dim H^*(M)\geq 2^c\geq 2^{2k-n}$. Combining both inequalities we obtain that
\begin{align*}
2k-n \leq n/r \iff k\leq n/2+n/(2r)
\end{align*}
If equality holds, we derive that $H^*(M)=2^c$ and $c=2k-n$, whence the action is isotropy maximal. We then draw on the depicted classification result \cite[Theorem B]{GKRW18}. The statement for $r\geq 3$ follows since the quotient of a free torus action on a product of spheres of dimension at least $3$ can only be rationally $2$-connected if the torus is trivial.

\end{proof}

\section{Further (counter)examples}\label{secex}
\begin{ex}\label{ex:restriction}
We show that the notions we defined in Section \ref{sec:notionsection} are \textbf{not preserved by restriction to subgroups of smaller rank}. Consider the threefold Hopf action of $T^3$ on $M=(S^{3})^3$. The space $M_{T^3}$ is formal so it is in particular $\mathcal{MOD}$-formal as well as formally based. However, if we restrict the action to $T^2$ along the homomorphism $(s,t)\mapsto (s,st,t)$, the model of the Borel fibration becomes
\[(\Lambda(X,Y),0)\rightarrow (\Lambda(X,Y,a,b,c),D)\rightarrow (\Lambda(a,b,c),0),\]
with $|X|=|Y|=2$, $|a|=|b|=|c|=3$, $D(a)=X^2$, $D(b)=X^2+2XY+Y^2$, and $D(c)=Y^2$. Now $2Ya-X(b-a-c)$ defines a nonzero cohomology class in degree 5 which maps to $0$ in $H^5(M)=0$. Also $H^3_{T^2}(M)=0$ so said class does not lie in $\mathfrak{m}H^*_{T^2}(M)$, where $\mathfrak{m}=(X,Y)$. It follows that the action is not spherical. It is also not formally based because of the nontrivial Massey product $\langle X,X,Y\rangle$: the image of $H^*(BT^2)\rightarrow H_{T^2}^*(M)$ is isomorphic to $\Lambda(X,Y)/(X^2,XY,Y^2)$. Let $(C,d)$ be as in Definition \ref{def:formallybased} and $\alpha,\beta,\gamma\in C$ with $ d\alpha=X^2$, $d\beta=X^2+2XY+Y^2$, and $d\gamma=Y^2$. Any morphism $(C,d)\rightarrow (\Lambda(X,Y,a,b,c),D)$ that is the identity on $R$ also has to map $\alpha\mapsto a$, $\beta\mapsto b$, $\gamma\mapsto c$, which is not possible since there is nontrivial cohomology represented in the $\Lambda(X,Y)$-span of $a,b$, and $c$ but not in that of $\alpha,\beta$, and $\gamma$.
\end{ex}

\begin{ex}\label{ex:sphericalnotalmostmod}
We construct an action which is \textbf{spherical but not almost $\mathcal{MOD}$-formal}. Consider $\Lambda (X_1,X_2,X_3)$ with $X_i$ in degree $2$ and set the differential $D$ to be trivial on the $X_i$. Add generators $a_{ij}$ for $i\leq j\in\{1,2,3\}$ and set $D(a_{ij})=X_iX_j$.
One checks that a basis of the kernel of $D$ in degree $5$ is given by
\begin{itemize}
\item $m_{iij}=X_ja_{ii}-X_ia_{ij}$, for $(i,j)\in\{1,2,3\}^2$ with $i\neq j$, where $a_{ij}:=a_{ji}$ in case $i>j$.

\item $m_{312}=X_3a_{12}-X_2a_{13}$, $m_{123}=X_1a_{23}-X_2a_{13}$.
\end{itemize}
The easiest way to verify this is by observing that the listed elements are linearly independent and that the differential maps the $18$-dimensional degree $5$ component surjectively onto the $10$-dimensional $\Lambda^3(X_1,X_2,X_3)$.

Now for each of the $m_{ijk}$ except for $m_{112}$, we introduce a generator $c_{ijk}$ in degree $4$ with $D(c_{ijk})=m_{ijk}$.
Furthermore, we add a generator $b$ in degree $3$ with $Db=0$ and glue it to the remaining Massey product by introducing the generator $c_{112}$ in degree $4$ and setting $D(c_{112})=m_{112}-X_1b$. We extend the cdga
\[(A,D)=(\Lambda(X_i,a_{ij},b,c_{ijk}),D)\]
to a minimal Sullivan algebra $(C,D)$ by adding generators in degrees $\geq 5$ to inductively kill all cohomology in degrees $\geq 6$.

Setting $R=\Lambda(X_1,X_2,X_3)$, we have $H^*(C)=R/R^{\geq 3}\otimes \Lambda (b)$. By Proposition \ref{prop:realizeasaction} we find a compact space $X$ with a $T^3$ action such that the Borel fibration has minimal model
\[R\rightarrow (C,D)\rightarrow (\overline{C},\overline{D})\]
with the latter cdga being the quotient by the ideal generated by $R^+$.
As an $R$-module, $H^*_{T^3}(X)=H^*(C)$ is generated by the classes of $1$ and $b$ which map injectively to $H^*(X)=H^*(\overline{C})$ so the action is spherical by Lemma \ref{spherical-generator-lem}.

Let us now argue that it is not almost $\mathcal{MOD}$-formal. By Proposition \ref{E3collapse}, it suffices to find a nontrivial differential on the $E_3$-page of the spectral sequence of the fibration $T^3\rightarrow X\rightarrow X/T^3$. It is obtained by filtering $(C\otimes \Lambda S,D)$ in the degree of $C$ where $S=\Lambda(s_1,s_2,s_3)$ with $Ds_i=X_i$. We find a nontrivial differential by defining a suitable zigzag: consider the element $\alpha_2=X_1s_{123}$ in filtration degree $2$, where the multi index stands for the respective product of the $s_i$. We have $D(\alpha_2)=X_1^2s_{23}-X_1X_2s_{13}+X_1X_3s_{12}$. Thus for $\alpha_3=-a_{11}s_{23}+a_{12}s_{13}-a_{13}s_{12}$, we obtain \[D(\alpha_2+\alpha_3)=(X_2a_{13}-X_3a_{12})s_1+(X_3a_{11}-X_1a_{13})s_2+(X_1a_{12}-X_2a_{11})s_3\]
in filtration degree $5$. We claim that this defines a nontrivial element in $E_3^{5,1}$, which implies that we have a nontrivial differential starting from $E_3^{2,3}$.

First note that $D(\alpha_2+\alpha_3)\equiv D(\alpha_2+\alpha_3)+D(c_{312}s_1-c_{113}s_2)\equiv-m_{112}s_3+X_1c_{213}-X_2c_{113}$ in $E_3^{5,1}$. For this to be trivial in $E_3^{5,1}$ we need to have
\[m_{112}s_3-D(e_{12}s_{12}+e_{13}s_{13}+e_{23}s_{23}+e_1s_1+e_2s_2+e_3s_3)\in C^6\]
for some $e_{ij}\in A^3$ and $e_i\in A^4$. We see that the $e_{ij}$ are necessarily closed and hence multiples of $b$. If the $e_{ij}$ were nontrivial, then their image under $D$ would produce error terms of the form $X_ib_js_k$ which cannot cancel with the $D(e_is_i)$ because no element in $R^2\cdot \langle b\rangle$ is exact. Thus we have $e_{ij}=0$. But then it follows that the projection of the whole right hand expression $D(\cdots)$ to the $A^5\otimes \langle s_3\rangle$ component is given by $D(e_3)s_3$, which can never be equal to $m_{112}s_3$ because $m_{112}$ is not exact in $A$.
\end{ex}

\begin{ex}\label{ex:almostmodnotmod}
We construct an \textbf{almost $\mathcal{MOD}$-formal but not $\mathcal{MOD}$-formal} free action. Consider the cdga $(C,D)$ from Example \ref{ex:sphericalnotalmostmod}. We set $R=\Lambda (X_1,X_2)$ and obtain a compact $T^2$-space $Y$ such that the Borel fibration is given by $R\rightarrow C$.
The $R$-module $H^*_{T^2}(Y)$ is generated by the classes of $1$, $X_3$, $b$, and $X_3b$ which map injectively to $H^*(Y)$. Thus the action is spherical by Lemma \ref{spherical-generator-lem}. A nontrivial differential past the second page of the spectral sequence from Proposition \ref{E3collapse} would need to map from $E_3^{*,2}$ to $E_3^{*,0}$ because we only have a $T^2$-action. But those are also trivial according to Proposition \ref{sphericalSS}. Thus the action is almost $\mathcal{MOD}$-formal.

Finally, we argue that the action is not $\mathcal{MOD}$-formal. To see this, we construct the lower part of the Hirsch--Brown model $\varphi\colon R\otimes V\rightarrow C$ (see \ref{HBconstruction}). In degree $0$ we introduce the generator $\overline{1}$ and set $\varphi(\overline{1})=1$. In degree $1$ there is nothing to do and in degree $2$ we introduce generators $\overline{X_3}$ and define $\varphi(\overline{X_3})=X_3$. To achieve surjectivity in degree $3$, we introduce $\overline{b}$, with $\varphi(\overline{b})=b$. Injectivity of $\varphi$ in degree $4$ is achieved by introducing $\overline{a_{11}},\overline{a_{12}},\overline{a_{22}}$ with $D(\overline{a_{ij}})=X_iX_j\cdot\overline{1}$ as well as $\overline{a_{13}}$ and $\overline{a_{23}}$
with $D(\overline{a_{i3}})=X_i\cdot\overline{X_3}$ and setting $\varphi(\overline{a_{ij}})=a_{ij}$.
The procedure is continued by adding generators of degree $\geq 4$ to form the Hirsch--Brown model.

If the action were $\mathcal{MOD}$-formal, the criterion in Lemma \ref{formality criterion} would tell us that we can choose preimages $\alpha,\beta,\gamma\in V^3$ of the exact cocycles $X_1^2\cdot\overline{1}$, $X_1X_2\cdot\overline{1}$, and $X_2^2\cdot\overline{1}$ such that every closed element in $R\otimes \langle \alpha,\beta,\gamma\rangle_\mathbb{Q}$ is exact. Every such $\alpha$, $\beta$, and $\gamma$ are of the form $\alpha=\overline{a_{11}}+t_\alpha \overline{b}$, $\beta=\overline{a_{12}}+t_\beta \overline{b}$, $\gamma=\overline{a_{22}}+t_\gamma \overline{b}$ for $t_\alpha,t_\beta,t_\gamma\in \mathbb{Q}$. But no choice of scalars fulfils the condition that
\begin{align*}
\varphi(X_2\alpha-X_1\beta)&=m_{112}+(t_\alpha X_2-t_\beta X_1)b\\ \text{and} \quad
\varphi(X_2\beta-X_1\gamma)&=-m_{221}+(t_\beta X_2-t_\gamma X_1)b
\end{align*}
are simultaneously exact: exactness of the first element would require $(t_\alpha,t_\beta)=(0,1)$, whereas for the second one we need $(t_\beta,t_\gamma)=(0,0)$.
\end{ex}

\begin{rem}\label{rem:manifoldcounterexample}
We point out that the counterexamples \ref{ex:sphericalnotalmostmod} and \ref{ex:almostmodnotmod} can be modified to produce \textbf{simply-connected compact manifolds}. By Proposition \ref{prop:realizeasaction} it suffices to extend the cdga $(C,D)$ to a cdga whose cohomology satisfies Poincaré duality and check that the arguments carry over.

To do this, choose any big enough degree $N$ in which the fundamental class shall live. We assume $N\geq 11$ for convenience in the arguments below. Now introduce generators $e_1$, $e_2$, and $e_3$ in degree $N-5$ and map them to $0$ under the differential. In degree $N-3$, cohomology is now generated by the $X_ie_j$ for $1\leq i,j\leq 3$. We introduce new generators in degree $N-4$ and map them (under the differential) to $X_ie_j$ for $i\neq j$ as well as to $X_1e_1-X_2e_2$ and $X_2e_2-X_3e_3$. Let $V^{N-4}$ denote the space spanned by the newly introduced generators. Now the cohomology in degree $N-3$ is $1$-dimensional and represented by any of the $X_ie_i$. Also, degree $N-2$ cohomology is represented by $be_1$, $be_2$, and $be_3$ and cocycles in $V^{N-4}\cdot \langle X_1,X_2,X_3\rangle$. Introduce new generators $V^{N-3}$ to kill all cohomology of the latter kind in degree $N-2$. Since the differential is injective on $C^3\cdot V^{N-4}$, cohomology in degree $N-1$ is entirely represented in $V^{N-3}\cdot \langle X_1,X_2,X_3\rangle$. We kill this cohomology by introducing generators in $V^{N-2}$.

We claim that in degree $N$, the elements of the form $X_ibe_i$ are not exact (although they are cohomologous to one another). Indeed, since for $i=N-3,N-2$ the differential maps $V^i$ into the ideal generated by $V^{i-1}$, it suffices to check whether e.g.\ $X_1be_1$ is in the image of $C^3\cdot V^{N-4}$, which is clearly not the case. We choose a complement of $\langle [X_1be_1]\rangle$ in degree $N$ cohomology and introduce generators $V^{N-1}$ which map bijectively to representatives of a basis of the complement. Now inductively kill all cohomology in degrees $>N$. Representatives for a basis of the cohomology are given by\smallskip\\
\begin{center}
\begin{tabular}{c|cccccccc}
 degree& 0 & 2 & 3 & 5 &$N-5$ & $N-3$& $N-2$& $N$ \\
\hline & & $X_1$& & $X_1b$ & $e_1$ & & $be_1$ & \\
generators & $1$ & $X_2$ & $b$ & $X_2b$ & $e_2$ & $X_ie_i$ & $be_2$ & $X_ibe_i$\\
  & & $X_3$& & $X_3b$ & $e_3$ & & $be_3$ &
\end{tabular}
\end{center}\smallskip
and we observe that Poincaré duality holds.

To check that the arguments in Examples \ref{ex:sphericalnotalmostmod} and \ref{ex:almostmodnotmod} carry over, note first that both the $T^3$-action defined by $X_1,X_2,X_3$ as well as the $T^2$-action defined by $X_1,X_2$ are still spherical: for $R=\Lambda (X_1,X_2,X_3)$, a generating set of the cohomology as an $R$-module is given by $1$, $b$, $e_i$, and $e_ib$ for $i=1,2,3$. For $R=\Lambda (X_1,X_2)$ we need to also add $X_3$ and $X_3b$ to the list. In any case, none of those generators become exact when dividing by the ideal of $R^+$ so both actions are spherical by Lemma \ref{spherical-generator-lem}. The arguments showing that the actions are not (almost) $\mathcal{MOD}$-formal took place only in the lower half of the cdgas, which we did not modify.

\end{rem}

\begin{ex}\label{ex:nontrivmodform}
We construct a \textbf{$\mathcal{MOD}$-formal action} that is \textbf{neither equivariantly formal nor has a formal homotopy quotient}. We expand on a discussion from \cite{kadeishvili}. Consider the graded vector space $H$ with Betti numbers $1,0,3,0,0,1$. For degree reasons, all operations of a unital $A_\infty$-algebra structure on $H$ vanish except for $m_3\colon H^2\otimes H^2\otimes H^2\rightarrow H^5$.  In turn one checks that any specification of $m_3$ does indeed yield an $A_\infty$-structure, where the formal one is just given by the cohomology of $S^2\vee S^2\vee S^2\vee S^5$.
For an $A_\infty$-structure to be $C_\infty$ it is required that it vanishes on all shuffles. In our case this is equivalent to $m_3$ vanishing on all elements of the form $a\otimes b\otimes c-a\otimes c\otimes b+c\otimes a\otimes b$ for $a,b,c\in H^2$. Let $\alpha\in H^5$ be a generator and let $X,Y,Z\in H^2$ be a basis. Then by setting $m_3(Z\otimes Z\otimes X)=\alpha$, $m_3(X\otimes Z\otimes Z)=-\alpha$, and $m_3(a\otimes b\otimes c)=0$ for all other tensors with $a,b,c\in\{X,Y,Z\}$, we obtain a $C_\infty$-structure on $H$.
We observe that $m_3$ vanishes on $H^2\otimes A\otimes A$, where $A=\langle X,Y\rangle$ is the sub-algebra of $H$ (with trivial multiplication) generated by $X$ and $Y$.
Consequently, the classes $X,Y$ satisfy the requirements of Remark \ref{Cinftyexampleschmiede} and we obtain a free $\mathcal{MOD}$-formal $T^2$-action on a compact space such that the orbit space has the rational homotopy type of $(H,m_3)$.

To see that this is indeed not the formal rational homotopy type, note that the only nontrivial component of a $C_\infty$-morphism $f\colon(H,m_3)\rightarrow (H,m_3')$ is $f_1$, again for degree reasons.
Thus $m_3$ and $m_3'$ yield isomorphic $C_\infty$-structures if and only if $f_1\circ m_3=m_3'\circ (f_1\otimes f_1\otimes f_1)$ for an automorphism $f_1$ of the graded vector space $H$. In particular $(H,m_3)$ and $(H,0)$ are not isomorphic.
\end{ex}

\begin{ex}\label{ex:injectively based, not spherical}
We construct a \textbf{hyperformal action} which is \textbf{not spherical}.
Consider the nilmanifold with model \[(N,d)=\Lambda(y_1,\ldots,y_5,z_1,\ldots,z_4,d)\] with $dy_i=0$, $dz_i=y_iy_{i+1}.$ The extension
\[(\Lambda(X_1,X_2),0)\rightarrow(\Lambda(X_1,X_2)\otimes N,D)\rightarrow (N,d)\]
with $X_i$ of degree 2, $Dz_1=X_1+y_1y_2$, $Dz_4=X_2+y_4y_5$, and $D=d$ on the other generators is the model of the Borel fibration of a free $T^2$-action. By dividing out a contractible ideal, we obtain a commutative diagram

\[ \xymatrix{(\Lambda(X_1,X_2),0)\ar[r]\ar[dr]& (\Lambda(X_1,X_2)\otimes N,D)\ar[d]\ar[r] & (N,d)\\ & (\Lambda(y_1,\ldots,y_5,z_2,z_3),d)\ar[ur] & } \]
where the vertical map is a quasi-isomorphism with $X_1\mapsto-y_1y_2$, $X_2\mapsto-y_4y_5$, $z_1,z_4\mapsto 0$ and which is the identity on the remaining generators. Since $y_1y_2y_4y_5$ defines a nonzero cohomology class in the lower cdga while $(y_1y_2)^2=(y_4y_5)^2=0$, the kernel of $\Lambda(X_1,X_2)\rightarrow H^*(\Lambda(X_1,X_2)\otimes N,D)$ equals $(X_1^2,X_2^2)$. Hence the action is hyperformal.

Now consider the cocycle $\alpha=-y_1y_4y_5z_2+y_1y_2y_5z_3$ in the bottom cdga of the above diagram. One checks that the cohomology class $[\alpha]$ is not in the algebra span of the degree $2$ classes which are represented by $y_iy_j$, $y_2z_2$, $y_3z_2$, $y_3z_3$, and $y_4z_3$. In particular $[\alpha]$ does not lie in $\Lambda^+(X_1,X_2)\cdot H^*(\Lambda(y_i,z_2,z_3),d)$ with respect to the module structure defined by the diagram.
However, it becomes exact in $(N,d)$ where \[\alpha=d(y_3z_1z_4-y_1z_2z_4-y_5z_1z_3).\]
This shows that the action is not spherical.
\end{ex}

\begin{ex}\label{ex:masseynotnecessary}
We show the existence of a \textbf{$\mathcal{MOD}$-formal} action that is \textbf{not formally based} by constructing the cdga $(\Lambda W,D)$ in the following way: introduce generators $X,Y$ in degree $2$ and $a,b,c$ in degree $3$ with $D(a)=X^2$, $D(b)=XY$, and $D(c)=Y^2$. Then the kernel in degree $5$ is generated by $\alpha_1=Ya-Xb$ and $\alpha_2=Yb-Xc$. Now introduce generators $d$ and $e$ in degree $4$ and set $D(d)=\alpha_1$, $D(e)=\alpha_2$. At this point, $(\Lambda W^{\leq 4},D)$ has no cohomology in degrees $3$, $4$, and $5$ while in degree $6$ representatives for a basis are given by $\alpha_3=Yd+ac+Xe$, $\alpha_4=Xd+ab$, and $\alpha_5=bc+Ye$. Now complete $\Lambda W$ to a minimal cdga by inductively killing all cohomology starting from degree $7$. By Proposition\ \ref{prop:realizeasaction}, there is a free $T^2$-action on a compact space $M$ whose homotopy quotient $M_{T^2}$ has $(\Lambda W,D)$ as minimal model. One quickly sees that $(\Lambda W,D)$ is not formal because $\alpha_1$, $\alpha_2$, and $\alpha_3$ represent nontrivial quadruple Massey products. However, we argue that it is formal as a $\Lambda(X,Y)$-module by constructing a quasi-isomorphism $\Lambda W\rightarrow H^*(M_{T^2})$. We set this map to be the canonical projection in degrees $0$ and $2$ and trivial in all other degrees except degree $6$. In degree $6$, we choose the canonical basis given by the products of the generators and map $\Lambda^3 W^2$ and $W^2\cdot W^4$ to $0$ while sending $ac$, $ab$, and $bc$ to $[\alpha_3]$, $[\alpha_4]$, and $[\alpha_5]$. It is easy to check that this map does have the desired properties.

To see that the action is not formally based, we observe that for degree reasons the only possible nontrivial operation of a unital $C_\infty$-structure on $H=H^*(M_{T^2})$ is \[m_4\colon H^2\otimes H^2\otimes H^2\otimes H^2\longrightarrow H^6.\]
This map has to be nontrivial for every $C_\infty$-model structure because $M_{T^2}$ is not formal. But $H^*(BT^2)=\Lambda (X,Y)\rightarrow H^*(\Lambda W,D)=H^*(M_{T^2})$ is surjective in degree $2$ so the requirements of Theorem  \ref{thm:Masseyformallybased} are not met.
\end{ex}

\begin{ex}\label{ex:biquotnichtinjbased}
We give an example of an \textbf{action on a homogeneous space} which has a \textbf{formal homotopy quotient} but \textbf{does not have formal core with respect to a smaller subalgebra} of $H_G^*(X)$. Consider the action of the maximal diagonal torus $T^4$ on $\U(4)$ by multiplication from the left. If $X_1,\ldots,X_4\in H^2(BT^4)$ is the basis dual to the standard basis of the Lie-Algebra of $T^4$ and $R=\Lambda(X_1,\ldots,X_4)$, then the model of the borel fibration is given as
\[(R,0)\rightarrow  (R\otimes \Lambda Z,D)\rightarrow (\Lambda Z,0)\]
where $Z$ is $4$-dimensional and $D$ maps a basis to the elementary symmetric polynomials $\sigma_1,\ldots,\sigma_4$ in the variables $X_i$. Now we make a change of basis by pulling back the action along the automorphism $\phi$ of $T^4$ which, in the standard basis of the Lie algebra, is given by the matrix
\[A=\begin{pmatrix}
1 & 0 & 0& 0\\ 0& 1& -2& 0\\ 0&0 & 1 &0\\ 0& 0& 0& 1
\end{pmatrix}.\]
This induces an automorphism $\phi^*$ of $R$ which, in the basis $X_1,\ldots,X_4$, is represented by $A^t$. The model of the Borel fibration of the new action is the same except $D$ is replaced by the differential $\tilde{D}$ which maps a basis of $Z$ to the polynomials $\phi^*(\sigma_i)$. We consider the splitting $T^4=T\times T'$ with $T$ consisting of the two circle factors on the left and $T'$ of those on the right. We claim that the $T$-action on $\U(4)/T'$ has a formal homotopy quotient but does not have formal core with respect to $\im(H^*(BT)\rightarrow H_T(\U(4)/T')$.

The model of the Borel fibration of the $T$-action is
\[(\Lambda(X_1,X_2),0)\rightarrow (R\otimes\Lambda Z,\tilde{D})\rightarrow (\Lambda(X_3,X_4)\otimes \Lambda Z, \overline{D}).\]
The middle cdga is formal so the action does indeed have a formal homotopy quotient. Let $S=\Lambda(X_1,X_2)$ and $i\colon S\rightarrow R$ be the inclusion. Then $J:=\ker (S\rightarrow H_T^*(\U(4)/T'))= i^{-1}(\phi^*(I))$, where $I\subset R$ is the ideal generated by the $\sigma_i$. One computes that $(\phi^*)^{-1}\circ i(X_1)=X_1$ as well as $(\phi^*)^{-1}\circ i(X_2)=X_2+2X_3$ and, using appropriate tools, we obtain
\[J=(X_1^4, 28X_1^3X_2^2+12X_1^2X_2^3+3X_1X_2^4,X_2^6).\]
We recommend the freely available software Macaulay2 for such computations and the ones that follow below. Let $(C,d)$ be a relative minimal model of $(S,0)\rightarrow (S/J,0)$. It is our goal to show that a morphism
\[(C\otimes \Lambda Z',d)\rightarrow (R\otimes \Lambda Z\otimes \Lambda Z',\tilde{D})\]
can not be cohomologically injective, where the differentials map the generators $Z'=\langle z_1,z_2\rangle$ to $X_1$ and $X_2$. This follows for dimensional reasons: formality provides us with quasi-isomorphisms
\[(C\otimes \Lambda Z',d)\rightarrow (S/J\otimes \Lambda Z',d)\quad\text{and}\quad
(R\otimes \Lambda Z\otimes \Lambda Z',\tilde{D})\rightarrow (R/\phi^*(I)\otimes \Lambda Z',d).\]
The cohomologies are thus given by
\[\mathrm{Tor}^S_*(S/J,S/(X_1,X_2))\quad\text{and}\quad \mathrm{Tor}^S_*(R/\phi^*(I),S/(X_1,X_2))\] which can again be computed with standard software. Doing so, one finds the first one to be nontrivial in degree $11$ (in the cdga grading) while the second one is $0$ in this degree. It follows that the action has the desired properties. As a side note we want to add that without the change of basis via $\phi$, the analogous construction does yield an action that has formal core with respect to $\im(S\rightarrow H_T(\U(4)/T')$.
\end{ex}

\appendix

\section{Differential graded modules and their minimal models}\label{app:HB}

We construct minimal models of differential graded modules. The material is certainly not new: these kinds of models were introduced in \cite{AvramovHalperin}, \cite{Bibel} (see in particular Exercise 8 of Chapter 6 for the notion of minimality) and were also described in \cite[Appendix B]{AlldayPuppe2}.
Furthermore, the concept of minimality has been discussed in the more general framework of model categories in \cite{roig1}, which applies to the setting of differential graded modules (cf.\ \cite{hinich}). More explicit applications of the general theory to differential graded modules were discussed in \cite{roig2}.
However, the theory needed in this article can be developed rather quickly (the proofs are mainly simplified versions of those for the corresponding statements for cdgas), which we present for the convenience of the reader and to be able to occasionally refer to the explicit constructions.

\subsection{Definitions and properties}

Let $k$ be some ground field and $(R,d)$ be a cdga over $k$.

\begin{defn}\begin{itemize}
\item A graded $R$-module is a graded $k$-vector space $M$ together with an $R$ module structure for which the multiplication map $R\otimes M\rightarrow M$ is a graded map of degree $0$.
\item If a graded $R$-module carries a differential $D$ of degree $1$ such that $D(a\cdot m)=da\cdot m+(-1)^{k}a\cdot Dm$ for any $a\in R^k$, then we call $(M,D)$ a differential graded $R$-module (dg$R$m).
\item A morphism $f\colon (M,D_M)\rightarrow (N,D_N)$ of dg$R$ms is a degree $0$ map which is $R$-linear and commutes with the differentials.
\item A quasi-isomorphism of dg$R$ms is a morphism that induces an isomorphism in cohomology.
\item Two morphisms $f,g\colon (M,D_M)\rightarrow (N,D_N)$ are homotopic if there is an $R$-linear map $h\colon M\rightarrow N$ of degree $-1$ which satisfies $f-g=D_Nh+ hD_M$.
\end{itemize}
\end{defn}

In the definition above we restricted ourselves to morphisms of degree $0$ to simplify the language throughout the article. Also, in what follows we shall work under the following assumptions: the cdga $R$ is concentrated in non-negative degrees and is simply-connected (i.e.\ $R^0=k$ and $R^1=0$). Furthermore all differential graded $R$-modules are assumed to be bounded from below (i.e.\ there is some $k\in\mathbb{Z}$ for which $M^{\leq k}=0$). We denote by $\mathfrak{m}=R^{+}$ the maximal homogeneous ideal in $R$.

\begin{defn}
A dg$R$m $(M,d)$ is minimal if $M=R\otimes V$ is a free module over some graded $k$-vector space $V$ and $\im(d)\subset \mathfrak{m}M$.
\end{defn}

\begin{prop}[Existence]\label{HBconstruction}
For every dg$R$m $(M,D)$ there is a minimal model, i.e.\ a quasi-isomorphism $\varphi\colon (R\otimes V,d)\rightarrow (M,D)$ from a minimal dg$R$m into $(M,D)$.
\end{prop}

\begin{proof}
Assume inductively that for some $n$ we have constructed a map
\[\varphi_n\colon (R\otimes V^{\leq n},d)\longrightarrow (M,D)\] from a minimal dg$R$m, which induces an isomorphism on cohomology in degrees up to $n$ and is injective in degree $n+1$. Let $A=\coker(H^{n+1}(\varphi_n)\colon H^{n+1}(R\otimes V^{\leq n})\rightarrow H^{n+1}(M,D)$, set $d=0$ on $A$, and extend $\varphi_n$ to $R\otimes A$ by linearly choosing representatives from $(\ker D)^{n+1}$.
The resulting map $\varphi_{n+1}'\colon R\otimes (V^{\leq n}\oplus A)\rightarrow (M,D)$ induces an isomorphism in degrees up to $n+1$. To achieve injectivity in degree $n+2$, take $B=\ker H(\varphi_{n+1}')^{n+2}$, considered as a vector space in degree $n+1$, and extend $d$ to $B$ by choosing a basis of $B$ and mapping that basis onto representatives in $(\ker d)^{n+2}$. Setting $V^{n+1}=A\oplus B$, we can then extend $\varphi_{n+1}'$ to a map $\varphi_{n+1}\colon (R\otimes V^{\leq n+1},d)\rightarrow (M,D)$. As $R$ is simply-connected, the introduction of $B$ in degree $n+1$ does not generate any new cohomology in degree $n+2$. Thus $\varphi_{n+1}$ is injective on degree $n+2$ cohomology.
Also observe that $d(B)\subset (R\otimes (V^{\leq n}\oplus A))^{n+2}$ is contained in $\mathfrak{m}(V^{\leq n}\oplus A)$ so this inductive construction indeed yields a minimal model.
\end{proof}

Before discussing the fundamental properties of minimal models, we want to point out that being homotopic is an equivalence relation on the set of morphisms between two dg$R$ms and is furthermore compatible with composition of morphisms. This is easily verified and one of the points where the theory of dg$R$ms is much easier than that of cdgas. Also, we will need the following formulation of homotopy: consider the complex $I=(\langle p_0,p_1,p\rangle_k,d)$, with $p_0$, $p_1$ in degree $0$, $p$ in degree $1$ and $dp_0=p$, $dp_1=-p$, $dp=0$. It comes with two projections $i_j\colon I\rightarrow k$, $j=0,1$ defined by sending $p_j$ to $1$ and the other generators to $0$. A homotopy $h$ between two morphisms $f,g\colon M\rightarrow N$ induces a morphism
\[H\colon M\rightarrow I\otimes N\] of dg$R$ms by setting $H(x)=p\otimes h(x)+p_0\otimes f(x)+p_1\otimes g(x)$. In turn any such $H$ defines a homotopy between $(i_0\otimes \Id_N)\circ H$ and $(i_1\otimes \Id_N)\circ H$.

\begin{prop}[Lifting]\label{lifting}
Let $f\colon (N,D_N)\rightarrow (M,D_M)$ be a quasi-isomorphism and $\varphi\colon(R\otimes V,d)\rightarrow (M,D_M)$ a morphism from a minimal dg$R$m.
\begin{enumerate}[(i)]
\item If $f$ is surjective there is $\tilde{\varphi}\colon (R\otimes V,d)\rightarrow (N,D_N)$ such that $f\circ \tilde{\varphi}=\varphi$.
\item There exists $\tilde{\varphi}\colon (R\otimes V,d)\rightarrow (N,D_N)$, unique up to homotopy, such that $f\circ \tilde{\varphi}$ is homotopic to $\varphi$.
\end{enumerate}
\end{prop}

\begin{proof}
We prove the existence of $\tilde{\varphi}$ in $(i)$.
Assume inductively that we have constructed the lift $\tilde{\varphi}\colon R\otimes V^{\leq n}\rightarrow (N,D_N)$ for some $n$ such that $\varphi|_{R\otimes V^{\leq n}}=f\circ\tilde{\varphi}$. For any $x$ in a fixed basis of  $V^{n+1}$, we have $dx\in R\otimes V^{\leq n}$ due to the fact that $(R\otimes V,d)$ is minimal and $R$ is simply-connected. As a result, $\tilde{\varphi}(dx)$ is defined and closed. We have $f(\tilde{\varphi}(dx))=\varphi(dx)=D_M\varphi(x)$ so $\tilde{\varphi}(dx)$ is exact because $f$ is injective on cohomology. Choose $y\in N$ with $D_N y=\tilde{\varphi}(dx)$ and observe that $f(y)-\varphi(x)\in\ker D_M$. We claim that we find $z\in \ker D_N$ with $f(z)=f(y)-\varphi(x)$. If we do, we can extend $\tilde{\varphi}$ by setting $\tilde{\varphi}(x)=y-z$ which completes the induction. Since $f$ is surjective on cohomology, we find some $a\in \ker D_N$, $b\in M$ such that $f(a)=f(y)-\varphi(x)+D_Mb$. Also $f$ is surjective so we find $c\in N$ with $f(c)=b$. Thus the element $z=a-D_N c$ has the desired properties.

The existence in $(ii)$ follows from $(i)$ in the following way: Consider the acyclic module $(M\oplus \delta M,\delta)$, where the differential is an isomorphism $\delta\colon  M\cong \delta M$ and vanishes on $\delta M$. This maps surjectively onto $(M,D_M)$ by $m\mapsto m$, $\delta m\mapsto D_Mm$. So we obtain a surjective quasi-isomorphism $(N\oplus M\oplus\delta M,D_N\oplus\delta)\rightarrow (M,D_M)$. Now $(i)$ yields the dashed arrow in the diagram

\[\xymatrix{ & N\oplus M\oplus \delta M\ar[d]\ar[r] & N\ar[dl]\\
R\otimes V\ar@{-->}[ru]\ar[r] & M &
}\]
in which the left hand triangle commutes and the right hand triangle commutes up to homotopy.
Composing with the arrow to $N$ yields the desired lift.

Finally we argue that the lift is unique up to homotopy, which will be achieved by lifting homotopies. We start by the observation that two maps $\tilde{\varphi}_1,\tilde{\varphi}_2\colon R\otimes V\rightarrow N$ are homotopic if and only if their compositions with $N\rightarrow N\oplus M\oplus \delta M$ are homotopic. This holds because the latter map is a homotopy equivalence with the top horizontal arrow of the above diagram as a homotopy inverse. Thus in what follows, we can assume $f$ to be surjective.

Consider the fiber product $F=(I\otimes M)\times_{M\oplus M} (N\oplus N)$, which is explicitly given by
\[F=\{(x,y,z)\in (I\otimes M)\oplus N\oplus N~|~ (i_0\otimes \Id_M (x),i_1\otimes \Id_M (x))=(f(y),f(x))\}.\]
The map $\psi=(\Id_I\otimes f\oplus i_0\otimes \Id_N\oplus i_1\otimes \Id_N)\colon I\otimes N\rightarrow F$ defines a surjective quasi-isomorphism. This follows from a straight forward investigation of what it means to be a (exact) cocycle in the two objects, which we leave to the reader. Now if $\tilde{\varphi}_1,\tilde{\varphi}_2\colon R\otimes V\rightarrow N$ are two morphisms such that $f\circ \tilde{\varphi}_1$ and $f\circ \tilde{\varphi}_2$ are homotopic via a homotopy $H\colon R\otimes V\rightarrow I\otimes M$, then we obtain the map $H\oplus \tilde{\varphi_1}\oplus \tilde{\varphi_2}\colon R\otimes V\rightarrow F$. By $(i)$ we can lift this map through the surjective quasi-isomorphism $\psi$ which yields a homotopy between $\tilde{\varphi_1}$ and $\tilde{\varphi_2}$.
\end{proof}

\begin{prop}[Uniqueness]\label{uniqueness}
A quasi-isomorphism between minimal dg$R$ms is an isomorphism.
\end{prop}
\begin{proof}
Let $\varphi\colon (R\otimes V,d)\rightarrow (R\otimes V',d')$ be a quasi-isomorphism of minimal dg$R$ms. This induces a map $\overline{\varphi}\colon V\cong (R\otimes V)/\mathfrak{m}V\rightarrow (R\otimes V')/\mathfrak{m}V'\cong V'$. We show that $\overline{\varphi}$ is an isomorphism which implies that $\varphi$ is an isomorphism as well.

Lifting $\Id_{R\otimes V'}$ through $\varphi$ provides us with a morphism $\psi\colon (R\otimes V',d')\rightarrow (R\otimes V,d)$ such that $\varphi\circ\psi$ is homotopic to $\Id_{R\otimes V'}$.  By the definition of homotopy and minimality, it follows that $\varphi\circ\psi- \Id_{R\otimes V'}$ takes values in $\mathfrak{m}V'$. In particular $\overline{\varphi}\circ \overline{\psi}$ is an isomorphism, where $\overline{\psi}\colon V'\rightarrow V$ is defined analogously.

Since $\psi\circ \varphi \circ \psi \circ\varphi \simeq \psi\circ \varphi$, the uniqueness in Proposition \ref{lifting} implies that also $\psi\circ \varphi\simeq \Id_{R\otimes V}$. Hence, as before, we deduce that $\overline{\psi}\circ\overline{\varphi}$ is an isomorphism as well. Consequently $\overline{\varphi}$ is an isomorphism.
\end{proof}

Together, Propositions \ref{lifting} and \ref{uniqueness} imply the following

\begin{cor}
Two dg$R$ms are connected by a chain of quasi-isomorphisms (that is to say they have the same quasi-isomorphism-type) if and only if they have isomorphic minimal models.
\end{cor}

Let $(R,0)\rightarrow(R\otimes \Lambda V,D)\rightarrow(\Lambda V,d)$ be a relative minimal Sullivan model for the Borel fibration of an action of a compact Lie-group $G$ on a space $X$, where $R=H^*(BG)$ is a polynomial ring and $(\Lambda V,d)$ is the minimal model of $X$.

\begin{defn}
The minimal dg$R$m-model of $(R\otimes \Lambda V,d)$ is called the (minimal) Hirsch-Brown model of the action.\end{defn}

\begin{rem}
One can show that it has the form $(R\otimes H^*(X),D)$ (see e.g. \cite[Cor. B.2.4]{AlldayPuppe2}).
\end{rem}

\section{Strong homotopy algebras and modules}\label{secstr}

We give a brief summary of some basic definitions and results surrounding $A_\infty$-algebras and modules. The aim here is to access the results we need while staying as elementary as possible. For a more detailed introduction to the subject see e.g.\ \cite{proute}, \cite{lefevre}, or alternatively \cite{LodayValette} for the operadic viewpoint which we will occasionally refer to. For a broader overview on the subject see \cite{keller}. In what follows, all vector spaces are considered over a field $k$ of {characteristic 0}. We will make use of the Koszul sign convention:

\[(f\otimes g)(a\otimes b)=(-1)^{|g|\cdot |a|}f(a)\otimes g(b)\]
if $f$ and $g$ are graded linear maps and $a,b$ are homogeneous elements from the respective domains of $f$ and $g$. Also the general assumptions that all (commutative) differential graded algebras are non-negatively graded and unital will be suspended within this section (unless stated otherwise).

\subsection{$A_\infty$- and $C_\infty$-algebras}\label{sec:ainftycinfty}

\begin{defn}\label{def:Ainfty}
An $A_\infty$-algebra is a graded vector space $A$ together with linear maps $m_i\colon A^{\otimes i}\rightarrow A$ of degree $2-i$ for each $i\geq 1$, satisfying for each $n$ the relation
\[\sum (-1)^{jk+l}m_{i}({\bf{1}}^{\otimes j}\otimes m_k\otimes {\bf{1}}^{\otimes l})=0\]
where the sum runs over all decompositions $j+k+l=n$ with $k\geq 1$ and $i=j+l+1$.
\end{defn}

While the equations may look complicated at first glance, they take a familiar form if $n$ is small: for $n=1$ we obtain $m_1^2=0$ so $m_1$ is a differential. For $n=2$ we obtain the statement that $m_1$ is a derivation with respect to the binary product $m_2$. In general $m_2$ is not associative but it is so up to a homotopy given by $m_3$ in the equation for $n=3$.

\begin{defn}
The cohomology of $(A;m_i)$ is the cohomology of the chain complex $(A,m_1)$.
\end{defn}

By the discussion above, the product $m_2$ induces a product on cohomology. This product is associative giving the cohomology the structure of a graded algebra.

\begin{rem}\label{rem:inclusiondga} Any ordinary differential graded algebra $(A,d)$ can be considered as an $A_\infty$-algebra by taking $m_1$ to be $d$, $m_2$ the multiplication in $A$, and $m_i=0$ for $i\geq 3$.
\end{rem}

Before we turn our attention to the definition of a morphism between $A_\infty$-algebras, let us reinterpret the defining equations. We define the suspension $sA$ of $A$ via $(sA)^n=A^{n+1}$ and also denote by $s\colon A\rightarrow sA$ the canonical isomorphism of degree $-1$. Consider the reduced tensor coalgebra
\[\overline{T}sA=sA\oplus (sA\otimes sA)\oplus\ldots\]
and the map $\overline{T}sA\rightarrow  sA$ of degree $1$ which, on $n$-tensors, we define to be $b_i=-s^{-1}\circ m_i\circ s^{\otimes i}$. This map extends uniquely to a coderivation $b$ of $\overline{T}sA$ by setting
\[b|_{sA^{\otimes n}}=\sum_{i+k+l=n} \Id_{sA}^{\otimes k}\otimes b_i\otimes \Id_{sA}^{\otimes l}.\]
The equations in Definition \ref{def:Ainfty} are equivalent to the fact that $b^2=0$. In particular there is a one-to-one correspondence between $A_\infty$-structures on $A$ and codifferentials of the coalgebra $\overline{T}sA$.

\begin{defn}
The differential graded coalgebra $(\overline{T}sA,b)$ is called the bar construction of $(A;m_i)$ and denoted $BA$.
\end{defn}

\begin{rem}
The transition from the $m_i$ to the $b_i$ is not canonical and there exist different sign conventions in the literature, giving rise to different signs in the definition of $A_\infty$-algebra. We stick to the ones used e.g.\ in \cite{lefevre} and \cite{markl}.
\end{rem}

A morphism between $A_\infty$-algebras $(A;m_i)$ and $(C;m_i)$ can just be defined as a morphism of differential graded coalgebras $f\colon BA\rightarrow BC$. By the universal property of the cofree coalgebra, such a morphism is defined by the projection $\overline{f}\colon BA\rightarrow sC$. This data is equivalent to a collection of maps $f_i\colon A^{\otimes i}\rightarrow C$ of degree $i-1$ such that on $i$-tensors, $\overline{f}$ is given by $s^{-1}\circ f_i\circ s^{\otimes i}$. The condition of $f$ being a morphism of differential graded coalgebras translates, for each $n$, to the equation
\[\sum (-1)^{jk+l}f_i(\Id^{\otimes j}\otimes m_k\otimes \Id^{\otimes l})=\sum (-1)^{s}m_k(f_{i_1}\otimes\ldots\otimes f_{i_k}),
\]
where the left hand sum runs over all decompositions $n=j+k+l$ with $k\geq 1$ and $i=j+1+l$ and the right hand sum runs over all decompositions $n=i_1+\ldots+i_k$ and
\[s=\sum_{1\leq \alpha<\beta\leq k} (i_\alpha)(i_\beta+1)\]
which we could take as an alternative definition of morphism. The equations show that $f_1\colon A\rightarrow B$ is always a chain map with respect to the differentials $m_1$ on $A$ and $C$. Therefore it induces a map $H^*(A)\rightarrow H^*(C)$. We call $f$ a quasi-isomorphism if $f_1$ induces an isomorphism  on cohomology.

We briefly introduce the concept of homotopy for morphisms of $A_\infty$-algebras. Consider the dg-coalgebra $(I,d)$ where $I$ has basis $e$ in degree $-1$ and $e_0,e_1$ in degree $0$. The differential is defined by $d(e)=e_0-e_1$ and the coalgebra structure is defined by $\Delta e=e_0\otimes e+e\otimes e_1$, $\Delta e_0=e_0\otimes e_0$, and $\Delta e_1=e_1\otimes e_1$. If $A$ is an $A_\infty$-algebra, we may form the coalgebra $BA\otimes I$ together with the two inclusions $i_0$ and $i_1$ mapping $BA$ to $BA\otimes e_0$ and $BA\otimes e_1$ and the projection $p\colon BA\otimes I\rightarrow BA$ which maps $x\otimes e+x_0\otimes e_0+x_1\otimes e_1$ to $x_0+x_1$. Then $p$ is a quasi-isomorphism and $p\circ i_k=\Id_{BA}$ (see \cite[1.3.4.1]{lefevre}).

\begin{defn}\label{def:homotopy}
We say two morphisms $f,g\colon A\rightarrow C$ are homotopic if there is a morphism of coalgebras $h\colon BA\otimes I\rightarrow BC$ with $h\circ i_0=f$ and $h\circ i_1=g$.
\end{defn}

There is also a commutative version of $A_\infty$-algebras defined as follows:

\begin{defn}\label{Def:Cinfty} Let $(A;m_i)$ be an $A_\infty$-algebra.
\begin{enumerate}[(i)]
\item The shuffle product on $\overline{T}sA$ is defined by
\[(a_1\otimes \ldots\otimes a_n)*_{sh}(a_{n+1}\otimes\ldots\otimes a_{n+k})=\sum_{\sigma\in Sh(n,k)} (-1)^{s(\sigma)} a_{\sigma(1)}\otimes \ldots\otimes a_{\sigma(n+k)}\]
where $sh(n,k)$ consists of all permutations $\sigma$ with $\sigma(i)<\sigma(j)$ whenever $i<j$ and either $1\leq i,j\leq n$ or $n+1\leq i,j\leq n+k$. The sign $s(\sigma)$ is given by the usual sign of the action of the symmetric group on the graded space $(sA)^{\otimes n+k}$.
\item $(A;m_i)$ is called a $C_\infty$-algebra if the maps $m_i\circ (s^{-1})^{\otimes i}$ vanish on all shuffles in $\overline{T}sA$ (note that applying $(s^{-1})^{\otimes{i}}$ to elements of word length $i\geq 2$ changes signs by the Koszul sign rule).
\item A morphism of $C_\infty$-algebras is a morphism $f$ of $A_\infty$-algebras such that the maps $f_i\circ(s^{-1})^{\otimes i}$ vanish on all shuffles in $\overline{T}sA$.
\end{enumerate}

\begin{rem}
From an operadic viewpoint, the commutative version of an $A_\infty$-algebra would be to consider the free Lie coalgebra instead of the free coalgebra as we did in the bar construction of an $A_\infty$-algebra (cf. \cite{StasheffSchlesinger}). One can prove that the definition of a $C_\infty$-algebra above is equivalent to such a structure (see \cite[Prop.\ 13.1.14]{LodayValette}).
\end{rem}

\end{defn}

We conclude this section by introducing the unital and augmented versions of the previous structures. We consider the ground field $k$ as a cdga concentrated in degree $0$. All of the following notions for $C_\infty$-algebras have obvious analogous definitions for $A_\infty$-algebras which we will not repeat.

\begin{defn}\begin{enumerate}[(i)]
\item A strictly unital $C_\infty$-algebra is a $C_\infty$-morphism $\eta\colon  k\rightarrow (A;m_i)$ such that \[m_i(\Id_A\otimes \ldots\otimes \Id_A \otimes \eta\otimes \Id_A\otimes\ldots\otimes \Id_A)=0\] for $i\geq 3$ and $m_2(\Id_A\otimes \eta)=m_2(\eta\otimes \Id_A)=\Id_A$. A morphism of strictly unital $C_\infty$-algebras is a morphism of the underlying $C_\infty$-algebras that commutes with the units.
\item An augmented $C_\infty$-algebra is a strictly unital $C_\infty$-algebra $A$ together with a strict morphism $\varepsilon\colon A\rightarrow k$ of strictly unital $C_\infty$-algebras. By $\varepsilon$ being strict we mean that the only nontrivial component is $\varepsilon_1$. A morphism of augmented $C_\infty$-algebras is a morphism of the underlying strictly unital $C_\infty$-algebras that commutes with the augmentations.
\end{enumerate}
\end{defn}

We point out that there is an equivalence of categories between $C_\infty$-algebras and augmented $C_\infty$-algebras. In the presence of an augmentation $\varepsilon$, one can consider the non-augmented $C_\infty$-algebra $\ker \varepsilon$ with the induced structure. Conversely, for any $C_\infty$-algebra $(A;m_i)$, we can consider $A\oplus k$ with the unique strictly unital $C_\infty$-structure that agrees with the $m_i$ on $A$, where the unit and augmentation are given by the canonical inclusion and projection of $k$. The two constructions are naturally inverse to another. All notions such as quasi-isomorphisms or homotopy carry over to the augmented setting.

\begin{rem}\label{rem:augmented} If $A$ is strictly unital, concentrated in positive degrees and $A^0=k$, then it is naturally augmented by projecting onto the degree $0$ component.
\end{rem}

\subsection{$A_\infty$-modules}

\begin{defn}
Let $(A;m_i)$ be an $A_\infty$-algebra. An $A_\infty$-$A$-module structure on a graded vector space $M$ is a collection of maps $m_i^M\colon M\otimes A^{\otimes i-1}\rightarrow M$ of degree $2-i$ for each $i\geq 1$, satisfying for each $n$ the relation
\[\sum (-1)^{jk+l}m_{i}^M({\bf{1}}^{\otimes j}\otimes m_k\otimes {\bf{1}}^{\otimes l})=0\]
where the sum runs over all decompositions $j+k+l=n$ with $k\geq 1$, $i=j+l+1$, and the expression ${\bf{1}}^{\otimes j}\otimes m_k\otimes {\bf{1}}^{\otimes l}$
is interpreted as $m_k^M\otimes {\bf{1}}^{\otimes l}$ whenever $j=0$.
\end{defn}

Let $TsA=k\oplus \overline{T}sA$ be the free augmented coalgebra with comultiplication \[\Delta_{TsA}(x)=x\otimes 1+\Delta_{\overline{T}sA}(x)+1\otimes x.\] As before, the above data defines a map $sM\otimes TsA\rightarrow sM$ of degree $1$  which we define to be $-s\circ m_i\circ (s^{-1})^{\otimes i}$ on $sM\otimes sA^{\otimes i-1}$.
Said map can be extended uniquely to a differential $d$ on the cofree comodule $sM\otimes TsA$ as in the following

\begin{lem}\label{modcoderivations}
Let $V$ be a vector space and $(C,\Delta_C, d_C)$ a differential graded coalgebra. Then any linear map $\overline{d}\colon V\otimes C\rightarrow V$ can be coextended uniquely to a coderivation $d$ of the co-free comodule $V\otimes C$ by setting

\[d=\Id_V\otimes d_C+(\overline{d}\otimes \Id_C)\circ (\Id_V\otimes\Delta_C).\] For counitary coalgebras this yields a one-to-one correspondence of maps $\overline{d}$ and coderivations $d$.
\end{lem}

\begin{defn}
The cofree differential graded comodule $(sM\otimes TsA,d)$ is called the bar construction of $(M;m_i^M)$ and denoted $BM$.
\end{defn}

A morphism of $A_\infty$-modules is defined as a dg-comodule map $f$ between the respective bar constructions. Such a map $f\colon BM\rightarrow BN$ is determined by the projection \[\overline{f}\colon sM\otimes TsA\rightarrow sN\]
and can thus be expressed by a collection of maps $f_i\colon M\otimes A^{\otimes i-1}\rightarrow N$ of degree $1-i$ such that $\overline{f}=s\circ f_i\circ (s^{-1})^{\otimes i}$ on word length $i$. The fact that $f$ commutes with the differentials of the bar constructions translates to the equations

\[\sum (-1)^{jk+l}f_i(\Id^{\otimes j}\otimes m_k\otimes \Id^{\otimes l})=\sum m_{k+1}(f_{i}\otimes \Id^{\otimes k})
\]
for every $n\geq 1$, where the left hand sum runs over all decompositions $n=j+k+l$, $k\geq 1$ and the right hand sum runs over all decompositions $n=i+k$, $i\geq 1$.
From this, $f$ is reconstructed by just setting $f=(\overline{f}\otimes \Id_{TsA})\circ (\Id_{sM}\otimes \Delta_{TsA})$.

As for $A_\infty$-algebras, the operation $m_1^M$ of an $A_\infty$-module $(M;m_i^M)$ is a differential and we define the cohomology of $M$ as that of the chain complex $(M,m_1^M)$. If $f\colon (M;m_i^M)\rightarrow (N;m_i^N)$ is a morphism of $A_\infty$-modules, then the above equation for $n=1$ implies that$f_1$ is a chain map. Thus we obtain a map $H^*(M)\rightarrow H^*(N)$ on cohomology. We call $f$ a quasi-isomorphism if the map on cohomology is an isomorphism.

\begin{rem}\label{rem:inclusiondgRm}
As for $A_\infty$-algebras, if $(A,d)$ is a regular dga, the category of classical dg-modules over $A$ can be seen as a subcategory of $A_\infty$-$A$-modules. Just define the operation $m_1^M$ to be the differential, $m_2^M$ to be the multiplication map of the module, and $m_i^M=0$ for $i\geq 3$.
\end{rem}

\begin{rem}
If $(A;m_i)$ is an $A_\infty$-algebra, it is in particular an $A_\infty$-module over itself because $\overline{T}sA=A\otimes TsA$.
\end{rem}

We discuss how $A_\infty$-modules behave with respect to maps of the underlying $A_\infty$-algebras. Analogous to the case of regular modules, a morphism $f\colon (A;m_i)\rightarrow (C;m_i)$ of $A_\infty$-algebras yields a restriction functor from $A_\infty$-$C$-modules to $A_\infty$-$A$-modules. If $(M;m_i^M)$ is a $C$-module and $f_i\colon A^{\otimes i}\rightarrow C$ are the components of $f$, then the $A$-module $f^*M$ is defined by

\[m_i^{f^*M}=\sum(-1)^{s} m_{r+1}^M(\Id_M\otimes f_{i_1}\otimes\ldots\otimes f_{i_r})
\]
where the sum runs over all decompositions $i-1=i_1+\ldots+{i_r}$ for $r=1,\ldots,i-1$ and \[s=\sum_{1\leq \alpha<\beta\leq r} i_\alpha(i_\beta+1)+\sum_{j=1}^r (i_j-1).\]

In the language of the corresponding bar constructions $BM$ and $Bf^*M$, this can be understood as the pullback of the differential along the map $f\colon BA\rightarrow BC$ in the following sense: the differential on $BM=(sM\otimes TsC,d)$ is the coextension of the projection $\overline{d}\colon sM\otimes TsC\rightarrow sM$ (see Lemma \ref{modcoderivations}). Then the corresponding differential of $Bf^*M=(sM\otimes TsA)$ is the coextension of $\overline{d}\circ (\Id_{sM}\otimes f^+)$ where $f^+=(\Id_k,f)\colon k\oplus \overline{T}sA\rightarrow k\oplus \overline{T}sC$.

\begin{lem}\label{lem:trianglemorphism}
Consider a commutative triangle
\[\xymatrix{
(R;m_i)\ar[r]^f\ar[rd]^h & (A;m_i)\ar[d]^g\\ & (C;m_i)
}\]
of $A_\infty$-Algebras. Then $g$ induces an $A_\infty$-module homomorphism between the $A_\infty$-$R$-module structures induced on $A$ and $C$ by the morphisms $f$ and $h$.
\end{lem}

\begin{proof}
Let $d_A$, $d_C$, $d_R$ be the differentials of the bar constructions $BA$, $BC$, $BR$ of $A_\infty$-algebras and let $d_A^+$, $d_C^+$, $d_R^+$ denote the extensions to $TsA$, $TsC$, $TsR$ by sending $k$ to $0$. Also denote by $\pi_A \colon BA\rightarrow sA$ and $\pi_C\colon BC\rightarrow sC$ the projections and set $\overline{g}=\pi_C\circ g$ as well as $\overline{d_A}=\pi_A\circ d_A$, $\overline{d_C}=\pi_C\circ d_C$.

We want to define a map ${\varphi}$ between the bar constructions of $A$ and $C$ when considered as $A_\infty$-modules over $(R;m_i)$. Explicitly, the bar constructions are given as $(sA\otimes TsR,D_A)$, $(sC\otimes TsR,D_C)$, where $D_A$ and $D_C$ are the coextensions (in the sense of Lemma \ref{modcoderivations}) of
\[\overline{D_A}=\overline{d_A}\circ(\Id_{sA}\otimes f^+)\quad\text{and}\quad\overline{D_C}=\overline{d_C}\circ(\Id_{sC}\otimes h^+).\]
We define $\overline{\varphi}\colon  sA\otimes TsR\rightarrow sC$ to be the composition $\overline{g}\circ(\Id_{sA}\otimes f^+)$ and define $\varphi$ to be the coextension of $\overline{\varphi}$, that is $\varphi=(\overline{\varphi}\otimes \Id_{TsR})\circ (\Id_{sA}\otimes \Delta_{TsR})$. It remains to check that $\varphi$ commutes with the differentials, which can be confirmed on cogenerators: we need to prove that $\overline{D_C}\circ\varphi=\overline{\varphi}\circ D_A$ and calculate
\begin{align*}
\overline{\varphi}\circ D_A&=\overline{g}\circ (\Id_{sA}\otimes f^+)\circ (\Id_{sA}\otimes d^+_{R}+(\overline{D_A}\otimes \Id_{TsR})\circ (\Id_{sA}\otimes \Delta_{TsR}))\\
&=\overline{g}\circ (\Id_{sA}\otimes (f^+\circ d^+_R)+(\overline{d_A}\otimes \Id_{TsA})\circ (\Id_{sA}\otimes f^+\otimes f^+)\circ (\Id_{sA}\otimes \Delta_{TsR})\\
&=\overline{g}\circ(\Id_{sA}\otimes d_A^+ +(\overline{d_A}\otimes \Id_{TsA})\circ(\Id_{sA}\otimes \Delta_{TsA}))\circ (\Id_{sA}\otimes  f^+)\\
&=\overline{g}\circ d_A\circ (\Id_{sA}\otimes f^+)\\
&=\overline{d_C}\circ g\circ (\Id_{sA}\otimes f^+)\\
&=\overline{d_C}\circ(\overline{g}\otimes g^+)\circ (\Id_{sA}\otimes \Delta_{TsA})\circ (\Id_{sA}\otimes f^+)\\
&=\overline{d_C}\circ(\overline{g}\otimes g^+)\circ (\Id_{sA}\otimes f^+\otimes f^+)\circ (\Id_{sA}\otimes \Delta_{TsR})\\
&=\overline{d_C}\circ (\Id_{sC}\otimes h^+)\circ (\overline{\varphi}\otimes \Id_{TsR})\circ (\Id_{sA}\otimes \Delta_{TsR})\\
&= \overline{D_C}\circ \varphi,
\end{align*}
where we have identified $\overline{T}sA=sA\otimes TsA$ and $\overline{T}sC=sC\otimes TsC$.
\end{proof}

\begin{lem}\label{lem:homotopicrestriction}
Let $f,g\colon (A;m_i)\rightarrow (C;m_i)$ be two homotopic morphisms between $A_\infty$-algebras. Then the two induced $A$-module structures on $C$ are quasi-isomorphic.
\end{lem}

For the proof we will need the following result (\cite[Section 6.2]{keller}, \cite[Théorème 4.1.2.4]{lefevre}). The notion of homotopy category is recalled in the next section.

\begin{thm}
Let $f\colon (A;m_i)\rightarrow (C;m_i)$ be a quasi-isomorphism of $A_\infty$-algebras. Then the restriction functor defines an equivalence between the homotopy categories $\mathcal{D}_\infty A$ and $\mathcal{D}_\infty C$ of the categories of $A_\infty$-modules over $A$ and $C$.
\end{thm}

\begin{proof}[Proof of the lemma.]
With the terminology of Definition \ref{def:homotopy}, it suffices to show that the restrictions of the $BA\otimes I$-module $h^*C$ along $i_0$ and $i_1$ are quasi-isomorphic, where $h$ is a homotopy between $f$ and $g$. Since $p$ is a quasi-isomorphism, restriction along $p$ defines an equivalence $\mathcal{D}_\infty BA\rightarrow \mathcal{D}_\infty BA\otimes I$ so $h^*C$ is quasi-isomorphic to $p^*M$ for some dg-$BA$-comodule $M$. But this means that $i_0^*(h^*C)$ is quasi-isomorphic to $i_0^*(p^*M)=M$ and the same holds for $i_1$.
\end{proof}

\subsection{On the homotopy categories}\label{sec:barcobarres}

Let $\mathcal{C}$ be one of the categories of augmented dgas, augmented cdgas, or dg-$A$-modules, where $A$ is a fixed dga with unit. Also let $\mathcal{C}_\infty$ be the corresponding category of either augmented $A_\infty$-algebras, augmented $C_\infty$-algebras, or strictly unital $A_\infty$-modules over $A$. The latter is defined as the full subcategory of $A_\infty$-modules $(M;m_i^M)$ over $A$ such that $m_i(\Id_M\otimes\Id_A\otimes\ldots\otimes\Id_A\otimes \eta\otimes\Id_A\otimes\ldots\otimes \Id_A)=0$ for $i\geq 3$ and $m_2(\Id_M\otimes \eta)=\Id_M$, where $\eta\colon k\rightarrow A$ is the unit of $A$.
We have seen that there is an inclusion of categories $\mathcal{C}\rightarrow \mathcal{C}_\infty$ (see Remarks \ref{rem:inclusiondga} and \ref{rem:inclusiondgRm}). The category $\mathcal{C}_\infty$ is much larger in the sense that it contains many new objects and morphisms. However, it does not introduce new quasi-isomorphism types. Hence $\mathcal{C}_\infty$ can be very useful when studying the category $\mathcal{C}$ up to quasi-isomorphism.

Let $\mathrm{Ho}(\mathcal{C})$ and $\mathrm{Ho}(\mathcal{C}_\infty)$ denote the homotopy categories of $\mathcal{C}$ and $\mathcal{C}_\infty$ which are defined as the localizations of the respective category at the set of quasi-isomorphisms. The categories $\mathcal{C}$ and $\mathrm{Ho}(\mathcal{C})$ have the same objects and two of them are isomorphic in $\mathrm{Ho}(\mathcal{C})$ if and only if they are connected by a zigzag of quasi-isomorphisms in $\mathcal{C}$ (analogous for $\mathcal{C}_\infty$).

\begin{thm}\label{thm:homotopycategories}
The inclusion $\mathcal{C}\rightarrow \mathcal{C}_\infty$ induces an equivalence of categories $\mathrm{Ho}(\mathcal{C})\simeq \mathrm{Ho}(\mathcal{C}_\infty)$.
\end{thm}

For $A_\infty$- and $C_\infty$-algebras this follows from \cite[Thm. 11.4.12]{LodayValette} which covers the general setting of homotopy algebras over Koszul operads and naturally carries over to the augmented case through the equivalence at the end of Section \ref{sec:ainftycinfty}. The case of $A_\infty$-modules over $A$ was proved in \cite[Lemma 4.1.3.8, Prop.\ 3.3.1.8]{lefevre}. A similar result on $A_\infty$-modules is more explicitly stated in \cite[Section 4.3]{keller}, however referring to the previous reference for proofs.

\subsection{Minimal models and formality}

\label{sec:minmodsec}

It has long been known that Massey products form an obstruction to formality. Conversely, the uniform vanishing of all Massey products implies formality. This idea of uniform vanishing is best captured by seeing Massey products as the higher operations in an $A_\infty$-structure on the cohomology. The following theorem goes back to \cite{kadeishvili2} and \cite{kadeishvili1.5}. It has since been generalized in the language of operads \cite[Theorem 10.3.15]{LodayValette}.

\begin{thm}
Let $(A;m_i^A)$ be an $A_\infty$-Algebra (resp.\ $C_\infty$-algebra). Then there is an $A_\infty$- (resp.\ $C_\infty$-)algebra structure $(H^*(A);m_i)$ on the cohomology such that
\begin{itemize}
\item $m_1=0$ and $m_2$ is the product induced by $m_2^A$.
\item there is a quasi-isomorphism $(H^*(A);m_i)\rightarrow (A;m_i)$ of $A_\infty$- (resp.\ $C_\infty$-)algebras lifting the identity on cohomology.
\end{itemize}
This structure is unique in the sense that two of them are isomorphic via an isomorphism of $A_\infty$- (resp.\ $C_\infty$-)algebras whose first component is the identity.
\end{thm}

We will refer to the quasi-isomorphism $H^*(A)\rightarrow A$ as a minimal model for $A$. More generally we will call an $A_\infty$-(resp.\ $C_\infty$-)algebra minimal if the operation $m_1$ vanishes. The operations $m_i$, $i\geq 3$ of the minimal model are also referred to as the higher Massey products. There are several known formulas which compute these operations from the $A_\infty$-structure on $A$ (see e.g.\ \cite{merkulov}, \cite{LodayValette}). We quickly recall the original construction by Kadeishvili (\cite[Theorem 1]{kadeishvili3}) of how to compute the minimal $A_\infty$-model of a dga. It was shown in \cite{markl} that the same construction yields a minimal $C_\infty$-model when applied to a cdga.

Let $(A,d)$ be a (c)dga. Set $m_1=0$ and let $f_1\colon H^*(A)\rightarrow A$ be a cycle choosing homomorphism.
Assume inductively that $f_i$ and $m_i$ have been constructed until $i=n-1$. Define the operator $U_n\colon H^*(A)^{\otimes n}\rightarrow A$ as $U_n^1+U_n^2$ where
\begin{align*}
U_n^1&=\sum_{k=1}^{n-1} (-1)^{k(n+k+1)} f_k\cdot f_{n-k}\\
U_n^2&=-\sum_{k=2}^{n-1}\sum_{i=0}^{n-k}(-1)^{ik+n+k+i}f_{n-k+1}(\Id_A^{\otimes i} \otimes m_k \otimes \Id_A^{\otimes n-i-k}).
\end{align*}
One can check that $U_n$ maps to cocycles and we define $m_n:=[U_n]$. Now choose $f_n$ in a way that $df_n=f_1m_n-U_n$.

\begin{rem}\label{rem:minmod}
If $(A,d)$ is a formal cdga, then it has a minimal Sullivan model $(\Lambda V,d)$ with a splitting $V=W_1\oplus W_2$ such that $d(W_1)=0$ and any closed element in the ideal generated by $W_2$ is exact. In the construction of the minimal $C_\infty$-model, $f_n$ $(n\geq 2$) can be chosen to have image in the ideal generated by $W_2$ which yields $m_n=0$ for $n\geq 3$.
\end{rem}

The converse statement of the remark is also true (see e.g.\ \cite[Theorem 8]{kadeishvili}):

\begin{thm}\label{thm:formalityofcdga}
A cdga $(A,d)$ is formal if and only if it has a minimal $C_\infty$-model of the form $(H^*(A);m_i)$ with $m_i=0$ for $i\neq 2$.
\end{thm}

We will also make use of the following observations. A minimal model as in $(ii)$ will be called a unital minimal model.
\begin{lem}\label{lem:inftyminmodkram}
\begin{enumerate}[(i)]
\item Let $\varphi\colon A\rightarrow B$ be a cohomologically injective morphism of (c)dgas. Then for arbitrary choices of the $f_n^A$ in the construction of the minimal model $(H^*(A);m_i^A)$ above, we can choose the maps $f_n^B$ for the construction of $(H^*(B);m_i^B)$ in a way that $\varphi^*\circ m_n^A=m_n^B\circ(\varphi^*)^{\otimes n}$ and $\varphi\circ f_n^A=f_n^B\circ(\varphi^*)^{\otimes n}$.
\item The $A_\infty$- ($C_\infty$-)minimal model of a unital $(c)dga$ $A$ can be constructed in a way that $H^*(A)\rightarrow A$ is a morphism of strictly unital $A_\infty$- ($C_\infty$-)algebras. If $A$ is formal, then the construction of the unital minimal model is compatible with Remark \ref{rem:minmod}.
\item If $\varphi\colon A\rightarrow B$ is a cohomologically injective morphism of unital (c)dgas and $A$ is formal, then the unital minimal model of $B$ can be constructed such that $m_n^B\circ (\varphi^*)^{\otimes n}$ vanishes for $n\geq 3$.
\end{enumerate}
\end{lem}

\begin{proof}
In the situation of $(i)$ we choose any cycle choosing homomorphisms $f_1^A$. Due to the injectivity of $\varphi^*$ we can define $f_1^B$ in a way that $\varphi\circ f_1^A=f_1^B\circ \varphi^*$. Assume inductively that we have constructed $m_{n-1}^A$, $m_{n-1}^B$, $f_{n-1}^A$, and $f_{n-1}^B$ such that the statement of the lemma holds. By the formulas in the construction we also have $\varphi\circ U_n^A=U_n^B\circ (\varphi^*)^{\otimes n}$. As a result \[\varphi(f_1^A m_n^A-U_n^A)=(f_1^B m_n^B-U_n^B)\circ (\varphi^*)^{\otimes n}.\]
Now for any choice of $f_n^A$ we can define $f_n^B$ on $\im(\varphi^*)^{\otimes n}$ to be $\varphi\circ f_n^A\circ((\varphi^*)^{\otimes n})^{-1}$ and complete the definition arbitrarily on a complement. Then $f_n^A$, $f_n^B$, $m_n^A$, and $m_n^B$ satisfy the desired properties.

For the proof of $(ii)$ we can choose $f_1$ such that it maps $1\in H^*(A)$ to $1\in A$. The map $f_2$ can be chosen such that $f_2(1\otimes a)=f_2(a\otimes 1)=0$ for all $a\in H^*(A)$. From there one proves inductively that $U_n$ vanishes on $a_1\otimes\ldots \otimes a_n$, whenever $n\geq 3$ and $a_j=1$ for some $1\leq j\leq n$, and thus also $f_n$ can be chosen equal to $0$ on this kind of tensor. If $A$ is formal and has a minimal Sullivan model $(\Lambda V,d)$ as in Remark \ref{rem:minmod}, then we can choose $f_n$ to have image in the ideal generated by $W_2$ whenever we want $f_n$ to be nontrivial.

Finally, $(iii)$ results from the fact that the choices made in $(i)$ and $(ii)$ can be performed in a compatible way: as in $(ii)$, we construct the unital minimal model of $A$ such that the higher operations vanish. Then the condition $f_n^B|_{\im(\varphi^*)^{\otimes n}}=\varphi\circ f_n^A\circ((\varphi^*)^{\otimes n})^{-1}$ already implies that $f_n^B$ vanishes on pure tensors in $\im (\varphi^*)^{\otimes n}$ which have the unit as a factor. Thus we can extend this partial definition of $f_n^B$ in a way that produces a unital minimal model of $B$ as in $(ii)$.
\end{proof}

\begin{rem}
We have only discussed minimal models of $A_\infty$- and $C_\infty$-algebras but analogous concepts do exists for $A_\infty$-modules.
\end{rem}

\def\cprime{$'$}

\
\vfill

%
%

\vfill

\noindent
\parbox[t]{0.5\textwidth}{\raggedright
\tiny \noindent \textsc
{Manuel Amann} \\
\textsc{Institut f\"ur Mathematik}\\
\textsc{Universit\"at Augsburg}\\
\textsc{Universit\"atsstra{\ss}e 14}\\
\textsc{86159 Augsburg}\\
\textsc{Germany}\\
[1ex]
\textsf{manuel.amann@math.uni-augsburg.de}\\
\textsf{www.math.uni-augsburg.de/prof/diff/arbeitsgruppe/amann/}
}
\hfill\parbox[t]{0.4\textwidth}{
\tiny \noindent \textsc
{Leopold Zoller} \\
\textsc{Mathematisches Institut}\\
\textsc{Universit\"at M\"unchen}\\
\textsc{Theresienstr.\ 39}\\
\textsc{80333 M\"unchen}\\
\textsc{Germany}\\
[1ex]
\textsf{L.Zoller@lmu.de}\\
}

\end{document}